\documentclass{degruyter-proceedings}          

\usepackage{graphicx}        					
\usepackage{multicol}        					
\usepackage[bottom]{footmisc}					
\usepackage{dsfont} 						
\usepackage{amsmath}
\usepackage{amsfonts} 							
\usepackage{newtxmath}						
\usepackage{amssymb} 						
\usepackage{color}
\definecolor{gray(x11gray)}{rgb}{0.75, 0.75, 0.75}
\usepackage{booktabs} 						
\usepackage{multirow}						
\usepackage{subfig}						
\usepackage{capt-of}

\usepackage[normalem]{ulem} 					

\usepackage{colortbl}						
\definecolor{gainsboro}{rgb}{0.86, 0.86, 0.86}
\definecolor{lightgray}{rgb}{0.83, 0.83, 0.83}
\definecolor{silver}{rgb}{0.75, 0.75, 0.75}
\definecolor{ashgrey}{rgb}{0.7, 0.75, 0.71}
\definecolor{battleshipgrey}{rgb}{0.52, 0.52, 0.51}

\usepackage{tikz}							
\usetikzlibrary{arrows,chains,matrix,positioning,scopes}

\def \LMRspacetimeaxis#1 {
\begin{tikzpicture}
  	\node[anchor=south west,inner sep=0] (img) at (0,0) {#1};
	\node[below=of img, node distance=0cm, yshift=1.05cm, xshift=0cm, font=\color{gray}] {$x$};
  	\node[left=of img, node distance=0cm, rotate=90, anchor=center, yshift=-0.8cm,font=\color{gray}] {$t$};
\end{tikzpicture}
}
\makeindex             
\newcommand{\sveta}[1]{{{#1}}}
\definecolor{green}{rgb}{0.09, 0.45, 0.27}

\title{Adaptive Space-Time Isogeometric Analysis for Parabolic Evolution Problems}
\headlinetitle{Adaptive Space-Time IgA of Parabolic Evolution Problems}

\lastnameone{Langer}
\firstnameone{Ulrich}
\nameshortone{U. Langer}
\addressone{RICAM Linz, Johann Radon Institute, Linz}
\countryone{Austria}
\emailone{ulrich.langer@ricam.oeaw.ac.at}

\lastnametwo{Matculevich}
\firstnametwo{Svetlana}
\nameshorttwo{S. Matculevich}
\addresstwo{RICAM Linz, Johann Radon Institute, Linz}
\countrytwo{Austria}
\emailtwo{smatculevich@ricam.oeaw.ac.at}

\lastnamethree{Repin}
\firstnamethree{Sergey}
\nameshortthree{S. Repin}
\addressthree{University of Jyvaskyla, Jyvaskyla; Peter the Great St.Petersburg Polytechnic University, Polytechnicheskaya, 29, St.Petersburg}
\countrythree{Finland; Russia}
\emailthree{serepin@jyu.fi}



\researchsupported{The research is supported by the Austrian Science Fund (FWF) through the NFN S117-03 project.}
\acknowledgments{
We would like to thank A. Mantzaflaris for his permanent support of the open-source C++ library 
G+Smo \cite{LMRgismoweb2015} that was used to 
implement  the adaptive space-time IgA schemes and 
to perform all the numeral tests presented in this work.}
\firstpage{1}

\abstract{
The paper is concerned with locally stabilized space-time IgA approximations to initial boundary value 
problems of the parabolic type. Originally, similar schemes (but weighted with a global 
mesh parameter) were presented and studied by U.~Langer, 
M.~Neu\-m\"uller, and S.~Moore (2016). 
The current work devises a localised version of this scheme. 
The localization of the stabilizations enables local mesh refinement
that is one of the main ingredients of adaptive algorithms.
We establish coercivity, 
boundedness, and consistency of the corresponding bilinear form. Using these fundamental properties 
together with the corresponding approximation error estimates for B-splines, 
we show that the space-time IgA solutions generated by the new scheme satisfy 
asymptotically optimal a priori discretization error estimates. 
{The adaptive mesh refinement algorithm proposed in the paper}
is based on a posteriori error estimates
of the functional type   that has been 
rigorously studied in earlier works by S.~Repin (2002) and U.~Langer, S.~Matculevich, and S.~Repin (2017).  
Numerical results presented in 
the paper confirm the improved convergence of global approximation errors. 
{Moreover, these results also}
confirm local efficiency of the error indicators produced by the error majorants.}

\begin{document}


\section{Introduction}
Time-dependent problems governed by parabolic partial differential equations (PDEs) 
are typical models 
in many scientific and engineering applications, e.g.,
heat conduction and diffusion, 
changing in time processes in social and life sciences, etc.
This fact triggers their active investigation in modelling,
mathematical analysis, and numerical solution. 
This paper is focused on the numerical treatment of parabolic problems  by means of
{\em Isogeometric Analysis (IgA)} \cite{LMR:HughesCottrellBazilevs:2005a}
combined with a full  {\em space-time approach} 
that treats time as yet another variable; 
see \cite{LMRGander2015}  and \cite{LMR:SteinbachYang:2018a}
for time-parallel and space-time methods.
Due to the fast development of parallel computers, this approach
to quantitative analysis of evolutionary problems has became quite natural. Moreover, this way
of treating evolutionary systems is not affected by the curse of sequentiality 
typical for  time-marching schemes. 
Various versions of the space--time method can be efficiently used in combination with parallelisation methods; see, e.g., 
 \cite{LMRGander2015, LMRGanderNeumueller2016a, LMRLangerMooreNeumueller2016a,LMRHoferLangerNeumuellerToulopoulos2017}.

{This paper uses} the idea similar to that applied
in \cite{LMRLangerMooreNeumueller2016a} for the derivation of the globally stabilized space-time 
scheme. It is based on testing the corresponding integral identity with the help of
`time-upwind' test functions, which are motivated by the space-time 
streamline diffusion method studied in \cite{LMRHansbo1994, LMRJohnson1987, LMRJohnsonSaranen1986}.
In contrast to the scheme presented in \cite{LMRLangerMooreNeumueller2016a}, this work 
is focused on element-wise analysis  that leads to a {locally stabilized} space-time IgA scheme.
 
{One of the attractive features of the  IgA method is high accuracy and flexibility of approximations obtained due to 
the high smoothness 
of the respective basis functions.} This fact allows a user to combine space-time schemes with IgA 
technologies, and construct fully-adaptive 
schemes aiming 
to tackle problems \sveta{generated by} industrial applications; several earlier studied examples
can be found in \cite{LMRTakizawaTezduyar2011,LMRTakizawaTezduyar2014}. 

Construction of effective adaptive refinement techniques is highly important for the design of fast 
and efficient {numerical methods for solving} PDEs. Adaptivity 
relies strongly on the reliable and locally quantitatively efficient a posteriori error estimation. 
We refer to  \cite{LMRAinsworthOden2000, LMRBangerthRannacher03,LMRRepin2008, LMRMalietall2014} 
for the overview of different error estimators. An efficient error indicator supposes to identify the areas, 
where discretization errors are excessively high,
in order to refine the mesh and minimise local errors. A smart combination of 
solvers and error indicators could potentially provide a fully automated refinement algorithm taking into 
account special features of the problem, and generating a discretisation 
that produces approximate solutions with the desired accuracy. 

Due to a tensor-product setting of IgA splines, mesh refinement has global effects, including a large 
percentage of superfluous control points. Challenges, arising along with these disadvantages, 
have triggered the development of local refinement techniques for IgA, such as 
{\em truncated B-splines} (T-splines) (introduced in \cite{LMRSederbergetall2003, LMRSederbergetall2004} 
and analysed in \cite{LMRBazilevsetall2010, LMRBeiraodaVeigaetall2011, LMRScottetall2011, LMRScottetall2012}), 
{\em hierarchical} (HB-splines) \cite{LMRForseyBartels1988, LMRKraft1997} and 
{\em truncated hierarchical B-splines} (THB-splines) \cite{LMRVuongetall2011, LMRGiannelliJuttlerSpeleers2012}, 
{\em patchwork splines} (PB-splines) \cite{LMREngleitnerJuttler2017}, 
{\em locally refined splines} (LR-splines) \cite{LMRDokkenLychePettersen2013, LMRBressan2013}, 
{\em polynomial splines over hierarchical T-meshes} (PHB-splines) \cite{LMRNguyenThanhetall2011, 
LMRWangetall2011}, etc.
%
%
{In the case of elliptic boundary value problems, local}
refinement IgA techniques were combined with some a posteriori error estimation approaches 
in several publications (a posteriori error estimates using the hierarchical bases 
in \cite{LMRDorfelJuttlerSimeon2010, LMRVuongetall2011}, residual-based a posteriori error estimators
and their modifications in \cite{LMRJohannessen2009, LMRWangetall2011, LMRBuffaGiannelli2015, 
LMRKumarKvamsdalJohannessen2015}, and goal-oriented error estimators in 
\cite{LMRZeeVerhoosel2011, LMRDedeSantos2012, LMRKuru2013, LMRKuruetall2014}).

In this paper, we deduce 
fully guaranteed error 
estimates in terms of several global norms equivalent to the norm of the functional space
containing the corresponding generalised solution. These estimates 
do not use mesh-dependent constants (which must be recalculated in the process of mesh adaptation),
and include only global constants characterising the geometry.
Henceforth, we shortly call them error majorants.
A posteriori error estimates of this type  were originally introduced in 
\cite{LMRRepin1997, LMRRepin1999} and later applied to various problems; see 
\cite{LMRRepin2008, LMRMalietall2014} and reference therein. { 
These estimates are valid for any approximation from the admissible functional space. 
They do not use special properties of approximations (e.g., Galerkin orthogonality) or/and additional 
requirements for the exact solution (e.g., extra regularity beyond the minimal 
energy class that guarantees the existence of the unique generalised solution) 
and are valid for any approximation from the admissible functional 
space. Moreover, the majorant also generates
efficient  indicators  of local (element-wise) error distribution over the domain. }

{We present a new localised space-time IgA scheme, where the adaptivity is driven 
by the functional type a posteriori error estimates. By exploiting the universality and 
efficiency of these error estimates as well as taking an advantage of smoothness of the IgA 
approximations, we aim at constructing fast fully adaptive 
space-time methods that 
could tackle complicated problems inspired by industrial applications. 
These two techniques were already combined in application to elliptic problems in \cite{LMRKleissTomar2015} 
and \cite{LMRMatculevich2017} using tensor-based splines and THB-splines \cite{LMRGiannelliJuttlerSpeleers2012, 
LMRGiannelliJuttlerSpeleers2014, LMRGiannellietall2016}, respectively. Both papers
confirmed that the majorants provide not only reliable and efficient upper bounds of the total energy 
error but a quantitatively sharp indicator of local element-wise errors. 

For the time-dependent problems, the simplest form of such error bounds was derived for the heat equation 
in \cite{LMRRepin2002} and tested for the generalised diffusion equation in \cite{LMRGaevskayaRepin2005}. 
Majorants for approximations to the evolutionary convection-diffusion problem 
having jumps in time were considered in \cite{LMRRepinTomar2010}. In \cite{LMRMatculevichRepin2014}, authors 
study the majorant's robustness to a drastic change in values of the reaction parameter in 
evolutionary reaction-diffusion problems and provide the comparison of upper bound to newly 
introduced minorant of the error. Another extensive discussion on the numerical properties of the 
above-mentioned error estimates w.r.t. both time-marching and space-time methods can be found in \cite{LMRHolmMatculevich2017}.

Paper \cite{LMRLangerMatculevichRepinArxiv2017}, that proceeds the current study, 
presents new functional-type a posteriori error estimates in a context of globally weighed space-time 
IgA schemes introduced in \cite{LMRLangerMooreNeumueller2016a}. It 
illustrates the reliability and efficiency of functional a posterior error estimates for IgA solutions w.r.t 
several examples exhibiting different features and reports on the computing cost for these bounds. 
Moreover, the numerical examples discussed in \cite{LMRLangerMatculevichRepinArxiv2017} 
demonstrate the efficiency of the space-time THB-spline-based adaptive procedure. Therefore, 
the importance of locally stabilized space-time IgA schemes
as well as the investigation of 
their numerical properties are rather inevitable in the context of the construction of fully 
adaptive schemes for initial-boundary value problems (I-BVPs).
}

This work is organized as follows: Section \ref{sec:variational-fomulation} defines  the  model
evolutionary  problem and recapitulates  notation and functional spaces used throughout the paper. 
Section \ref{eq:preliminaries} presents a concise overview of the IgA framework  and respective notions and
definitions. 
Furthermore,
it presents the globally stabilized space-time IgA scheme from \cite{LMRLangerMooreNeumueller2016a}
and discusses its main properties. Section \ref{sec:localized-scheme} 
introduces the new locally stabilized version of the space-time IgA scheme, and provides the 
proofs of coercivity, 
boundedness, and consistency of the bilinear form corresponding to the IgA scheme. 
We also establish a priori error estimates for the considered class of approximations.
The last section is dedicated to a posteriori estimates and practical aspects of the efficient combination 
of {locally stabilized} scheme and functional error majorants as well as their application 
to a series of numerical examples possessing different features.


\section{Space-time variational formulation}
\label{sec:variational-fomulation}

{Let 
$\Omega \subset \mathds{R}^d$, $d \in \{1, 2, 3\}$ be a bounded domain with Lipschitz
continuous boundary 
$\partial \Omega$ and  $(0, T)$, 
$0 < T < +\infty$ be  a given time interval.
By
$Q := \Omega \times (0, T)$ and $\overline{Q} := Q \cup \partial Q$ we denote the space-time cylinder 
and its closure, respectively. The lateral surface 
of $Q$ is defined as $\partial Q := \Sigma \cup \overline{\Sigma}_{0} \cup \overline{\Sigma}_{T}$, where 
$\Sigma = \partial \Omega \times (0, T)$,  $\Sigma_{0} =  \Omega \times \{0\}$ and 
$\Sigma_{T} = \Omega \times \{T\}$. 
%

We discuss an approach to  adaptive space-time 
IgA approximations of evolutionary problems using the classical model of the {\em linear parabolic initial-boundary value 
problem}: find $u: \overline{Q} \rightarrow \mathds{R}$ satisfying the equations
\begin{alignat}{2}
\partial_t u - \Delta_x u= f \quad {\rm in} \quad Q, \qquad 
u = 0 \quad {\rm on} \quad \Sigma, \qquad 
u = u_0 \quad {\rm on} \quad { \overline{\Sigma}_0},
\label{eq:equation}
\end{alignat}
where $\partial_t$ denotes the time derivative, 
$\Delta_x$ is the spatial Laplace operator, 
$f$ is a source function, and 
$u_0(x)$ is a given initial state. 

{Let us now introduce the functions spaces that we need in the following.
The norm and scalar product in the Lebesgue space $L_{2}(Q)$ of square-integrable functions  
in the space-time cylinder $Q$
are denoted by 
$\| \, v \, \|_{Q} := \| \, v \, \|_{L_{2}(Q)}$ and
$(v, w)_Q := \int_Q v(x,t) w(x,t) {{\rm d}x {\rm d}t}$, $\forall v, w \in L_{2} (Q),$ 
respectively, with the corresponding changes for spaces of vector-valued fields.
By $H^{s}(Q)$, $s \geq 1$, we denote standard Sobolev spaces 
supplied with the norm 
$\| v \|_{H^s(Q)} := \Big(\int_Q \sum_{|\alpha| \leq s} \partial^{\alpha} v dxdt \Big)^{1/2}$
for $s \in \mathds{N} \cup 0$, where $\alpha := \{\alpha_1, \ldots, \alpha_d\}$ is a multi-index, and 
$\partial^\alpha v := {\partial^{|\alpha|} v}/{\partial^{\alpha_1}_1\, ..., \partial^{\alpha_d}_d}$. 
Then, $| v |_{H^s(Q)} := \Big(\int_Q \sum_{|\alpha| = s} \partial^{\alpha} v dxdt \Big)^{1/2}$
denotes the $H^s$-seminorm. 
Next, we introduce the following spaces
\begin{equation*}
\begin{alignedat}{3}
V^{1, 0}_{0} := H^{1, 0}_{0}(Q) & := \Big\{\, u \in L_{2}(Q) && : 
\! \nabla_x u \in [L_{2}(Q)]^d, u = 0 \; \mbox{on} \; {\Sigma} \, \Big\}, \\
%
%
V^1_{0, \overline{0}} := H^{1}_{0, \overline{0}}(Q) & := \big\{\, u \in V^{1, 0}_{0} && : 
\partial_t u \in L_{2}(Q), \, u = 0 \; \mbox{on} \; {\Sigma_T} \, \big\}, \\
V^1_{0,\underline{0}} := H^{1}_{0, \underline{0}} (Q) & := \big\{ \,u \in  V^{1, 0}_{0} && : 
\partial_t u \in L_{2}(Q), \, u = 0 \; \mbox{on} \; {\Sigma_0} \, \big\}, 
\end{alignedat}
\label{eq:spaces-1}
\end{equation*}
and
\begin{equation*}
\begin{alignedat}{3}
V^{\Delta_x, 1}_{0} := H^{\Delta_x, 1}_{0}(Q) & := \big\{u \in  V^{1, 0}_{0} && :
\, \Delta_x u \in L_{2}(Q), \, \partial_t u \in L_{2}(Q) \big\},
\end{alignedat}
\end{equation*}
where the latter is equipped with the norm 
$\| w \|^2_{V^{\Delta_x, 1}_{0}} := \| \Delta_x w \|^2_Q + \| \partial_t w \|^2_Q$.
Finally, 
$$V^s_{0, \underline{0}} := H^{s}(Q) \cap V^{1}_{0, \underline{0}}$$ 
and
$$
	H^{{\rm div}_x, 0}(Q) := 
	\Big \{ \, \boldsymbol{y} \in [L_{2}(Q)]^d \;:\; 
	         {\rm div}_x \boldsymbol{y} \in L_{2} (Q) \,
	\Big \}
	$$
%
equipped with a scalar product 
$(\boldsymbol{v}, \boldsymbol{w})_{{\rm div}_x, 0} := 
(\boldsymbol{v}, \boldsymbol{w})_Q + ({\rm div}_x \boldsymbol{v}, {\rm div}_x \boldsymbol{w})_Q$.
Since $\Omega$ is bounded, we have the Friedrichs inequality 
$\| w \|_\Omega \leq {C_{{\rm F}}} \, \| \nabla_x w \|_\Omega$ for all $w \in H^{1}_0(\Omega)$, 
which also implies 
$\| w \|_Q \leq {C_{{\rm F}}} \, \| \nabla_x w \|_Q$ for all $w \in V^{1, 0}_0(Q)$.

It is proven
in \cite{LMRLadyzhenskaya1985} that the standard space-time variational 
formulation of the initial-boundary value problem (\ref{eq:equation}), 
find $u \in V^{1, 0}_{0}$
\begin{equation}
a(u, w) = \ell(w), \quad \forall w \in V^1_{0, \overline{0}} ,
\label{eq:variational-formulation}
\end{equation}
with the bilinear form
\begin{equation*}
	a(u, w) := 
	(\nabla_x {u}, \nabla_x {w})_Q - (u, \partial_t w)_Q,
\end{equation*}
and the linear form
\begin{equation*}
\ell(w) := (f, {w})_Q + (u_0, {w})_{\Sigma_0},
\end{equation*}
has a unique solution provided that 
$f \in L_{2,1}(Q) := \Big\{ v \in  L_{1}(Q) \, |  \, \int_0^T \| v(t) \|_{\Omega} \, dt < \infty \Big\}$
and $u_0 \in L_{2}(\Omega)$.
Here and later on,  
$(u_0, w)_{\Sigma_0} := \int_{\Sigma_0} u_0(x) \, {w}(x,0) dx = \int_{\Omega} u_0(x) \,{w}(x,0) dx$.
Moreover, if $f \in L_{2}(Q)$ and $u_0 \in H^{1}_{0}(\Omega)$, then 
problem \eqref{eq:variational-formulation} is uniquely solvable in $V^{\Delta_x, 1}_0$, 
and the solution $u$ continuously depends on $t$ in the norm of the space $H^{1}_{0}(\Omega)$ 
(see, e.g., \cite{LMRLadyzhenskaya1954} and \cite[Theorem 2.1]{LMRLadyzhenskaya1985}). 
Furthermore, according to \cite[Remark 2.2]{LMRLadyzhenskaya1985}, $\| \, u_{x} (\cdot, t) \, \|^2_{\Omega}$ 
is an absolutely continuous function of $t \in [0, T]$ for any $u \in V^{\Delta_x, 1}_0$. 

Throughout the paper, we assume that $f \in L_{2}(Q)$ 
and $u_0 \in H^{1}_{0}(\Sigma_0)$, 
i.e., we know that the solution $u$ of the space-time variational problem 
\eqref{eq:variational-formulation} belongs to $V^{\Delta_x, 1}_0$.
In this case, without a loss of generality, we can assume homogeneous initial conditions $u_0 = 0$;
cf. also \cite{LMRHoferLangerNeumuellerToulopoulos2017}.
}

\section{IgA framework}
\label{eq:preliminaries}

For the convenience of the reader, we recall the general concept of the IgA {technology}, the definition 
of B-splines, NURBS, and THB-splines, and their use in the geometrical representation of the space-time 
cylinder $Q$, as well as the construction of the IgA trial {and discretization} spaces, 
which are used to approximate solutions 
satisfying the variational formulation of \eqref{eq:variational-formulation}.

Let $p \geq 2$ be the polynomial degree, and let $n$ denote the number of basis functions used to construct 
a $B$-spline curve. The knot-vector in one dimension is a non-decreasing set of coordinates in the 
parameter domain, written as $\Xi = \{ \xi_1, \ldots, \xi_{n+p+1}\}$, $\xi_i \in \mathds{R}$, 
where $\xi_1 = 0$ and $\xi_{n+p+1} = 1$. The knots can be repeated, and the multiplicity of the $i$-th 
knot is indicated by $m_i$. Throughout the paper, we consider only open knot vectors, i.e., {the multiplicity 
$m_1$ and $m_{n+p+1}$ of the first and the last knots, respectively, is equal to $p+1$.}
In the case of the one-dimensional parametric domain $\hat{Q} = (0, 1)$, there is an underlying mesh 
of elements $\hat{K} \in \mathcal{\hat{K}}_h$ such that each of them is 
constructed by the distinct neighbouring knots. The global size of $ \mathcal{\hat{K}}_h$ is denoted by 
$\hat{h} := \max\limits_{\hat{K} \in \mathcal{\hat{K}}_h} \{ {\hat{h}}_{\hat{K}}\}, 
\, \mbox{where} \,  
{\hat{h}}_{\hat{K}} := {\rm diam} (\hat{K}).$
%
%
For the time being, we assume locally quasi-uniform meshes, 
i.e., the ratio of two neighbouring elements $\hat{K}$ and $\hat{K}'$ satisfies the inequality 
$c_1 \leq {{\hat{h}}_{\hat{K}}}/{{\hat{h}}_{\hat{K}'}} \leq c_2$, where $c_1, c_2$ {are positive constants}.

The univariate B-spline basis functions ${\hat{B}}_{i, p}: \hat{Q} \rightarrow \mathds{R}$ are defined by 
means of Cox-de Boor formula 
%
${\hat{B}}_{i, p} (\xi) := \tfrac{\xi - \xi_i}{\xi_{i+p} - \xi_i} \, {\hat{B}}_{i, p-1} (\xi)
                         + \tfrac{\xi_{i+p+1} - \xi}{\xi_{i+p+1} - \xi_{i+1}} {\hat{B}}_{i+1, p-1} (\xi),$ %
%
with 
$\hat{B}_{i, 0} (\xi) \, := 
\big\{ 1 \; \mbox{if} \; \xi_i \leq \xi \leq \xi_{i+1},  \; \mbox{and} \; 0 \; \mbox{otherwise} \big\},$
%
where a division by zero is defined to be zero. One of the most crucial properties of these basis functions 
is their $(p-m_i)$-times continuous differentiability across the $i$-th knot with multiplicity $m_i$. Hence, 
if $m_i = 1$ for every inner knot, then B-splines of the degree $p$ are $C^{p-1}$ continuous. For the knots 
lying on the boundary of the parametric domain, the multiplicity is $p + 1$, which makes 
the B-spline discontinuous on the patch interfaces. 
We note that analysis provided in this paper is valid for domains represented by a single-patch. 
Extensions to the multi-patch case will be considered in the subsequent paper.

{We now consider} the multivariate B-splines on the space-time parameter domain \linebreak 
$\hat{Q} := (0, 1)^{d+1}$, $d = \{1, 2, 3\}$, as a tensor-product of the corresponding univariate B-splines. 
For that, we define the knot-vector dependent on the space-time direction 
$\Xi^\alpha := \{ \xi^\alpha_1, \ldots, \xi^\alpha_{n^\alpha+p^\alpha+1}\}$, 
$\xi^\alpha_i \in \mathds{R}$, where $\alpha = 1, \ldots, d+1$ is the index indicating the direction. 
Furthermore, we introduce 
${\mathcal{I}} = \big\{\, i = (i_1,  \ldots, i_{d+1}): i_\alpha = 1, \ldots, n_\alpha; 
\alpha = 1, \ldots, d+1 \big\},$ the set used to number basis number functions,
and multi-indices standing for the order of polynomials $p := (p_1, \ldots, p_{d+1})$. 
The tensor-product of the univariate B-spline basis functions generates a multivariate splines defined as
%
${\hat{B}}_{i, p} ({\boldsymbol \xi}) 
:= \prod\limits_{\alpha = 1}^{d+1} {\hat{B}}_{i_\alpha, p_\alpha} (\xi^\alpha), 
\; \mbox{where} \; {\boldsymbol \xi} = (\xi^1, \ldots, \xi^{d+1}) \in \hat{Q}.$
%
The univariate and multivariate NURBS basis functions are defined in the parametric domain by means 
of the corresponding B-spine $\big\{ {\hat{B}}_{i, p} \big\}_{i \in {\mathcal{I}}}$. 
For the given $p := (p_1, \ldots, p_{d+1})$ and for any $i \in {\mathcal{I}}$, NURBS are defined as follows:
%
${\hat{R}}_{i, p}: \hat{Q} \rightarrow \mathds{R}$, 
${\hat{R}}_{i, p} ({\boldsymbol \xi}) 
:= \tfrac{w_i \, {\hat{B}}_{i, p} ({\boldsymbol \xi})}{W({\boldsymbol \xi})},$
%
with a weighting function
%
$W: \hat{Q} \rightarrow \mathds{R}$,  
$W({\boldsymbol \xi}) := \sum\limits_{i \in {\mathcal{I}}} w_i \, {\hat{B}}_{i, p} ({\boldsymbol \xi}),$
%
where $w_i >0$ are real numbers and $\sum\limits_{i \in {\mathcal{I}}} w_i = 1$. 

In the association with the knot-vectors $\Xi^\alpha$, $\alpha = 1, \ldots, d+1$, we define a mesh 
$\mathcal{\hat{K}}_h$ partitioning $\hat{Q}$ into $d+1$-dimensional open knot spans 
(elements)
$$\mathcal{\hat{K}}_h 
= \mathcal{\hat{K}}_h(\Xi^1, \ldots, \Xi^{d+1}) 
:= \Big\{ \, \hat{Q} = \otimes^{d+1}_{\alpha = 1} (\xi^{\alpha}_{i_\alpha},  \xi^{\alpha}_{i_{\alpha+1}}) \;:\; 
\hat{Q} \neq \mbox{\O}, \;
p^\alpha+1 \leq i^\alpha \leq n^\alpha - 1
\, \Big\}.$$
A non-empty element $\hat{K} = \otimes^{d+1}_{\alpha = 1} 
(\xi^{\alpha}_{i_\alpha},  \xi^{\alpha}_{i_{\alpha+1}}) \in \mathcal{\hat{K}}_h$
is characterized by its diameter ${\hat{h}}_{\hat{K}}$. To $\hat{K}$, we associate 
$\underline{\hat{K}} \in \hat{Q}$ defined as
$$\underline{\hat{K}} = \otimes^{d+1}_{\alpha = 1} 
 (\xi^{\alpha}_{i_\alpha - p_\alpha},  \xi^{\alpha}_{i_{\alpha} + p_\alpha + 1}) \in \mathcal{\hat{K}}_h.$$
The set $\underline{\hat{K}}$ represents the support extension of $\hat{K}$ and is constructed 
by the union of the supports of basis functions intersecting with ${\hat{K}}$.

The physical space-time domain $Q \subset \mathds{R}^{d+1}$ is defined from the parametric domain 
$\hat{Q} = (0, 1)^{d+1}$ by the geometrical mapping:
\begin{equation*}
\Phi: \hat{Q} \rightarrow Q := \Phi(\hat{Q} ) \subset \mathds{R}^{d+1}, \quad 
\Phi({\boldsymbol \xi}) := \sum\limits_{i \in \mathcal{I}} {\hat{R}}_{i, p}({\boldsymbol \xi}) \, {\bf P}_i,
\end{equation*}
where 
$\{{\bf P}_i\}_{i \in \mathcal{I}} \in \mathds{R}^{d+1}$ are the control points. For the simplicity, 
we assume below the same polynomial degree for all directions, i.e., 
$p_{\alpha} = p$, for all $\alpha = 1,  ... , d+1$.%

For each $\hat{K} \in \mathcal{\hat{K}}_h$ and $\underline{\hat{K}}$, 
we obtain an element and a support extension on the physical domain 
%
%
\begin{equation}
K = \Phi(\hat{K}) := \Big\{ \, \Phi(\xi) \, : \, \xi \in \hat{K}) \, \Big\} 
\quad
\mbox{and}
\quad
\underline{K} := \Phi(\underline{\hat{K}}),
\label{eq:physical-element-support-ext}
\end{equation}
respectively.
The physical mesh $\mathcal{K}_h$ is defined on the space-time cylinder $Q$ as follows 
$$\mathcal{K}_h := \big\{K = \Phi(\hat{K}) : \hat{K} \in \mathcal{\hat{K}}_h \big\}.$$
The global mesh size is denoted by
\begin{equation}
h := \max\limits_{K \in \mathcal{K}_h} \{ \, h_{K}\,\}, \quad  
h_{K} := \| \nabla_x\Phi \|_{L_\infty (K)} {\hat{h}}_{\hat{K}}.
\label{eq:global-mesh-size}
\end{equation}
Moreover, we assume that the physical mesh is also quasi-uniform, i.e., 
\begin{equation}
h_{K} \leq h \leq C_u \, h_{K}.
\label{eq:hk-and-h-relation}
\end{equation}

The set of {facets} corresponding to the discretisation $\mathcal{K}_h$ is denoted by 
$\mathcal{E}_h$ and can be split into the inner {facets}
$$\mathcal{E}^{I}_h = \{E \in \mathcal{E}_h^{K}: \exists \; 
K, K' \in  \mathcal{K}_h: E = \partial K \cap \partial K' \; \wedge \; E  \not\subset \partial Q\},$$
and
the {facets} intersecting with the boundary, namely,
$$\mathcal{E}^{\partial Q}_h = \{E \in \mathcal{E}_h^{K}:  \exists  \;K, K' \in  \mathcal{K}_h: 
E = \partial K \cap \partial K'  \;\wedge\; E \cap \partial Q \neq \mbox{\O}\,\}.$$
The latter one, in particular, contains inside the sets
\begin{alignat*}{2}
\mathcal{E}^{\Sigma}_h & = \{E \in \mathcal{E}_h^{K}:  \exists  \; K, K' \in \mathcal{K}_h: 
E = \partial K \cap \partial K'  \;\wedge\; E \cap \Sigma \neq \mbox{\O} \} 
\quad 
\mbox{and} \quad \\
\quad 
\mathcal{E}^{\Sigma_T}_h & = \{E \in \mathcal{E}_h^{K}:  \exists  \; K, K' \in \mathcal{K}_h: 
E = \partial K \cap \partial K'  \;\wedge\; E \cap \Sigma_T \neq \mbox{\O}\}.
\end{alignat*}
Let $\mathcal{E}^K_h$ denote the set of {facets} of the local element $K \in \mathcal{K}_h$, 
i.e.,
$$\mathcal{E}^K_h := \{ E \in \mathcal{E}_h : E \cap \partial K \neq \mbox{\O}, K \in \mathcal{K}_h \}.$$
%
The discretisation spaces on $Q$ are constructed by a push-forward of the basis functions defined on the 
parametric domain
\begin{equation}	
V_h := {\rm span} \, 
\Big\{ 
\phi_{h, i} := \hat{S}^{p}_h \circ \Phi^{-1}, 
\;\; (\ell, \boldsymbol{i})  \in \boldsymbol{\mathcal{I}} \, 
\Big\}_{i \in \mathcal{I}},
\label{eq:vh-v0h}
\end{equation}
where ${\hat{\mathcal{S}}}^{p}_h$ is the space of splines (e.g., B-splines, NURBS, THB-splines) of a degree $p$,
and $\Phi$ is assumed to be invertible in $Q$, with smooth inverse on each 
element $K \in \mathcal{K}_h$ (see \cite{LMRBazilevsetal2006, LMRBeiraodaVeigaBuffaSangalliVazquez2014} 
and references therein). Moreover, we introduce the subspace \linebreak
$V_{0h} := V_h \cap V^1_{0, \underline{0}}$ 
for the functions satisfying homogeneous initial and boundary conditions.

Let us recall two fundamental inequalities, i.e., scaled trace and inverse inequalities, that are 
important for the derivation of a priori {discretization} error estimates for the space-time IgA scheme 
presented in the further sections.
\begin{lemma}{\cite[Theorem 3.2]{LMREvansHughes2013}}
\label{lm:lemma-2}
Let $K \in \mathcal{K}_h$. Then the {\bf \em scaled trace inequality} 
\begin{equation}
\| v \|_{\partial K} \leq C_{t\!r\!} \, h_{K}^{-1\!/_2} (\| v \|_{K} + h_{K} \, \| \nabla v \|_{K})
\label{eq:trace-inequality}
\end{equation}
holds for all $v \in H^{1}(K)$, where $\nabla = (\nabla_x, \partial_t )$, $h_K$ is a local mesh size (cf. \eqref{eq:global-mesh-size}), and 
$C_{t\!r\!}$ is a positive constant independent of $K \in \mathcal{K}_h$.
\end{lemma}

\begin{lemma}{\cite[Theorem 4.1]{LMRBazilevsetal2006} 
}
\label{lm:lemma-3}
Let $K \in \mathcal{K}_h$. Then the {\bf \em inverse inequalities}
\begin{alignat}{2}
\| \nabla_xv_h\|_{K} & \leq C_{i\!n\!t\!,\!1} \,h_{K}^{-1} \, \| v_h \|_{K}, \quad \mbox{and}
\label{eq:int-inequality-1} \\
%
%
\| v_h \|_{\partial K} &  \leq C_{i\!n\!t, 0} \,h_{K}^{-1\!/_2} \, \| v_h \|_{K}
\label{eq:int-inequality-0}
\end{alignat}
hold for all $v_h \in V_{h}$, where $C_{i\!n\!t, 0}$ and $C_{i\!n\!t\!,\!1}$ are positive  constants 
independent of $K \in \mathcal{K}_h$, and $h_K := {\rm diam}_{ K \in \mathcal{K}_h }$ is a local mesh 
size.


\end{lemma}

For the completeness, we recall fundamental results on the approximation properties of spaces generated 
by NURBS using \cite[Section 3]{LMRBazilevsetal2006}. 
It states the existences of a projection operator that provides an asymptotically optimal approximation result.
%
%
\begin{lemma}{\cite[Theorem 3.1]{LMRBazilevsetal2006}}
\label{lm:lemma-5}
Let $\ell, s \in \mathds{N}$ be $0 \leq \ell \leq s \leq p+1$, $u \in V^s_{0, \underline{0}}$, and 
$K$ and $\underline{K}$ are elements defined in \eqref{eq:physical-element-support-ext}.
Then there exists a projection operator 
$\Pi_h: V^s_{0, \underline{0}} \rightarrow V_{0h}$ and a constant $C_s >0$ such that
\begin{equation}
|v - \Pi_h v|^2_{H^\ell (K)}
\leq C^2_{l\!,s} h^{2(s-\ell)}_K \, 
       \sum_{i = 0}^{s} c_K^{2(i-\ell)} \, | v |^2_{H^{i}(\underline{K})}, \quad 
       \forall v \in L_2(Q) \cap H^{\ell}(\underline{K}),
\label{eq:c-s-interpolation-estimate}
\end{equation}
 %
where $C_{l\!,s}$ is a dimensionless shape constant dependent on $s, \ell,$ and $p$, 
the shape regularity of $K$, described by $\Phi$ and its gradient, 
$h_K$ is a local mesh size (cf. \eqref{eq:global-mesh-size}), and 
$c_K := \| {\nabla_x\Phi} \|_{ L_{\infty}({\Phi}^{-1}(\underline{\hat{K}}))}$.
\end{lemma}

Unlike the classical finite element spaces of degree $p$, Lemma \ref{lm:lemma-5} provides the bound, 
where the $\ell^{th}$-order seminorm of the error $u - \Pi_h u$ is controlled by the full $s^{th}$-order 
norm of $u$. In particular, the following formulations of \eqref{eq:c-s-interpolation-estimate} will be 
used:
\begin{alignat}{2}
\| u - \Pi_h u \|^2_{L_2(K)}
& \leq C^2_{0,s} h^{2s}_K \, 
       \sum_{i = 0}^{s} c_K^{2i} \, | u |^2_{H^{i}(\underline{K})},
        \label{eq:c-s-interpolation-estimate-v1}\\
|u - \Pi_h u|^2_{H^1 (K)}
& \leq C^2_{1,s} h^{2(s-1)}_K \, c_K^{-2}
       \sum_{i = 0}^{s} c_K^{2i} \, | u |^2_{H^{i}(\underline{K})}, 
       \label{eq:c-s-interpolation-estimate-v2}\\
|u - \Pi_h u|^2_{H^2 (K)}
& \leq C^2_{2,s} h^{2(s-2)}_K \, c_K^{-4}
       \sum_{i = 0}^{s} c_K^{2i} \, | u |^2_{H^{i}(\underline{K})}, 
       \label{eq:c-s-interpolation-estimate-v3} 
\end{alignat}
for any $v \in L_2(Q) \cap H^{\ell}(\underline{K})$.

Globally stabilized space-time IgA scheme for parabolic equations have been presented and analysed in 
\cite{LMRLangerMooreNeumueller2016a}, where the authors proved its efficiency for fixed and moving 
spatial computational domains. In particular, it was shown that the corresponding discrete bilinear form is 
elliptic w.r.t. a discrete energy norm, bounded, consistent, and that generated IgA approximations
satisfy a priori discretisation error estimate. In order to derive a {\em globally stabilized discrete 
IgA space-time scheme}, the authors considered time-upwind test function 
%
$v_h + \delta_{h} \, \partial_t v_h$, $\delta_{h} = \theta h$, $v_h \in V_{0h},$
%
such that $\theta >0$ is an auxiliary constant and $h$ is the global mesh-size (cf. \eqref{eq:global-mesh-size}).
This implies the discrete stabilized space-time {IgA scheme}: find $u_h \in V_{0h}$ satisfying 
\begin{alignat}{2}
a_{h} (u_h, v_h) = \, & l_h(v_h), \quad \forall v_h \in V_{0h},
\label{eq:discrete-scheme} 
\end{alignat}
where 
\begin{alignat*}{2}
a_{h} (u_h, v_h) 
:= \, & (\partial_t u_h, v_h)_Q \!
+ (\nabla_x {u_h}, \nabla_x {v_h})_Q \!
+ \delta_{h} \Big( (\partial_t u_h, \partial_t v_h)_Q \!
			+ (\nabla_x {u_h}, \partial_t (\nabla_x v_h))_Q \Big),\\
l_h(v_h) := \, & (f, v_h + \delta_{h} \, \partial_w v_h)_Q.
\end{alignat*}
%
Combining coercivity and boundedness properties of $a_{h}(\cdot, \cdot)$ with the 
consistency of the scheme and approximation results for IgA spaces, we obtain {the corresponding 
 a priori error estimate w.r.t. the norm}
\begin{equation*}
| \! |\! | v_h | \! |\! |^2_{h} 
:= \| \nabla_x {v_h }\|^2_{Q} 
+ \delta_{h} \, \| \partial_t v_h  \|^2_{Q} + 
\| v_h \|^2_{\Sigma_T} 
+ \delta_{h} \, \| \nabla_x v_h  \|^2_{\Sigma_T},
\end{equation*}
which is presented in Theorem \ref{th:theorem-8} below.

\begin{theorem}
\label{th:theorem-8}{\cite{LMRLangerMooreNeumueller2016a, LMRLangerMatculevichRepinArxiv2016}}
Let $u \in V^{s}_{0} := V^{s}(Q) \cap V^{1, 0}_{0}$, $s \in \mathds{N}$, $s \geq 2$, 
be the exact solution to \eqref{eq:variational-formulation}, and let $u_h \in V_{0h}$ be the solution to 
\eqref{eq:discrete-scheme} with some fixed parameter $\theta >0$. Then, the following a priori 
{discretization} error estimate
%
\begin{equation*}
\| u - u_h\|_{h} \leq C \, h^{r-1}\, \| u \|_{H^r(Q)} 
\end{equation*}
holds with $r = \min \{ s, p+1 \}$ and {some} generic constant $C >0$ independent of $h$.
\end{theorem}

\section{Locally stabilized IgA schemes}
\label{sec:localized-scheme}

In the current section, we assume that $p \geq 2$ and $m \leq p - 1$, which yields that 
$V_{0h} \subset C^1(\overline{Q})$ providing the inclusion 
$V_{0h} \subset V^{\Delta_x, 1}_{0, \underline{0}} 
:= V^{\Delta_x, 1}_{0} \cap V^{0, 1}_{0, \underline{0}}$.
{We know that the solution $u$ of \eqref{eq:variational-formulation} belongs to 
$V^{\Delta_x, 1}_{0, \underline{0}}$ provided that $f \in L^2(Q)$. In this case, for all 
$K \in \mathcal{K}_h$, we can write the PDE $\partial_t u - \Delta_x u= f \; {\rm in} \; K$ 
and can multiply it with the localized test functions}
%
$$v_h + \delta_{K} \, \partial_t v_h, \quad 
\delta_{K} = \theta_{K} \, h_{K},  \quad 
 \theta_{K} > 0, \quad h_{K} := {\rm diam} (K),
$$
such that
$$\big(\partial_t u - \Delta_x u, v_h + \delta_{K} \, \partial_t v_h \big)_{K} 
= (f, v_h + \delta_{K} \, \partial_t v_h)_{K}, \quad \forall v_h \in V_{0h}.$$
By summing up all the elements in $\mathcal{K}_h$, we obtain the relation 
%
\begin{multline*}
(\partial_t u - \Delta_x u, v_h)_Q 
+ \!\!\! \sum\limits_{{K} \in \mathcal{K}_h} \! \! \!\delta_{K} \,  
   \big(\partial_t u - \Delta_x u, \partial_t v_h \big)_{K} 
= (f, v_h)_Q + \!\!\! \sum\limits_{{K} \in \mathcal{K}_h} \!\!\!  \delta_{K} \, (f, \partial_t v_h)_{K}.
\end{multline*}
The integration by parts w.r.t. to the space variable yields
\begin{alignat*}{2}
& \ell_{l\!o\!c\!,h}(v_h)
:= (f, v_h)_Q + \!\sum\limits_{{K} \in \mathcal{K}_h} \delta_{K} \, (f, \partial_t v_h)_{K}
=
(\partial_t u, v_h)_Q + (\nabla_x u, \nabla_x v_h)_Q \nonumber \\
&  
+ \!\sum\limits_{{K} \in \mathcal{K}_h}  \delta_{K}\,  
\Big( (\partial_t u, \partial_t v_h \big)_{K}
        + \big(\nabla_x u, \nabla_x \partial_t v_h \big)_{K} 
 - 
   \big< \boldsymbol{n}_{x}^{\partial K} \cdot \nabla_x u, \partial_t v_h \big>_{\partial K} \Big) =: a_{l\!o\!c\!,h}(u, v_h),
\end{alignat*}
%
where $\boldsymbol{n}_{x}^{\partial K}$ is an external normal vector to 
$\partial K$. Here, the last term is nothing else but a duality 
product $\big< \cdot, \cdot \big>_{\partial K} 
= \big< \cdot, \cdot \big>_{H^{-1/2}(\partial K) \times H^{1/2}(\partial K)}:
H^{-1/2}(\partial K) \times H^{1/2}(\partial K) \rightarrow \mathds{R}$, and 
$H^{-1/2}$ is dual space to $H^{1/2}$.
%
%
Thus, {we arrive at the finite dimensional problem}:
find $u_h \in V_{0h}$ satisfying the identity
\begin{equation}
{a_{l\!o\!c\!,h}(u_h, v_h) = \ell_{l\!o\!c\!,h}(v_h), \quad \forall u_h, v_h \in V_{0h},}
\label{eq:a-hk-l}
\end{equation}
where the bilinear form $a_{l\!o\!c\!,h}(u_h, v_h)$ can be written as follows
\begin{alignat*}{2}
a_{l\!o\!c\!,h}(u_h, v_h) : = 
(\partial_t u_h, v_h)_Q & + (\nabla_x u_h, \nabla_x v_h)_Q \nonumber\\
& + \sum\limits_{K \in \mathcal{K}_h}  \delta_{K}\,
\Big( (\partial_t u_h, \partial_t v_h )_{K} 
+ (\nabla_x u_h, \nabla_x \partial_t v_h )_{K} \Big) \nonumber\\
& -  
  \sum\limits_{K \in \mathcal{K}_h}  \delta_{K}\, 
   \sum\limits_{E \in \mathcal{E}^K_h \subset \mathcal{E}_h^I}
   \big(\boldsymbol{n}_{x}^{E} \cdot\nabla_x u_h, \partial_t v_h \big)_{E}.
\end{alignat*}
%
%
Due to the assumptions 
$v_h \big|_{\Sigma} = 0$ and $\boldsymbol{n}_{x}^{E} \big|_{\Sigma_0 \cup \Sigma_T} = {\bf 0}$, 
contributions of the terms \linebreak $\big(\boldsymbol{n}_{x}^{E} \cdot \nabla_x u_h , \partial_t v_h 
  \big)_{E  \in \mathcal{E}_{h}^{K} \subset \mathcal{E}^{\partial Q}_h }$ vanishes.

\subsection{Coercivity}
\label{ssec:coercivity-a-tilde-loc}

\begin{lemma}
\label{lm:a-tilde-hk-coercivity}
Let the parameters $\theta_{K}$ be sufficiently small, i.e., 
$\theta_{K} \in \Big(0, \tfrac{h_K}{d\, C^2_{i\!n\!t\!,\!1}} \Big]$, 
where $C_{i\!n\!t\!,\!1}$ is the interpolation constant in \eqref{eq:int-inequality-1} 
associated with $K \in \mathcal{K}_h$.
Then, the bilinear form $a_{l\!o\!c\!,h}(\cdot, \cdot): V_{0h} \times V_{0h} \rightarrow \mathds{R}$ is 
$V_{0h}$-coercive w.r.t. to the norm 
{\begin{equation}
|\!|\!| v_h |\!|\!|^2_{l\!o\!c\!,h} 
:= \| \nabla_x v_h\|^2_Q + \tfrac{1}{2} \| v_h \|^2_{\Sigma_T}
+ \sum\limits_{K \in \mathcal{K}_h} \delta_{K} \,  \| \partial_t v_h\|^2_{K},
\label{eq:norm-K-h}
\end{equation}
}
i.e., there exists a constant $\mu_{l\!o\!c, c} > 0$ such that 
\begin{equation}
a_{l\!o\!c\!,h}(v_h, v_h) \geq \mu_{l\!o\!c, c} \, |\!|\!| v_h |\!|\!|^2_{l\!o\!c\!,h}, 
\quad \forall v_h \in V_{0h}.
\label{eq:coercivity}
\end{equation}
\end{lemma}
\begin{proof}
Integration by parts of $a_{l\!o\!c\!,h}(v_h, v_h)$ yields
%
\begin{alignat}{2}
a_{l\!o\!c\!,h}& (v_h, v_h) := (\partial_t v_h, v_h)_Q + (\nabla_x v_h, \nabla_x v_h)_Q  \nonumber\\
& 
+ \sum\limits_{K \in \mathcal{K}_h} \delta_{K} 
\Big\{ (\partial_t v_h, \partial_t v_h \big)_{K} 
         + \,\big (\nabla_x v_h, \nabla_x \partial_t v_h)_{K}  
         -  \!\!
  \sum\limits_{E \in \mathcal{E}_h^I}
   (\boldsymbol{n}_{x}^{E} \cdot\nabla_x u_h, \partial_t v_h)_{E} \Big\}  
\nonumber\\
& = \tfrac{1}{2} \, \| v_h \|^2_{\Sigma_T} +  \| \nabla_x v_h \|^2_Q 
+ \sum\limits_{K \in \mathcal{K}_h} \delta_{K} \Big\{
 	\| \partial_t v_h \|^2_{K} - (\Delta_x v_h, \partial_t v_h )_{K} \Big\}. 
\label{eq:fiist-term-coercivity}
\end{alignat}
{Here, and later on, we assume that $E \in \mathcal{E}_h^{K}$ and therefore omit repeating it.}
In order to prove coercivity, we need to estimate the last term in 
\eqref{eq:fiist-term-coercivity}. By using \eqref{eq:int-inequality-1} and Young inequality, we arrive at
\begin{alignat*}{2}
 \sum\limits_{K \in \mathcal{K}_h} \delta_K \, & (\Delta_x v_h, \partial_t v_h )_{K} 
  \leq  \Big(\sum\limits_{K \in \mathcal{K}_h} \delta_K \, \| \Delta_x v_h \|^2_K \Big)^{{1}\!/{2}} \, 
          \Big(\sum\limits_{K \in \mathcal{K}_h} \delta_K \| \partial_t  v_h \|^2_K \Big)^{{1}\!/{2}}  \nonumber\\
& \leq  \Big( \sum\limits_{K \in \mathcal{K}_h} \delta_K \, d \, \sum\limits_{l = 1}^d \| \partial^2_{x_l} v_h \|^2_K\Big)^{{1}\!/{2}} \, 
          \Big( \sum\limits_{K \in \mathcal{K}_h} \delta_K \| \partial_t  v_h \|^2_K \Big)^{{1}\!/{2}}  \nonumber\\
& \leq  \Big( \sum\limits_{K \in \mathcal{K}_h} \theta_K \, h_K \, d \, 
        \sum\limits_{l = 1}^d C^2_{i\!n\!t\!,\!1}\, h^{-2}_K \, \| \partial_{x_l} v_h \|^2_K \Big)^{{1}\!/{2}}
       \Big(\sum\limits_{K \in \mathcal{K}_h} \delta_K \| \partial_t  v_h \|^2_K \Big)^{{1}\!/{2}} \nonumber\\
& \leq \Big( d \, \max\limits_{K \in \mathcal{K}_h} (\tfrac{\theta_K}{h_K} \, C^2_{i\!n\!t\!,\!1}) \, 
      \| \nabla_x v_h \|^2_Q \Big)^{{1}\!/{2}} 
      \Big(\sum\limits_{K \in \mathcal{K}_h} \delta_K \| \partial_t  v_h \|^2_K \Big)^{{1}\!/{2}} \nonumber\\
& \leq \tfrac{d}{2} \, 
          \max\limits_{K \in \mathcal{K}_h} \Big(\tfrac{\theta_K}{h_K} \, C^2_{i\!n\!t\!,\!1}\Big) \, 
      \Big(\| \nabla_x v_h \|^2_Q 
      + \sum\limits_{K \in \mathcal{K}_h} \delta_K \| \partial_t  v_h \|^2_K\Big).
\end{alignat*}
Therefore, $a_{l\!o\!c\!,h}(v_h, v_h)$ can be bounded from below as follows:
\begin{alignat*}{2}
a_{l\!o\!c\!,h}(v_h, v_h) & 
\geq \tfrac{1}{2} \, \| v_h \|^2_{\Sigma_T}
       + \Big(1 - \tfrac{d}{2} \max\limits_{K \in \mathcal{K}_h}\, \tfrac{\theta_K}{h_K} \, C^2_{i\!n\!t\!,\!1} \Big)\,
          \Big\{
          \sum_{K \in \mathcal{K}_h} \!\! \delta_{K} \, \| \partial_t v_h \|^2_{K} 
          +  \| \nabla_x v_h \|^2_Q \Big\} \\
& \geq \tfrac{1}{2} \, \| v_h \|^2_{\Sigma_T}
       + \Big(1 - \tfrac{d}{2} \max\limits_{K \in \mathcal{K}_h}\, \tfrac{\theta_K}{h_K} \, C^2_{i\!n\!t\!,\!1} \Big)\,
          \Big\{
          \sum_{K \in \mathcal{K}_h} \!\! \delta_{K} \, \| \partial_t v_h \|^2_{K} 
          +  \| \nabla_x v_h \|^2_Q \Big\} \\
& \geq \tfrac{1}{2} \, |\!|\!| v_h |\!|\!|^2_{l\!o\!c\!,h},        
\end{alignat*}
provided that $\theta_{K} \in \Big(0, \tfrac{h_K}{d\, C^2_{i\!n\!t\!,\!1}} \Big]$ for $K \in \mathcal{K}_h$.
\end{proof}

\begin{remark}
Computation of the constants $C_{i\!n\!t\!,\!1}$ in the inverse inequalities corresponds to the 
question of accurate estimation of maximal eigenvalues for generalised eigenvalue problems for 
considered differential equations.
In \cite{LMRKoutschanNeumuellerRaduArxiv2016}, the authors applied symbolic computation methods to
this problem defined on the square elements and were able to improve the previously known
upper bounds in \cite{LMRSchwab1998}.
\end{remark}

$V_{0h}$-coercivity of $a_{l\!o\!c\!,h}$ implies existence and uniqueness of the discrete solution 
$u_h \in V_{0h}$. 
{From Lemma \ref{lm:a-tilde-hk-coercivity}, it also immediately follows that the system matrix 
of the linear system generated by the 
bilinear form is positive definite.}

\subsection{Boundedness}
\label{ssec:boundedness}

To prove a priori error bounds, we need to show the uniform boundedness of the localised bilinear
form ${a}_{l\!o\!c\!,h}(\cdot, \cdot)$ on $V_{0h, *}  \times V_{0h}$, where 
{$V_{0h, *} := V^{1}_{0, \underline{0}} \cap V^{\Delta_x, 1}_{0, \underline{0}} + V_{0h}$ }
is equipped with the norm
\begin{equation*}
|\!|\!| v |\!|\!|^2_{l\!o\!c\!,h,*} 
:= |\!|\!| v |\!|\!|^2_{l\!o\!c\!,h} 
+ \sum\limits_{K \in \mathcal{K}_h } \big( \delta^{\, - 1}_{K} \| v \|^2_{K} + \delta_{K} \| \Delta_x v \|^2_{K} \big).
\end{equation*}

\begin{lemma} 
\label{lm:a-tilde-hk-boundedness}
Assume that $\theta_{K} \in \Big(0, \tfrac{h_K}{d\, C^2_{i\!n\!t\!,\!1}} \Big]$, $K \in \mathcal{K}_h$. 
Then, the bilinear form \linebreak 
${a}_{l\!o\!c\!,h}(\cdot, \cdot)$ is uniformly bounded on $V_{0h, *}  \times V_{0h}$,  
i.e., there exists a positive constant $\mu_{l\!o\!c, b}$ that does not depend on $h_K$ such that
\begin{equation}
|a_{l\!o\!c\!,h}(u, v_h)| \leq \mu_{l\!o\!c, b} \, \| u \|_{l\!o\!c\!,h,*}\, \| v_h\|_{l\!o\!c\!,h}, 
\quad \forall u \in V_{0h, *}, \quad \forall v_h \in V_{0h}.
\label{eq:boundedness}
\end{equation}
\end{lemma}

\begin{proof}
We estimate $a_{l\!o\!c\!,h}(u, v_h)$ term by term. For the first one, 
we apply integration by parts w.r.t. time and the Cauchy inequality:
\begin{alignat*}{2}
(\partial_t u, v_h)_Q & = (u, v_h)_{\Sigma_T} - (u, \partial_t  v_h)_Q  \nonumber\\[2pt]
& \leq \| u \|_{\Sigma_T}\, \| v_h \|_{\Sigma_T} 
      + \sum\limits_{K \in \mathcal{K}_h} 
         \delta_{K}^{\, - {1}\!/{2}} \, \| u \|_{K} \, 
         \delta_{K}^{{1}\!/{2}} \| \partial_t  v_h\|_{K} \nonumber\\
 &  \leq \Big[\| u \|^2_{\Sigma_T} 
             + \sum\limits_{K \in \mathcal{K}_h} \delta_{K}^{\, - 1} \, \|  u  \|^2_{K} \Big]^{{1}\!/{2}} \, 
      \Big[\| v_h \|^2_{\Sigma_T} 
            + \sum\limits_{K \in \mathcal{K}_h} \delta_{K} \, \|  \partial_t  v_h  \|^2_{K} \Big]^{{1}\!/{2}}.
\end{alignat*}
The second term is estimated by means of the H\"older inequality, i.e.,\\ 
$$(\nabla_x u, \nabla_x v_h \big)_{Q} \leq \| \nabla_x u\| \, \| \nabla_x v_h \|,$$
whereas the third one is treated as follows:
\begin{alignat*}{2}
\sum\limits_{K \in \mathcal{K}_h} \delta_{K} \, (\partial_t u, \partial_t v_h \big)_{K} 
& \leq  \Big[ \sum\limits_{K \in \mathcal{K}_h} \delta_{K} \| \partial_t u \|^2_{K} \Big]^{{1}\!/{2}} \, 
           \Big[ \sum\limits_{K \in \mathcal{K}_h} \delta_{K} \| \partial_t v_h \|^2_{K} \Big]^{{1}\!/{2}}.
\end{alignat*}
%
If we consider result of \eqref{eq:fiist-term-coercivity}, the last term can be estimated as
\begin{alignat*}{2}
- \sum\limits_{K \in \mathcal{K}_h} \delta_{K} \, \big(\Delta_x u, \partial_t v_h \big)_{K}
 & \leq \Big[\sum\limits_{K \in \mathcal{K}_h} \delta_{K} \| \Delta_x u \|^2_{K} \Big]^{{1}\!/{2}} \, 
 	 \Big[\sum\limits_{K \in \mathcal{K}_h} \, \delta_{K} \, \| \partial_t v_h \|^2_{K} \Big]^{{1}\!/{2}}. \end{alignat*}
By combining the obtained results, the bilinear form can be bounded as\\
\begin{alignat*}{2}
|a_{l\!o\!c\!,h}(u, v_h)| 
& \leq \Big[ \| u \|^2_{\Sigma_T} + \| \nabla_x u \|^2_Q 
                  + \sum\limits_{K \in \mathcal{K}_h} 
                  \Big \{ \delta_{K}^{- 1} \, \|  u  \|^2_{K}  
                            + \delta_{K} \,  (\| \partial_t u \|^2_{K} + \| \Delta_x u \|^2_{K}) \Big \} \Big]^{{1}\!/{2}} \\
& \qquad \times
   \Big[ \| v_h \|^2_{\Sigma_T} + \| \nabla_x v_h \|^2_Q
            + 3 \, \sum\limits_{K \in \mathcal{K}_h} \delta_{K} \, \|  \partial_t  v_h  \|^2_{K} \Big]^{{1}\!/{2}} \\
& \leq \mu_{l\!o\!c, b} \, | \! |\! | u | \! |\! |_{l\!o\!c\!,h,*} \, | \! |\! | u_h | \! |\! |_{l\!o\!c\!,h}, 
\end{alignat*}
where $\mu_{l\!o\!c, b} = 3$.
\end{proof} 

\subsection{Approximation properties}

The estimate \eqref{eq:c-s-interpolation-estimate} implies
a priori estimates of the interpolation error $u - \Pi_h u$, measured in terms of the $L_2$-norm and
the discrete norms $| \! |\! | \cdot  | \! |\! |_{l\!o\!c\!,h}$ and $| \! |\! | \cdot  | \! |\! |_{l\!o\!c\!,h,*}$, which 
we later need in order to obtain an a priori estimate for $u - u_h$.

\begin{lemma}
\label{lm:lemma-6-loc}
Let $l, s \in \mathds{N}$ be $1 \leq l \leq s \leq p+1$, and $u \in V^{s}_{0, \underline{0}}$. Then, 
there exists a projection operator  $\Pi_h: V^{s}_{0, \underline{0}} \rightarrow V_{0h}$ (see 
Lemma \ref{lm:lemma-5}) and positive constants $C_1, C_1$, and $C_2$, such that the following a priori 
error estimates hold
 \begin{alignat}{2}
| \! |\! | u - \Pi_h u  | \! |\! |^2_{l\!o\!c\!,h} & \leq C_1 \, 
\sum\limits_{K \in \mathcal{K}_h} h^{2(s - 1)}_K \, 
\sum_{i = 0}^{s} c_K^{2\,i} \, | u |^2_{H^{i}({\underline{K}})}, \, 
\label{eq:int-2-loc}\\
%
| \! |\! | u - \Pi_h u  | \! |\! |^2_{l\!o\!c\!,h,*} & \leq C_2 \, 
\sum\limits_{K \in \mathcal{K}_h}  h^{2(s - 1)}_K 
 \sum_{i = 0}^{s} c_K^{2\,i} \, | u |^2_{H^{i}({\underline{K}})}, 
\label{eq:int-3-loc}
\end{alignat}
{for all $u \in L_2(Q) \cap H^{s}(\underline{K})$.}
\end{lemma}

\begin{proof} 
%
To prove \eqref{eq:int-2-loc} and \eqref{eq:int-3-loc}, we need to provide estimates for each term 
in the norm $| \! |\! | u - \Pi_h u  | \! |\! |_{l\!o\!c\!,h}$. In order to bound the first term, we use 
\eqref{eq:c-s-interpolation-estimate-v2}, i.e.,
\begin{alignat}{2}
\| \nabla_x (u  - \Pi_h u)\|^2_Q 
& \leq  \sum\limits_{K \in \mathcal{K}_h} | u -  \Pi_h u |^2_{H^1(K)} \nonumber\\
& \leq  C^2_{1,s} \max\limits_{K \in \mathcal{K}_h}  c_K^{-2} \sum\limits_{K \in \mathcal{K}_h}  \, h^{2(s - 1)}_{K} \, 
\sum_{i = 0}^s c_K^{2i} \, | u |^2_{H^s({\underline{K}})}.
\label{eq:estimate-Q-nablax}
\end{alignat}
For the next one, we use \eqref{eq:hk-and-h-relation}, \eqref{eq:c-s-interpolation-estimate-v2}, 
and by similar approach derive:
\begin{alignat}{2}
\sum\limits_{K \in \mathcal{K}_h} \delta_{K} \,\| \partial_t (u  - \Pi_h u) & \|^2_{K} \; 
\leq \max\limits_{K \in \mathcal{K}_h} \delta_K \sum\limits_{K \in \mathcal{K}_h} | u -  \Pi_h u |^2_{H^1(K)} \nonumber \\
& \leq C^2_{1,s} \, \max\limits_{K \in \mathcal{K}_h} \Big\{ \delta_K c_K^{-2} \Big\} \sum\limits_{K \in \mathcal{K}_h}  \, h^{2(s - 1)}_{K} \, 
\sum_{i = 0}^s c_K^{2i} \, | u |^2_{H^s({\underline{K}})}.  
\label{eq:estimate-delta-dt-K}
\end{alignat}
Let $\mathcal{K}^{\Sigma_T}_h := \{ K \in \mathcal{K}_h | \partial K \cap \Sigma_T \neq \mbox{\scriptsize \O}\}$. 
By applying \eqref{eq:trace-inequality}, \eqref{eq:hk-and-h-relation}, \eqref{eq:c-s-interpolation-estimate-v1}, and 
\eqref{eq:c-s-interpolation-estimate-v2}, the estimate of the part of the norm on $\Sigma_T$ reads as\\
\begin{alignat}{2}
\| u -  \Pi_h u \|^2_{\Sigma_T} & = \sum\limits_{E \in \mathcal{E}^{\Sigma_T}_h} \| u -  \Pi_h u \|^2_{E} \nonumber\\
& \leq \sum\limits_{K \in \mathcal{K}^{\Sigma_T}_h} C^2_{t\!r} \, 
\Big( h^{\,-1}_{K} \| u - \Pi_h u \|^2_K + h_{K} \, |u - \Pi_h u|^2_{H^1(K)} \Big) \nonumber\\
& \leq \max\limits_{K \in \mathcal{K}^{\Sigma_T}_h} \{ C^2_{t\!r} \} \, C_{u} \, 
\Big( h^{-1} \, \sum\limits_{K \in \mathcal{K}^{\Sigma_T}_h} \| u - \Pi_h u \|^2_K 
+ h \, \sum\limits_{K \in \mathcal{K}^{\Sigma_T}_h} |u - \Pi_h u|^2_{H^1(K)} \Big) \nonumber\\
 & \leq \, \max\limits_{K \in \mathcal{K}^{\Sigma_T}_h} \{ C^2_{t\!r} \} \, C_{u} \, 
\Big( C^2_{0,s} 
\sum\limits_{K \in \mathcal{K}^{\Sigma_T}_h} h^{2s-1}_{K} \sum\limits_{i = 0}^{s} c^{2i}_K \, |u|_{H^i({\underline{K}})} 
\nonumber\\
& \qquad \qquad \qquad \qquad \qquad 
+ C^2_{1,s} \sum\limits_{K \in \mathcal{K}^{\Sigma_T}_h} h^{2s-1}_{K} \sum\limits_{i = 0}^{s} c^{2(i-1)}_K \, |u|_{H^i({\underline{K}})} \Big) \nonumber\\
& \leq C_{u} \, \max\limits_{K \in\mathcal{K}^{\Sigma_T}_h} 
\{ C^2_{t\!r} (C^2_{0,s} + C^2_{1,s} c^{-2}_K) \}
\sum\limits_{K \in \mathcal{K}^{\Sigma_T}_h} h^{2s-1}_{K} \sum\limits_{i = 0}^{s} c^{2i}_K \, |u|_{H^i({\underline{K}})} \nonumber\\
& \leq C_{\Sigma_T} \,  \sum\limits_{K \in \mathcal{K}^{\Sigma_T}_h} h^{2s-1}_{K} \sum\limits_{i = 0}^{s} c^{2i}_K \, |u|_{H^i({\underline{K}})},
\label{eq:estimate-sigmat}
\end{alignat}
where 
\begin{equation}
C_{\Sigma_T} = C_{u} \, \max\limits_{K \in\mathcal{K}^{\Sigma_T}_h} 
\Big\{ C^2_{t\!r} (C^2_{0,s} + C^2_{1,s} c^{-2}_K) \Big\}.
\label{eq:c-sigmaT}
\end{equation} 
Combining \eqref{eq:estimate-Q-nablax}--\eqref{eq:estimate-sigmat}, we obtain the bound 
\begin{equation}
| \! |\! | u - \Pi_h u  | \! |\! |^2_{l\!o\!c\!,h} 
\leq  {C}_1 \, \sum\limits_{K \in \mathcal{K}_h} h^{2(s-1)}_{K} \sum\limits_{i = 0}^{s} c^{2i}_K \, |u|_{H^i({\underline{K}})},
\end{equation}
where 
%
${C}_1 = \max\limits_{K \in \mathcal{K}_h} 
\Big \{ C^2_{1,s} \, (1 + \delta_K) c_K^{-2} + C_{\Sigma_T} \Big \}$ with constant $C_{\Sigma_T}$ defined in \eqref{eq:c-sigmaT}.

In order to prove \eqref{eq:int-3-loc}, we need to estimate 
$\sum\limits_{K \in \mathcal{K}_h } \delta^{-1}_{K} \| \cdot \|^2_{K}$ 
and $\delta_{K} \| \Delta_x v \|^2_{K}$ included into
$| \! |\! | \cdot  | \! |\! |_{l\!o\!c\!,h,*}$. First, using \eqref{eq:c-s-interpolation-estimate-v1}, 
we obtain
\begin{alignat*}{2}
\sum\limits_{K \in \mathcal{K}_h } \delta^{-1}_{K} \| u -  \Pi_h u \|^2_{K} 
& \leq C^2_{0,s} \, C_u \, \max\limits_{K \in \mathcal{K}_h} \tfrac{h_K}{\theta_K} \, 
\sum\limits_{K \in \mathcal{K}_h} h^{2(s-1)}_{K} \sum\limits_{i = 0}^{s} c^{2i}_K \, |u|_{H^i({\underline{K}})}.
\end{alignat*}
By accounting $\theta_K \leq \tfrac{h_K}{d\, C^2_{i\!n\!t\!,\!1}}$ and \eqref{eq:c-s-interpolation-estimate-v3},
the second term is bounded as follows 
\begin{alignat*}{2}
\sum\limits_{K \in \mathcal{K}_h } \, \delta_{K} \| \Delta_x (u - \Pi_h u) \|^2_{K} 
&  \leq \sum\limits_{K \in \mathcal{K}_h}  \tfrac{h_K}{d\, C^2_{i\!n\!t\!,\!1}} \, h_K \, d \, | u - \Pi_h u |^2_{H^2(K)} \\
& \leq 
C^2_{2,s} \, \sum\limits_{K \in \mathcal{K}_h}  \, C^{-2}_{i\!n\!t\!,\!1} \, h^2_K \, c_K^{-4} \, h^{2(s - 2)}_{K} \, 
\sum_{i = 0}^s c_K^{2i} \, \| u \|^2_{H^s({\underline{K}})} \\
& \leq C^2_{2,s} \, \max\limits_{K \in \mathcal{K}_h} \Big\{ C^{-2}_{i\!n\!t\!,\!1} \, c_K^{-4} \Big\}
\sum\limits_{K \in \mathcal{K}_h} h^{2(s - 1)}_{K} \, \sum_{i = 0}^s c_K^{2i} \, \| u \|^2_{H^s({\underline{K}})}.             
\end{alignat*}
%
%
Thus, we obtain
\begin{equation*}
| \! |\! | u - \Pi_h u  | \! |\! |^2_{l\!o\!c\!,h,*} 
\leq {C}_2 \, \max\limits_{K \in \mathcal{K}_h} \tfrac{h_K}{\theta_K} \, 
\sum\limits_{K \in \mathcal{K}_h} h^{2(s-1)}_{K} \sum\limits_{i = 0}^{s} c^{2i}_K \, |u|_{H^i({\underline{K}})},
\end{equation*}
where  
\begin{equation}
{C}_2
= \max\limits_{K \in \mathcal{K}_h} 
\Big\{C_{\Sigma_T} + C^2_{2,s} \, C^{-2}_{i\!n\!t\!,\!1} \, c_K^{-4} +  C^2_{0,s} \, C_{u}\, \tfrac{h_K}{\theta_K}\Big\},
\label{eq:c-2}
\end{equation}
where $C_{\Sigma_T} $ is defined in \eqref{eq:c-sigmaT}.
\end{proof}

\subsection{Consistency}
\label{ssec:consistency-a-tilde-loc}

\begin{lemma}
\label{lm:a-hk-consistency}
If the solution $u \in V^{1, 0}_{0}$ of \eqref{eq:variational-formulation} also belongs to 
{ $V^{\Delta_x, 1}_{0, \underline{0}}$, }
then it satisfies the consistency identity 
\begin{equation}
a_{l\!o\!c\!,h}(u, v_h) = \ell_{l\!o\!c\!,h}(v_h),  \quad v_h \in V_{0h}.
\label{eq:a-hk-consistency}
\end{equation} 
\end{lemma}
\begin{proof}
Since { $u \in V^{\Delta_x, 1}_{0, \underline{0}}$}, by integration by parts in 
\eqref{eq:variational-formulation} w.r.t. to $x$ and $t$ as well as density arguments, 
we obtain $u_t - \Delta_x u = f$ in $L_2(Q)$ and $u \big|_{\Sigma} = 0$.
%
%
%
%
The consistency identity $a_{l\!o\!c\!,h}(u, v_h) = \ell_{l\!o\!c\!,h}(v_h)$,  $v_h \in V_{0h}$ 
is derived along with the discrete space-time formulation \eqref{eq:a-hk-l}.
\end{proof}
%
\subsection{A priori estimates of approximation errors}
\label{ssec:apriori-a-tilde-loc}
{
\begin{lemma}
\label{lm:a-hk-apriori}
Let $u \in V^{\Delta_x, 1}_{0, \underline{0}}$ be an exact solution of 
\eqref{eq:variational-formulation}, and $u_h \in V_{0h}$ (with $p \geq 2$) 
be an approximate solution generated by \eqref{eq:a-hk-l}. Then, the best approximation estimate
\begin{equation}
|\!|\!| u  - u_h |\!|\!|_{l\!o\!c\!,h} 
\leq (1 + \tfrac{\mu_{l\!o\!c, b}}{\mu_{l\!o\!c, c}}) \, 
\inf_{v_h \in V_{0h}} \, \| u  - v_h\|_{l\!o\!c\!,h,*}
\label{eq:clothest-approximation}
\end{equation}
holds. Here, ${\mu_{l\!o\!c, c}}$ and ${\mu_{l\!o\!c, b}}$ are positive constants from Lemmas 
\ref{lm:a-tilde-hk-coercivity} and \ref{lm:a-tilde-hk-boundedness}, respectively, that do not depend on 
$h_K$.
\end{lemma}
}
\begin{proof}
The Galerkin orthogonality 
\begin{equation}
a_{l\!o\!c\!,h}(u - u_h, v_h) = 0.
\label{eq:a-hk-galerkin-orthogonality}
\end{equation}
follows from \eqref{eq:a-hk-consistency}.
Applying the triangle inequality, we estimate the discretisation error $u - u_h$ as follows:
\begin{equation}
|\!|\!| u - u_h |\!|\!|_{l\!o\!c\!,h} 
\leq |\!|\!| u - \Pi_h u  |\!|\!|_{l\!o\!c\!,h} + |\!|\!| \Pi_h u  - u_h  |\!|\!|_{l\!o\!c\!,h}.
\label{eq:l-13-triangle-ineq}
\end{equation}
The first term on the RHS of \eqref{eq:l-13-triangle-ineq} can easily be estimated by means of Lemma 
\ref{lm:lemma-6-loc}. 
For the estimation of $|\!|\!| \Pi_h u  - u_h  |\!|\!|_{l\!o\!c\!,h}$, we first use $V_{0h}$-ellipticity of 
$a_{l\!o\!c\!,h}(\cdot, \cdot)$ w.r.t. the norm $|\!|\!| \cdot |\!|\!|_{l\!o\!c\!,h}$ (see Lemma 
\ref{lm:a-tilde-hk-coercivity}), i.e.,
$$ \mu_{l\!o\!c, c} \, |\!|\!| \Pi_h u  - u_h  |\!|\!|_{l\!o\!c\!,h} \leq |a_{l\!o\!c\!,h}(\Pi_h u  - u_h, \Pi_h u  - u_h)|.$$
Next, by means of the Galerkin orthogonality \eqref{eq:a-hk-galerkin-orthogonality}, we obtain
\begin{alignat*}{2}
\mu_{l\!o\!c, c} \,  |\!|\!| \Pi_h u  - u_h |\!|\!|^2_{l\!o\!c\!,h} 
& \leq a_{l\!o\!c\!,h}(\Pi_h u  - u_h, \Pi_h u  - u_h)  \\[0pt]
& = a_{l\!o\!c\!,h}(\Pi_h u  - u, \Pi_h u  - u_h).
\end{alignat*}
Finally, { we apply Lemma \ref{lm:a-tilde-hk-coercivity} and obtain the estimate}
$$\mu_{l\!o\!c, c} \, |\!|\!| \Pi_h u  - u_h |\!|\!|^2_{l\!o\!c\!,h} \leq \mu_{l\!o\!c, b}  \, \| \Pi_h u  - u \|_{l\!o\!c\!,h,*}\, \| \Pi_h u  - u \|_{l\!o\!c\!,h},$$
which automatically yields
\begin{equation}
|\!|\!| \Pi_h u  - u_h |\!|\!|_{l\!o\!c\!,h} \leq \tfrac{\mu_{l\!o\!c, b}}{\mu_{l\!o\!c, c}} \, \| \Pi_h u  - u \|_{l\!o\!c\!,h,*}.
\label{eq:eq-1}
\end{equation}
Combining $\| \Pi_h u  - u \|_{l\!o\!c\!,h} \leq \| \Pi_h u  - u \|_{l\!o\!c\!,h,*}$, 
\eqref{eq:eq-1}, and \eqref{eq:l-13-triangle-ineq}, we arrive at
\begin{equation*}
|\!|\!| u  - u_h |\!|\!|_{l\!o\!c\!,h} 
\leq (1 + \tfrac{\mu_{l\!o\!c, b}}{\mu_{l\!o\!c, c}}) \, 
\| u  - \Pi_h u \|_{l\!o\!c\!,h,*}.
\end{equation*}
\end{proof}

\begin{theorem}
\label{th:a-hk-apriori}
Let $p \geq 2$, $u \in V^{s}_{0}$, $s \geq 2$, be an exact solution of 
\eqref{eq:variational-formulation}, and $u_h \in V_{0h}$ be an approximate solution of  
\eqref{eq:a-hk-l} with $\theta_{K} \in \Big(0, \tfrac{h_K}{d\, C^2_{i\!n\!t\!,\!1}} \Big]$, 
$K \in \mathcal{K}_h$. Then, the discretisation error estimate
\begin{equation*}
|\!|\!| u - u_h |\!|\!|^2_{l\!o\!c\!,h} \leq C 
\sum\limits_{K \in \mathcal{K}_h}  \, h^{2(s - 1)}_{K} \, 
\sum_{i = 0}^s c_K^{2i} \, | u |^2_{H^i(K)}
\end{equation*}
%
%
hold, where $C = (1 + \tfrac{\mu_{l\!o\!c, b}}{\mu_{l\!o\!c, c}}) \, C_2$ is a constant independent of $h$,
$r = \min \{ s, p+1 \}$, and $p$ denotes the polynomial degree of the THB-splines, 
$\mu_{l\!o\!c, b}$ and $\mu_{l\!o\!c, c}$ are constant in boundedness \eqref{eq:coercivity} and 
coercivity \eqref{eq:boundedness} inequalities, respectively.
\end{theorem}

\begin{proof}

Application of estimate \eqref{eq:int-3-loc} yields
%
\begin{alignat*}{2}
|\!|\!| u - u_h |\!|\!|_{l\!o\!c\!,h} 
& \leq  
{(1 + \tfrac{\mu_{l\!o\!c, b}}{\mu_{l\!o\!c, c}}) \, C_2 }\, 
\sum\limits_{K \in \mathcal{K}_h}  \, h^{2(s - 1)}_{K} \, \sum_{i = 0}^s c^{2i}_K \, |u|_{H^i({\underline{K}})},
\end{alignat*}
where $C_2$ is defined in \eqref{eq:c-2}.
\end{proof}

\section{A posteriori error estimates and numerical experiments}

In this section, we discuss the implementation of the numerical scheme discussed above and the estimates 
used to control the quality of  approximations. 
Numerical experiments present the error order of convergence  (e.o.c.) in terms of the error norm 
\eqref{eq:norm-K-h}. Also, we discuss computational properties of the majorants
$\overline{\rm M}^{\rm I}$ and $\overline{\rm M}^{\rm I\!I}$ that follow from \cite{LMRRepin2002}
and of the error identity
${{\rm E \!\!\! Id}}$ \cite{LMRAnjamPauly2016}. Moreover, we compare  time expenditures that are
required  for getting approximations of the solution  with the time spent for computing efficient error 
bounds.

Let $u_h$ denote an approximation of $u$.  We assume that  \linebreak
$u_h \in V_{0h} := V_h \cap V^{\Delta_x, 1}_{0, \underline{0}}$ (cf. \eqref{eq:vh-v0h}),
%
and define
$$u_h(x, t) = u_h(x_1, . . . , x_{d+1}) 
:= \sum_{i \in \mathcal{I}} \underline{ \rm u}_{h,i} \, \phi_{h, i} (x_1, . . . , x_{d+1}) ,
$$
where $\underline{\rm u}_h 
:= \big[ \underline{ \rm u}_{h,i}\big]_{i \in \mathcal{I}} \in {\mathds{R}}^{|\mathcal{I}|}$
contains free parameters to be defined (it is the vector of degrees of freedom (d.o.f.) or, 
in the IgA community, vector of control points). This vector is generated by the linear system
%
\begin{equation}
{\rm K}_h \, \underline{ \rm  u}_h = {\rm f}_h, \quad 
{\rm K}_h := \big[a_{l\!o\!c\!,h}(\phi_{h,i},\phi_{h,j}) \big]_{i, j \in \mathcal{I}},
\quad
{\rm f}_h := \big[l_{l\!o\!c\!,h}(\phi_{h,i}) \big]_{i \in \mathcal{I}}.
\label{eq:system-uh}
\end{equation}
{ The system \eqref{eq:system-uh} is solved} by means of the sparse direct ${\rm LU}$ 
factorisations. This choice of the solution method
is { motivated by our intention to provide a fair comparison of time expenditures 
used for solving the system generating $u_h$ and $y_h$ (for the majorant $\overline{\rm M}^{\rm I}$) 
as well as $w_h$ (for $\overline{\rm M}^{\rm I\!I}$).}
Due to properties of IgA approximations the condition
$u_h \in C^{p-1}$ is automatically provided
%


{Approximation properties} of $u_h$ are analysed by studying convergence of the error 
$e = u - u_h$ measured in terms of different norms. The first 
norm is defined in \eqref{eq:norm-K-h} and the second one is
\begin{equation*}
{|\!|\!|  e |\!|\!|^2} := \| \nabla_x e \|^2_Q + \| e \|^2_{\Sigma_T}.
\end{equation*}
%
%
The norm ${|\!|\!|  e |\!|\!|^2}$ is controlled by the majorant (see, e.g.,  \cite{LMRRepin2002}) 
\begin{alignat*}{2}
\overline{\rm M}^{\rm I}(u_h,\boldsymbol{y}_h) 
:= \, & (1 + \beta) \, \|\boldsymbol{y}_h - \nabla_x u_h \|^2_Q
    + (1 + \tfrac{1}{\beta}) \, C_{{\rm F}}^2 \, 
    \| {\rm div}_x\boldsymbol{y}_h + f  - \partial_t u_h\|^2_Q\\
= \, & (1 + \beta) \,  \overline{\mathrm m}^{{\rm I}, 2}_{\mathrm{d}} 
   + (1 + \tfrac{1}{\beta}) \, C_{{\rm F}}^2 \,  \overline{\mathrm m}^{{\rm I}, 2}_{\mathrm{eq}}, 
\end{alignat*}
where $\beta > 0$ and ${\boldsymbol{y}}_h \in Y_h \subset H^{{\rm div}_x, 0}(Q)$. The space 
$$Y_h \equiv \mathcal{S}^{q}_h 
:= \big\{ {{\boldsymbol \psi}}_{h,i} := \oplus^{d+1} \hat{\mathcal{S}}^{q}_h \circ \Phi^{-1} \big\}$$
is generated by the push-forward of $\oplus^{d+1} \hat{\mathcal{S}}^{q}_h$, where 
$\hat{\mathcal{S}}^{q}_h$ is the space of splines of the degree $q$ used to approximate components
of $\boldsymbol{y}_h = \big(y_h^{(1)}, \ldots ,y_h^{(d+1)}\big)^{\rm T}$. 
The sharpest estimate is obtained by the minimisation of $\overline{\rm M}^{\rm I}(u_h,\boldsymbol{y}_h)$ 
w.r.t. 
$$\boldsymbol{y}_h(x, t) =\boldsymbol{y}_h(x_1, . . . , x_{d+1}) 
= \sum_{i \in \mathcal{I} \times (d+1)} 
\underline{ \rm \bf y}_{h,i} \, {{\boldsymbol \psi}}_{h,i}(x_1, . . . , x_{d+1}).$$
Here, $\underline{ \rm \bf y}_{h} 
:= \big[ \underline{ \rm \bf y}_{h,i}\big]_{i \in \mathcal{I}} \in {\mathds{R}}^{(d+1)|\mathcal{I}|}$ 
{(i.e., it is a vector of the dimension $(d+1)|\mathcal{I}|$)}  is defined by the linear system
\begin{equation}
\left( {C_{{\rm F}}^2} \, {\rm Div}_h + {\beta} \, {\rm M}_h \right)\, \underline{ \rm \bf y}_{h} 
= - {C_{{\rm F}}^2} \, {\rm z}_h + {\beta} \, {\rm g}_h, 
\label{eq:system-fluxh}
\end{equation}
%
where 
\begin{equation*}
\begin{array}{r@{$\;$}l l} 
	{ {\rm Div}_h} & 
	:= 
	\big[ ({\rm div}_x {{\boldsymbol \psi}}_{h,i}, {\rm div}_x {{\boldsymbol \psi}}_{h,j})_Q
	\big]_{i, j=1}^{(d+1)|\mathcal{I}|} 
, \quad 
	{\rm z}_{h} := 
	\big[\big(f - v_t,  {\rm div}_x {{\boldsymbol \psi}}_{h,j} \big)_Q  
	\big]_{j=1}^{(d+1)|\mathcal{I}|}  
	,  
	\\[5pt]
	{{\rm M}_h} & := 
	\big[ ({{\boldsymbol \psi}}_{h,i}, {{\boldsymbol \psi}}_{h,j})_Q
	\big]_{i, j=1}^{(d+1)|\mathcal{I}|}
	, 
	\qquad \qquad \;\;
	{\rm g}_h := 
	\big[ \big(\nabla_x v, {{\boldsymbol \psi}}_{h,j}\big)_Q 
	\big]_{j=1}^{(d+1)|\mathcal{I}|}. 
\end{array}
\end{equation*}
The optimal value for ${\beta}$ reads as ${\beta} 
:= {C_{{\rm F}}} \,  \overline{\mathrm m}^{{\rm I}}_{\mathrm{eq}} \, / \, \overline{\mathrm m}^{{\rm I}}_{\mathrm{d}}$.
According to numerical results obtained in 
\cite{LMRKleissTomar2015, LMRMatculevich2017, LMRLangerMatculevichRepinArxiv2017}, 
the most efficient majorant reconstruction is obtained with spline degree $q \gg p$. 
{ At the same time,
the approximation $u_h$ is reconstructed on the mesh $\mathcal{K}_h$, whereas a coarser mesh 
$\mathcal{K}_{Mh}$,  $M \in \mathds{N}^+$, is used to recover $\boldsymbol{y}_h$. This helps 
to minimise the number of d.o.f. for the latter one. 
The initial mesh $\mathcal{K}^0_{h}$ and corresponding basis functions are assumed to be given 
via the geometry representation of the computational domain. Throughout the set of numerical examples, 
$\mathcal{K}^0_{h}$ are generated by $N_{\rm ref, 0}$ initial uniform refinements
before actual testing.} In our implementation, 
\eqref{eq:system-fluxh} is solved by the sparse direct ${\rm LDL^{\rm T}}$ Cholesky factorisations.

In addition to $\overline{\rm M}^{\rm I}$, \cite{LMRRepin2002} provides an advanced form of the 
majorant $\overline{\rm M}^{\rm I\!I}(u_h,\boldsymbol{y}_h, w_h)$, i.e.,
%
%
\begin{alignat*}{2}
\| \nabla_x e \|^2_Q 
& \leq \overline{\rm M}^{\rm I\!I} (u_h, w_h) \\
& := \|  w_h - u_h \|^2_{\Sigma_T} + 2 \, \mathcal{F}(u_h, w_h) 
+ (1 + {\beta} ) \, \big\| {{\mathbf{r}}}^{\rm I\!I} _{\rm d}\big\|^2_Q
+ C_{{\rm F}}^2 \, (1 + \tfrac{1}{{\beta}} ) \, \big\| {{\mathbf{r}}}^{\rm I\!I}_{\rm eq}\big\|^2_Q, 
\end{alignat*}
where
\begin{alignat*}{2}
\mathcal{F}(u_h, w_h) 
& := (\nabla_x u_h, \nabla_x (w_h - u_h)) + (\partial_t u_h - f, w_h - u_h ), \\
{{\mathbf{r}}}^{\rm I\!I} _{\rm d}(u_h,\boldsymbol{y}_h, w_h) 
& :=\boldsymbol{y}_h + \nabla_x w_h - 2 \, \nabla_x u_h, \quad \mbox{and} \quad \\
{{\mathbf{r}}}^{\rm I\!I}_{\rm eq}(\boldsymbol{y}_h, w_h) 
& := {\rm div}_x\boldsymbol{y}_h + f - \partial_t w_h.
\end{alignat*}
%
Here, $w_h$ is the solution to \eqref{eq:a-hk-l} on the approximation space
$${ %
{W}_{0h} := W_h \cap V^{\Delta_x, 1}_{0}, \quad 
\mbox{with} \quad 
W_h \equiv {\mathcal{S}}^{r}_h 
:=\Big\{ \chi_{h,i} := \hat{\mathcal{S}}^{r}_h \circ {\chi}^{-1} \Big\},}$$ 
where ${\hat{\mathcal{S}}}^{r}_h$ is the space of degree $r$. 
The function $w_h$ can be represented by
$$w_h(x, t) = w_h(x_1, . . . , x_{d+1}) 
:= \sum_{i \in \mathcal{I}} \underline{ \rm w}_{h,i} \, \chi_{h,i}.
$$
Here, 
$\underline{\rm w}_h 
:= \big[ \underline{ \rm w}_{h,i}\big]_{i \in \mathcal{I}} \in {\mathds{R}}^{|\mathcal{I}|}$
is the vector of control points of $w_h$ defined by the linear system
%
${\rm K}^{(r)}_h \, \underline{ \rm  w}_h = {\rm f}^{(r)}_h,$
%
where 
${\rm K}^{(r)}_h := \big[a_{l\!o\!c\!,h}(\chi_{h, i}, \chi_{h, j}) \big]_{i, j \in \mathcal{I}}$, \linebreak
${\rm f}^{(r)}_h :=\big[l_{l\!o\!c\!,h}(\chi_{h, i}) \big]_{i \in \mathcal{I}}$.
%
Since $\partial_t w_h$ is approximated by a richer space, the term
$\big\| {{\mathbf{r}}}^{\rm I\!I}_{\rm eq}(\boldsymbol{y}_h, w_h) \big\|^2_Q$ is expected to be 
smaller than $\| {{\mathbf{r}}}_{\rm eq}(\boldsymbol{y}_h, u_h) \|^2_Q$. Therefore, the value of 
the error bound $\overline{\rm M}^{\rm I\!I}$ must be improved. The optimal parameter ${\beta}$ is 
calculated by \linebreak
${\beta} 
:= {C_{{\rm F}}} \| {{\mathbf{r}}}^{\rm I\!I}_{\rm eq}\|_Q / 
\| {{\mathbf{r}}}^{\rm I\!I}_{\rm d}\|_Q$.

The last error norm is generated by the solution operator $\mathcal{L} := \partial_t - \Delta_x$, i.e.,
\begin{equation*}
|\!|\!|  e |\!|\!|^2_{\mathcal{L}} 
:= \| \Delta_x e \|^2_Q + \| \partial_t e \|^2_Q +  \| \nabla_x e \|^2_{\Sigma_T}.
\end{equation*}
It is controlled by the error identity \cite{LMRAnjamPauly2016}
\begin{equation*}
{{\rm E \!\!\! Id}}^2 (u_h) 
:= \| \nabla_x (u_0 - u_h) \|^2_{\Sigma_0} + \|  \Delta_x u_h + f - \partial_t u_h\|^2_Q.
\end{equation*}
Marking of the elements in $\mathcal{K}_h$ is driven by the bulk marking criterion (also known as 
D\"orfler's marking \cite{LMRDoerfler1996}) denoted by ${\mathds{M}}_{\rm BULK}({\sigma})$, 
$\sigma \in [0, 1]$. Finally, the effectiveness of the error estimators is evaluated by 
efficiency indices, i.e.,
\begin{alignat*}{2}
I_{\rm eff} (\overline{\rm M}^{\rm I}) 
:= \tfrac{\overline{\rm M}^{\rm I}}{|\!|\!| e |\!|\!|_Q}, \quad
I_{\rm eff} (\overline{\rm M}^{\rm I\!I}) 
:= \tfrac{\overline{\rm M}^{\rm I\!I}}{\| \nabla_x e \|_Q}, 
\quad 
I_{\rm eff} ({{{\rm E \!\!\! Id}}}) := \tfrac{{{{\rm E \!\!\! Id}}}}{|\!|\!|  e |\!|\!|_{\mathcal{L}}} = 1.
\end{alignat*}

{ 
Below we study the behaviour of the above discussed error control tools within a series of benchmark 
examples. We begin with a rather simple example, which is intended to demonstrate important properties
of the numerical scheme. More complicated problems with non-trivial geometries and singular solutions 
are considered at the end of the section. The implementation was carried out using the open-source C++ 
library G+Smo \cite{LMRgismoweb2015}.}
%

\subsection{Example 1: polynomial solution}
\label{ex:unit-domain-example-2}
{
First, we consider a simple example, where the solution of \eqref{eq:equation} is a polynomial function}
$$u(x,t) = (1 - x) \, x^2 \,  (1 - t) \, t, \quad (x, t) \in \overline{Q} := [0, 1]^2, $$ 
and { generated by it RHS}
$$f(x,t) = - (1 - x)\, x^2 \,(1 - 2\,t) - (2 - 6\,x) \, (1 - t)\,t, \quad (x, t) \in Q := (0, 1)^2.$$ 
$u(x, t)$ satisfies homogeneous Dirichlet boundary and initial conditions on 
$\Sigma := \partial \Omega \times (0, 1)$ and ${\overline{\Sigma}}_0$, respectively.

The initial mesh is obtained by one global refinement ($N_{\rm ref, 0} = 1$). 
{ Further, refinements are done with eight steps (hence $N_{\rm ref} = 8$).}
The approximation space for $u_h$ is $S_{h}^{2}$. For the auxiliary functions, we assume 
that $\boldsymbol{y}_h \in \oplus^2 S_{5h}^{3}$, and $w_h \in S_{5h}^{3}$. { 
Such a choice of discretisation
spaces saves computational efforts in reconstruction of the error estimates considerably.} 
Table \ref{tab:unit-domain-example-2-time-expenses-v-2-y-3-adapt-ref} 
illustrates the ratio between the time spent for approximating $u_h$ to the time spent for its error estimation, 
i.e., $\tfrac{t_{\rm appr.}}{t_{\rm er.est.}}$, along with total time needed 
for assembling and solving systems generating d.o.f. of $u_h$, $\boldsymbol{y}_h$, and $w_h$.
%
Table \ref{tab:unit-domain-example-2-error-majorant-v-2-y-3-adapt-ref} illustrates convergences of the 
different error measures, i.e., $|\!|\!| e |\!|\!|_Q$ that is bounded by the majorants 
$\overline{\rm M}^{\rm I}$ and $\overline{\rm M}^{\rm I\!I}$, $|\!| e |\!|_{l\!o\!c\!,h}$, and 
$|\!|\!|  e |\!|\!|_{\mathcal{L}}$ controlled by ${{\rm E \!\!\! Id}}$.
{ Here, we consider bulk marking with parameters $\sigma = 0.4$ and 
$\sigma = 0.6$.} Both cases provide slightly improved convergences in comparison to the expected 
$O(h^2)$ for $|\!|\!|  e |\!|\!|_{l\!o\!c\!,h}$ and $O(h)$ for $|\!|\!|  e |\!|\!|_{\mathcal{L}}$. 
The time expenses for the $u_h$ naturally get lower in the case of $\sigma = 0.6$, since 
the d.o.f.($u_h$) does not grow as fast as in case with $\sigma = 0.6$.

{
Moreover, we compare the error order of convergence in Figure \ref{fig:example-1-eoc}. 
Here, the majorant is reconstructed with auxiliary functions $\boldsymbol{y}_h \in \oplus^2 S_{h}^{3}$ 
($M = 1$) and $\boldsymbol{y}_h \in \oplus^2 S_{7h}^{3}$ ($M = 7$). } 
The numerical test demonstrates that the efficient error estimation and 
its local indication can be achieved even using auxiliary fluxes on a very course mesh (in this particular 
case, $7$ times courser then the mesh $\mathcal{K}_{h}$ for $u_h$). 

\begin{table}[!t]
\scriptsize
\centering
\newcolumntype{g}{>{\columncolor{gainsboro}}c} 	
\newcolumntype{k}{>{\columncolor{lightgray}}c} 	
\newcolumntype{s}{>{\columncolor{silver}}c} 
\newcolumntype{a}{>{\columncolor{ashgrey}}c}
\newcolumntype{b}{>{\columncolor{battleshipgrey}}c}
\begin{tabular}{c|cgg|c|cg|gc}
\parbox[c]{0.4cm}{\centering \# ref. } & 
\parbox[c]{1.2cm}{\centering  $\| \nabla_x e \|_Q$}   & 	  
\parbox[c]{1.0cm}{\centering $I_{\rm eff} (\overline{\rm M}^{\rm I})$ } & 
\parbox[c]{1.0cm}{\centering $I_{\rm eff} (\overline{\rm M}^{\rm I\!I})$ } & 
\parbox[c]{1.2cm}{\centering  $|\!|\!|  e |\!|\!|_{l\!o\!c\!,h}$ }   & 	  
\parbox[c]{1.2cm}{\centering  $|\!|\!|  e |\!|\!|_{\mathcal{L}}$ }   & 	  
\parbox[c]{1.0cm}{\centering$I_{\rm eff} ({{\rm E \!\!\! Id}})$ } & 
\parbox[c]{1.2cm}{\centering e.o.c. ($|\!|\!|  e |\!|\!|_{l\!o\!c\!,h}$)} & 
\parbox[c]{1.2cm}{\centering e.o.c. ($|\!|\!|  e |\!|\!|_{\mathcal{L}}$)} \\[3pt]
\bottomrule
\multicolumn{9}{l}{ \rule{0pt}{3ex}   
(a) $\sigma =0.4$}\\[3pt]
\toprule
  2 &     2.5516e-03 &         1.07 &         1.03 &     2.5520e-03 &  
  7.9057e-02 &         1.00 &     3.43 &     1.71 \\
   4 &     2.2743e-04 &         1.41 &         1.19 &     2.2745e-04 &  
   2.1712e-02 &         1.00 &     2.36 &     1.37 \\
   6 &     2.9936e-05 &         1.09 &         1.02 &     2.9936e-05 &  
   7.9512e-03 &         1.00 &     2.71 &     1.28 \\
   8 &     4.9501e-06 &         1.12 &         1.05 &     4.9501e-06 &  
   3.1138e-03 &         1.00 &     1.51 &     0.93 \\
\bottomrule
\multicolumn{9}{l}{ \rule{0pt}{3ex}   
(b) $\sigma =0.6$}\\[3pt]
\toprule
  2 &     2.5516e-03 &         1.07 &         1.03 &     2.5520e-03 & 
  7.9057e-02 &         1.00 &     3.43 &     1.71 \\
   4 &     3.3298e-04 &         1.30 &         1.11 &     3.3305e-04 & 
   2.5410e-02 &         1.00 &     1.80 &     1.22 \\
   6 &     5.9048e-05 &         1.34 &         1.14 &     5.9050e-05 &  
   1.0976e-02 &         1.00 &     3.18 &     1.60 \\
   8 &     2.3071e-05 &         1.25 &         1.10 &     2.3072e-05 &  
   6.7335e-03 &         1.00 &     2.06 &     1.41 \\
\end{tabular}
\caption{{\em Example 1}. 
Efficiency of $\overline{\rm M}^{\rm I}$, $\overline{\rm M}^{\rm I\!I}$, and ${{\rm E \!\!\! Id}}$ and for 
$u_h \in S^{2}_{h}$, $\boldsymbol{y}_h \in \oplus^2 S^{3}_{5h}$, and $w_h \in S^{3}_{5h}$, and 
{ order of convergence} for $|\!|\!|  e |\!|\!|_{l\!o\!c\!,h}$ and $|\!|\!|  e |\!|\!|_{\mathcal{L}}$ 
($N_{\rm ref, 0} = 1$).}
\label{tab:unit-domain-example-2-error-majorant-v-2-y-3-adapt-ref}
\end{table}

\begin{table}[!t]
\scriptsize
\centering
\newcolumntype{g}{>{\columncolor{gainsboro}}c} 	
\begin{tabular}{c|ccc|cgg|cgg|c}
& \multicolumn{3}{c|}{ d.o.f. } 
& \multicolumn{3}{c|}{ $t_{\rm as}$ }
& \multicolumn{3}{c|}{ $t_{\rm sol}$ } 
& $\tfrac{t_{\rm appr.}}{t_{\rm er.est.}}$ \\
\midrule
\parbox[c]{0.2cm}{\# ref. } & 
\parbox[c]{0.4cm}{\centering $u_h$ } &  
\parbox[c]{0.4cm}{\centering $\boldsymbol{y}_h$ } &  
\parbox[c]{0.4cm}{\centering $w_h$ } & 
\parbox[c]{0.8cm}{\centering $u_h$ } & 
\parbox[c]{0.8cm}{\centering $\boldsymbol{y}_h$ } & 
\parbox[c]{0.8cm}{\centering $w_h$ } & 
\parbox[c]{0.8cm}{\centering $u_h$ } & 
\parbox[c]{0.8cm}{\centering $\boldsymbol{y}_h$ } & 
\parbox[c]{0.8cm}{\centering $w_h$ } &
\\
\bottomrule
\multicolumn{10}{l}{ \rule{0pt}{3ex}   
(a) $\sigma =0.4$ } \\[3pt]
\toprule
   4 &        240 &         50 &         25 &   2.64e-01 &   1.44e-02 &   1.08e-02 &         4.01e-03 &         2.36e-04 &         1.15e-04 &            18.3 \\
   6 &       2027 &         50 &         25 &   2.42e+00 &   1.81e-02 &   1.61e-02 &         1.94e-01 &         2.23e-04 &         1.34e-04 &            142.66 \\
   8 &      11512 &        152 &         76 &   1.35e+01 &   1.97e-01 &   1.68e-01 &         3.39e+00 &         7.13e-04 &         3.76e-04 &            85.42 \\
    \midrule
    &       &         &    &
    \multicolumn{3}{c|}{ $t_{\rm as} (u_h)$ \;:\; $t_{\rm as} (\boldsymbol{y}_h)$ \;:\; $t_{\rm as} (w_h)$ } &      
    \multicolumn{3}{c|}{\; $t_{\rm sol} (u_h)$ \;:\; $t_{\rm sol} (\boldsymbol{y}_h)$  \;:\;  $t_{\rm sol} (w_h)$\;} & \\
 \midrule
	 & 	 & 	 & 	 &      80.13 &       1.17 &       1.00 &          9021.85 &             1.90 &             1.00 &             \\
\bottomrule
\multicolumn{10}{l}{ \rule{0pt}{3ex}   
(b) $\sigma =0.6$} \\[3pt]
\toprule
%
   4 &        206 &         50 &         25 &   2.27e-01 &   1.94e-02 &   1.54e-02 &         3.35e-03 &         1.66e-04 &         8.60e-05 &             11.77 \\
   6 &        896 &         50 &         25 &   1.07e+00 &   1.20e-02 &   1.96e-02 &         4.17e-02 &         2.30e-04 &         1.25e-04 &            90.89 \\
   8 &       2706 &        158 &         79 &   3.44e+00 &   1.65e-01 &   1.37e-01 &         2.69e-01 &         6.72e-04 &         7.31e-04 &            22.3 \\
    \midrule
    &       &         &    &
    \multicolumn{3}{c|}{ $t_{\rm as} (u_h)$ \;:\; $t_{\rm as} (\boldsymbol{y}_h)$ \;:\; $t_{\rm as} (w_h)$ } &      
    \multicolumn{3}{c|}{\; $t_{\rm sol} (u_h)$ \;:\; $t_{\rm sol} (\boldsymbol{y}_h)$  \;:\;  $t_{\rm sol} (w_h)$\;} & \\
 \midrule
 	 & 	 & 	 & 	 &      25.20 &       1.21 &       1.00 &           367.33 &             0.92 &             1.00 &            12.25 \\
\end{tabular}
\caption{{\em Example 1}. 
Assembling and solving time (in seconds) spent for the systems defining d.o.f. of 
$u_h \in S^{2}_{h}$, 
$\boldsymbol{y}_h \in \oplus^2 S^{3}_{5h}$, and 
$w_h \in S^{3}_{5h}$ ($N_{\rm ref, 0} = 1$).}
\label{tab:unit-domain-example-2-time-expenses-v-2-y-3-adapt-ref}
\end{table}

\begin{figure}[!t]
	\centering
	\includegraphics[scale=0.6]{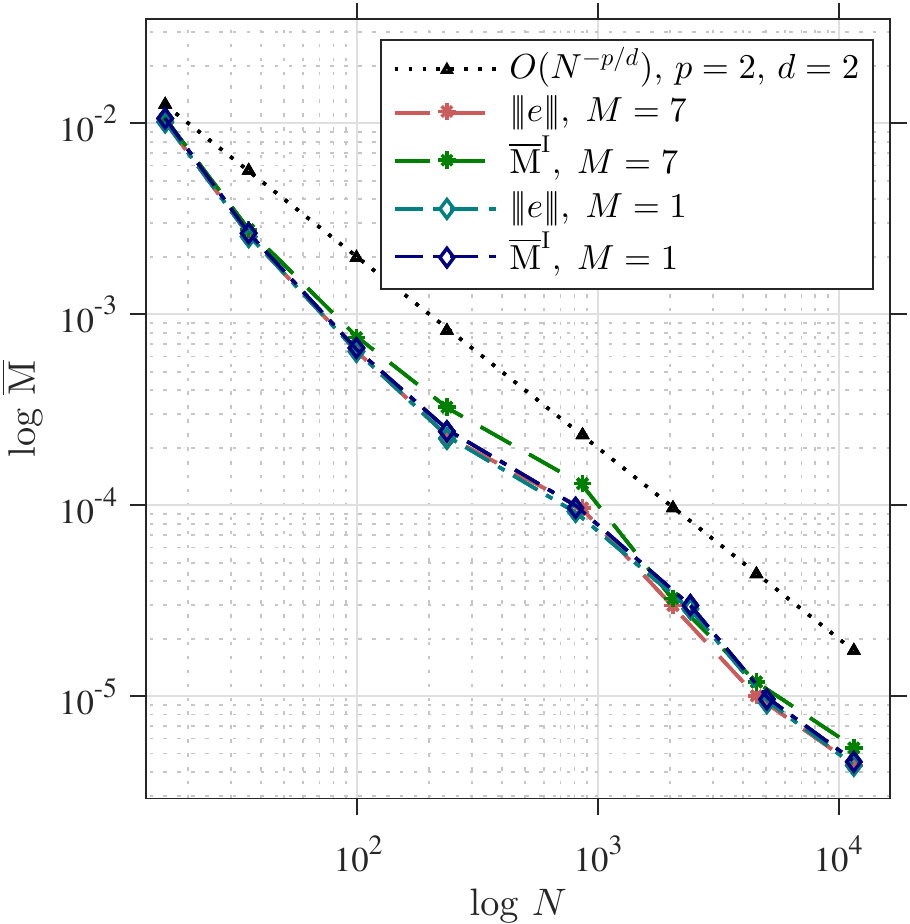}
	\label{fig:example-3-2-exact-solution-a}
	\caption{{\em Example 1}. 
	Comparison of the error and majorant $\overline{\rm  M}^{\rm I}$ 
	order of convergence for $\boldsymbol{y}_h \in \oplus^2 S_{7h}^{3}$ and for 
	$\boldsymbol{y}_h \in \oplus^2 S_{h}^{3}$.}
	\label{fig:example-1-eoc}
\end{figure}

\subsection{Example 2: parameterized solution}
\rm
Next, we discuss an example with the parameterized exact solution. Let $Q = (0, 1)^2$ be a unit square, 
and let the exact solution, the RHS, and the Dirichlet boundary condition be chosen as follows:
\begin{alignat*}{4}
u(x, t) 	& = \sin k_1\,\pi\,x\, \sin k_2\,\pi\,t  			
\quad && (x, t) && \in \overline{Q} = [0, 1]^2, \\
f(x, t) 		& = \sin k_1\,\pi\,x\,
\Big( k_2\,\pi\,\cos k_2 \, \pi \, t + k_1^2 \, \pi^2 \, \sin k_2 \, \pi \, t\Big) 		
\quad && (x, t) && \in {Q} = (0, 1)^2, \\
u_0(x, t) 	& = 0,	     								
\quad && (x, t)  && \in 
{\overline{\Sigma}}_0, \\
u_D(x, t) 	& = 0,							        		
\quad && (x, t) && \in \Sigma := \partial \Omega \times (0, 1).
\end{alignat*}
%
\begin{figure}[!t]
	\centering
	\quad
	\subfloat[$k_1 = k_2 = 1$]{
	\includegraphics[scale=0.6]{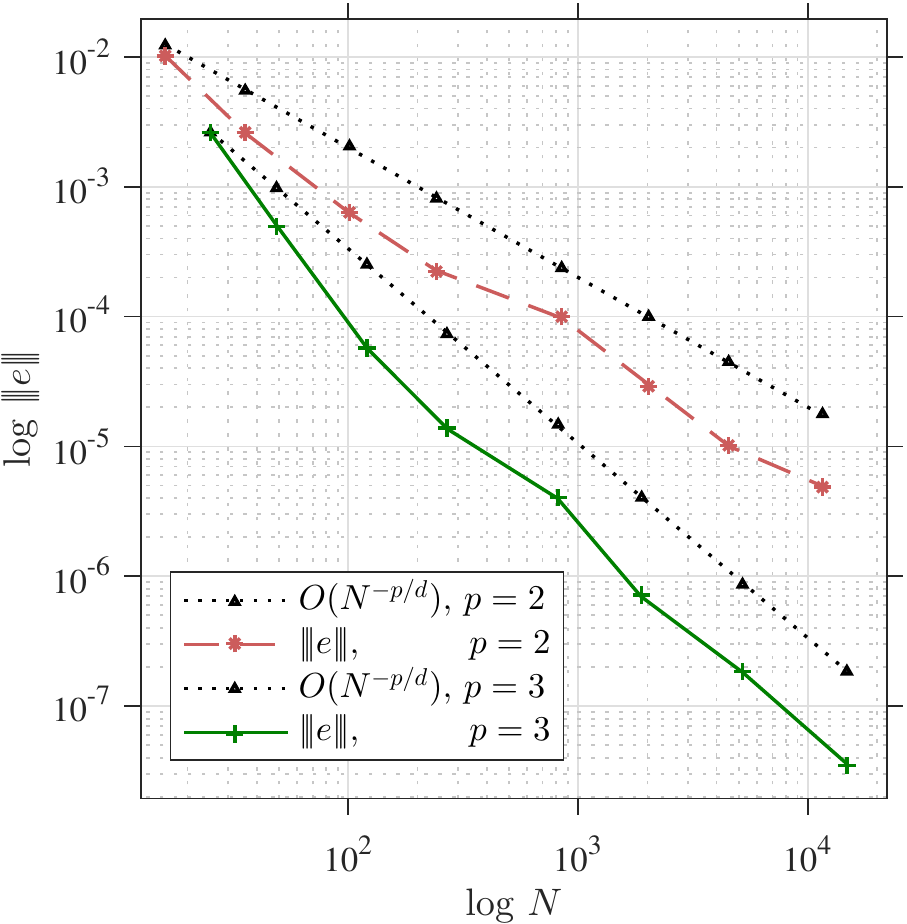}
	\label{fig:example-3-eoc-a}}
	\subfloat[$k_1 = 6, k_2 = 3$]{
	\includegraphics[scale=0.6]{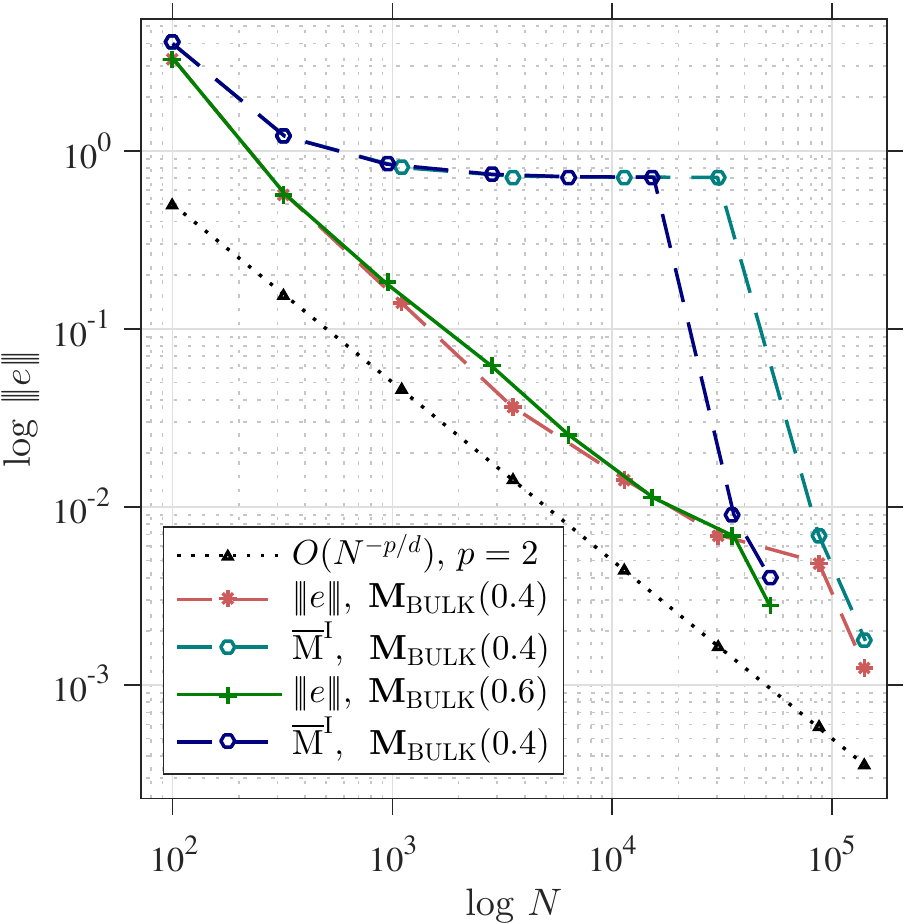}
	\label{fig:example-4-eoc-b}}
	\caption{{\em Example 2}. 
	The error order of convergence for (a) $k_1 = k_2 = 1$ and (b) $k_1 = 6, \; k_2 = 3$.}
	\label{fig:example-2-eoc}
\end{figure}
%
First, we set $k_1 = k_2 = 1$ ({\em Example 2-1}). 
We consider eight adaptive refinement steps ($N_{\rm ref} = 8$) preceded by three global refinements 
($N_{\rm ref, 0} = 3$) to generate the initial mesh $\mathcal{K}^0_{h}$. For the marking criterion, 
we chose bulk parameter 
$\sigma = 0.4$. The function $u_h$ is approximated both by $S_h^{2}$ (case (a)) and $S_h^{3}$ 
(case (b)) spaces, whereas corresponding auxiliary functions by 
$\boldsymbol{y}_h \in \oplus^2 S^{4}_{7h}$ and $w_h \in S^{4}_{7h}$ as well as  
$\boldsymbol{y}_h \in \oplus^2 S^{6}_{5h}$ and $w_h \in S^{6}_{5h}$, respectively, see 
Tables 
\ref{tab:unit-domain-example-3-error-majorant-v-2-y-5-adapt-ref}--\ref{tab:unit-domain-example-3-time-expenses-v-2-y-5-adapt-ref}. 
Figure \ref{fig:example-3-eoc-a} illustrates different error orders of convergence for different 
approximations $u_h \in S_h^{2}$ and $u_h \in S_h^{3}$, which perform slightly better than expected
rates $O(h^2)$ and $O(h^3)$, respectively.

We also demonstrate the quantitative effectiveness of the error indication provided 
by $\overline{\rm M}^{\rm I}$. 
In Figure \ref{fig:unit-domain-example-3-mesh-ref}, the comparison of the meshes 
illustrates that the refinement based on local values of $\| \nabla_x e \|_K$ (first row) and the indicator 
$\overline{\rm m}^{\rm I}_{\rm d, K}$ (second row) provide similar adaptive patterns. 

\begin{table}[!t]
\scriptsize
\centering
\newcolumntype{g}{>{\columncolor{gainsboro}}c} 	
\newcolumntype{k}{>{\columncolor{lightgray}}c} 	
\newcolumntype{s}{>{\columncolor{silver}}c} 
\newcolumntype{a}{>{\columncolor{ashgrey}}c}
\newcolumntype{b}{>{\columncolor{battleshipgrey}}c}
\begin{tabular}{c|cgg|c|cg|gc}
\parbox[c]{0.4cm}{\centering \# ref. } & 
\parbox[c]{1.2cm}{\centering  $\| \nabla_x e \|_Q$}   & 	  
\parbox[c]{1.0cm}{\centering $I_{\rm eff} (\overline{\rm M}^{\rm I})$ } & 
\parbox[c]{1.0cm}{\centering $I_{\rm eff} (\overline{\rm M}^{\rm I\!I})$ } & 
\parbox[c]{1.2cm}{\centering  $|\!|\!|  e |\!|\!|_{l\!o\!c\!,h}$ }   & 	  
\parbox[c]{1.2cm}{\centering  $|\!|\!|  e |\!|\!|_{\mathcal{L}}$ }   & 	  
\parbox[c]{1.0cm}{\centering$I_{\rm eff} ({{\rm E \!\!\! Id}})$ } & 
\parbox[c]{1.2cm}{\centering e.o.c. ($|\!|\!|  e |\!|\!|_{l\!o\!c\!,h}$)} & 
\parbox[c]{1.2cm}{\centering e.o.c. ($|\!|\!|  e |\!|\!|_{\mathcal{L}}$)} \\[3pt]
\bottomrule
\multicolumn{9}{l}{ \rule{0pt}{3ex}   
(a) $u_h \in S^{2}_{h}$, 
$\boldsymbol{y}_h \in \oplus^2 S^{4}_{7h}$, and 
$w_h \in S^{4}_{7h}$} \\[3pt]
\toprule
  2 &     2.9034e-03 &         1.94 &         1.17 &     3.0649e-03 & 
  2.9197e-01 &         1.00 &     2.38 &     1.40 \\
   4 &     3.3878e-04 &         3.14 &         1.33 &     3.5057e-04 & 
   9.3154e-02 &         1.00 &     1.96 &     1.07 \\
   6 &     4.8136e-05 &         4.13 &         1.70 &     4.8588e-05 &  
   3.7361e-02 &         1.00 &     2.36 &     1.31 \\
   8 &     9.2649e-06 &         5.78 &         3.23 &     9.2835e-06 &  
   1.7351e-02 &         1.00 &     3.79 &     1.79 \\
\bottomrule
\multicolumn{9}{l}{ \rule{0pt}{3ex}   
(b) $u_h \in S^{3}_{h}$, 
$\boldsymbol{y}_h \in \oplus^2 S^{6}_{5h}$, and 
$w_h \in S^{6}_{5h}$} \\[3pt]
\toprule
  2 &     4.9924e-03 &         1.31 &         1.04 &     5.0700e-03 &
  1.1918e-01 &         1.00 &     5.08 &     4.18 \\
   4 &     1.3562e-04 &         1.64 &         1.30 &     1.3591e-04 &         
   8.9725e-03 &         1.00 &     3.56 &     2.89 \\
   6 &     6.9962e-06 &        10.61 &        10.26 &     6.9982e-06 &       
   1.4163e-03 &         1.00 &     4.17 &     2.55 \\
   8 &     3.5507e-07 &         3.44 &         1.24 &     3.5535e-07 &         
   1.6376e-04 &         1.00 &     3.11 &     2.13 \\
\end{tabular}
\caption{{\em Example 2-1}. 
Efficiency of $\overline{\rm M}^{\rm I}$, $\overline{\rm M}^{\rm I\!I}$, 
$\overline{\rm M}^{\rm I}_{h}$, ${{\rm E \!\!\! Id}}$, 
and { order of convergence} of $|\!|\!|  e |\!|\!|_{l\!o\!c\!,h}$ and $|\!|\!|  e |\!|\!|_{\mathcal{L}}$ 
for $\sigma =0.4$ ($N_{\rm ref, 0}$ = 3).}
\label{tab:unit-domain-example-3-error-majorant-v-2-y-5-adapt-ref}
\end{table}

\begin{table}[!t]
\scriptsize
\centering
\newcolumntype{g}{>{\columncolor{gainsboro}}c} 	
\begin{tabular}{c|ccc|cgg|cgg|c}
& \multicolumn{3}{c|}{ d.o.f. } 
& \multicolumn{3}{c|}{ $t_{\rm as}$ }
& \multicolumn{3}{c|}{ $t_{\rm sol}$ } 
& $\tfrac{t_{\rm appr.}}{t_{\rm er.est.}}$ \\
\midrule
\parbox[c]{0.2cm}{\# ref. } & 
\parbox[c]{0.4cm}{\centering $u_h$ } &  
\parbox[c]{0.4cm}{\centering $\boldsymbol{y}_h$ } &  
\parbox[c]{0.4cm}{\centering $w_h$ } & 
\parbox[c]{0.8cm}{\centering $u_h$ } & 
\parbox[c]{0.8cm}{\centering $\boldsymbol{y}_h$ } & 
\parbox[c]{0.8cm}{\centering $w_h$ } & 
\parbox[c]{0.8cm}{\centering $u_h$ } & 
\parbox[c]{0.8cm}{\centering $\boldsymbol{y}_h$ } & 
\parbox[c]{0.8cm}{\centering $w_h$ } &
\\
\bottomrule
\multicolumn{10}{l}{ \rule{0pt}{3ex}   
(a) $u_h \in S^{2}_{h}$, 
$\boldsymbol{y}_h \in \oplus^2 S^{4}_{7h}$, and 
$w_h \in S^{4}_{7h}$}\\[3pt]
\toprule
   5 &       5695 &        288 &        144 &   7.89e+00 &   4.78e-01 &   3.84e-01 &         5.34e-01 &         2.36e-03 &         2.72e-03 &             17.53 \\
   6 &      12935 &        288 &        144 &   1.55e+01 &   3.97e-01 &   3.83e-01 &         2.17e+00 &         2.30e-03 &         1.37e-03 &            44.25 \\
   7 &      34037 &        288 &        144 &   4.90e+01 &   3.98e-01 &   3.73e-01 &         9.58e+00 &         3.36e-03 &         1.42e-03 &            145.95 \\
   8 &      61258 &        288 &        144 &   9.37e+01 &   3.80e-01 &   3.62e-01 &         2.42e+01 &         2.10e-03 &         1.83e-03 &           308.55 \\
    \midrule
    &       &         &    &
    \multicolumn{3}{c|}{ $t_{\rm as} (u_h)$ \;:\; $t_{\rm as} (\boldsymbol{y}_h)$ \;:\; $t_{\rm as} (w_h)$ } &      
    \multicolumn{3}{c|}{\; $t_{\rm sol} (u_h)$ \;:\; $t_{\rm sol} (\boldsymbol{y}_h)$  \;:\;  $t_{\rm sol} (w_h)$\;} & \\
 \midrule
	 & 	 & 	 & 	 &     258.63 &       1.05 &       1.00 &         13252.51 &             1.15 &             1.00 &           \\
\bottomrule
 \multicolumn{10}{l}{ \rule{0pt}{3ex}   
(b) $u_h \in S^{3}_{h}$, 
$\boldsymbol{y}_h \in \oplus^2 S^{6}_{5h}$, and 
$w_h \in S^{6}_{5h}$} \\[3pt]
\toprule
   5 &       6425 &        338 &        169 &   8.26e+00 &   6.93e-01 &   6.97e-01 &         7.12e-01 &         5.63e-03 &         3.73e-03 &            12.84 \\
   6 &      13742 &        338 &        169 &   1.62e+01 &   7.03e-01 &   7.03e-01 &         2.11e+00 &         2.53e-03 &         1.43e-03 &            25.95 \\
   7 &      35091 &        644 &        322 &   5.36e+01 &   5.65e+00 &   5.52e+00 &         1.10e+01 &         9.31e-03 &         5.29e-03 &          11.41 \\
   8 &      78561 &        744 &        372 &   1.91e+02 &   5.61e+00 &   5.03e+00 &         2.40e+01 &         2.51e-02 &         7.56e-03 &          38.15 \\
    \midrule
    &       &         &    &
    \multicolumn{3}{c|}{ $t_{\rm as} (u_h)$ \;:\; $t_{\rm as} (\boldsymbol{y}_h)$ \;:\; $t_{\rm as} (w_h)$ } &      
    \multicolumn{3}{c|}{\; $t_{\rm sol} (u_h)$ \;:\; $t_{\rm sol} (\boldsymbol{y}_h)$  \;:\;  $t_{\rm sol} (w_h)$\;} & \\
 \midrule
 	 & 	 & 	 & 	 &      37.97 &       1.11 &       1.00 &          3168.34 &             3.31 &             1.00 &            \\
\end{tabular}
\caption{{\em Example 2-1}. 
Assembling and solving time (in seconds) spent for the systems generating 
d.o.f. of $u_h$, $\boldsymbol{y}_h$, and $w_h$ for 
$\sigma =0.4$ ($N_{\rm ref, 0}$ = 3).}
\label{tab:unit-domain-example-3-time-expenses-v-2-y-5-adapt-ref}
\scriptsize
\end{table}

\begin{figure}[!t]
	\centering
	\subfloat[ref. 5]{
	 \LMRspacetimeaxis{
	\includegraphics[width=3.7cm, trim={3.7cm 3.7cm 3.7cm 3.7cm}, clip]{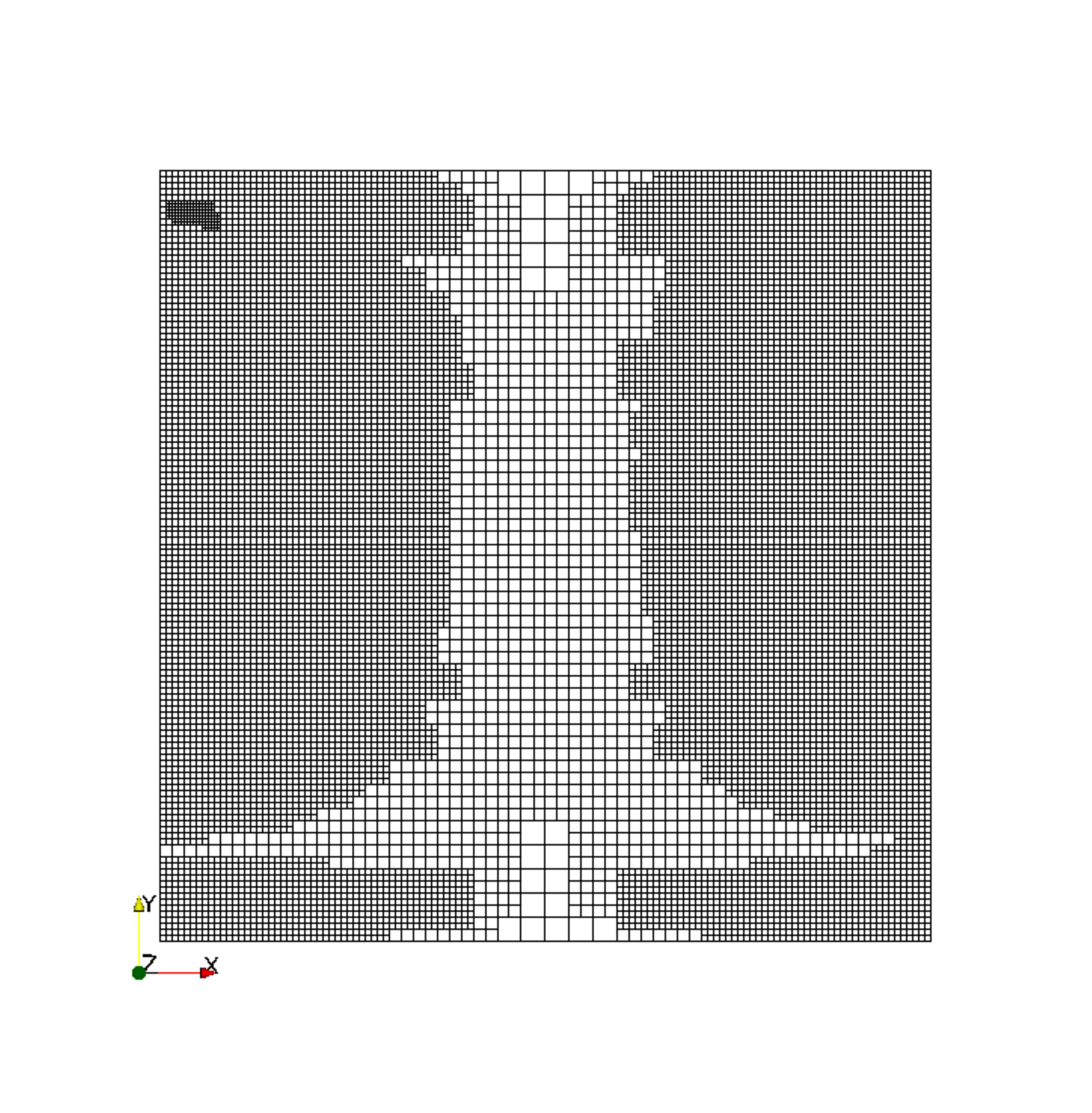}
	}
	} 
	\!\!\!\!
	\subfloat[ref. 6]{
	 \LMRspacetimeaxis{
	{
	\includegraphics[width=3.7cm, trim={3.7cm 3.7cm 3.7cm 3.7cm}, clip]{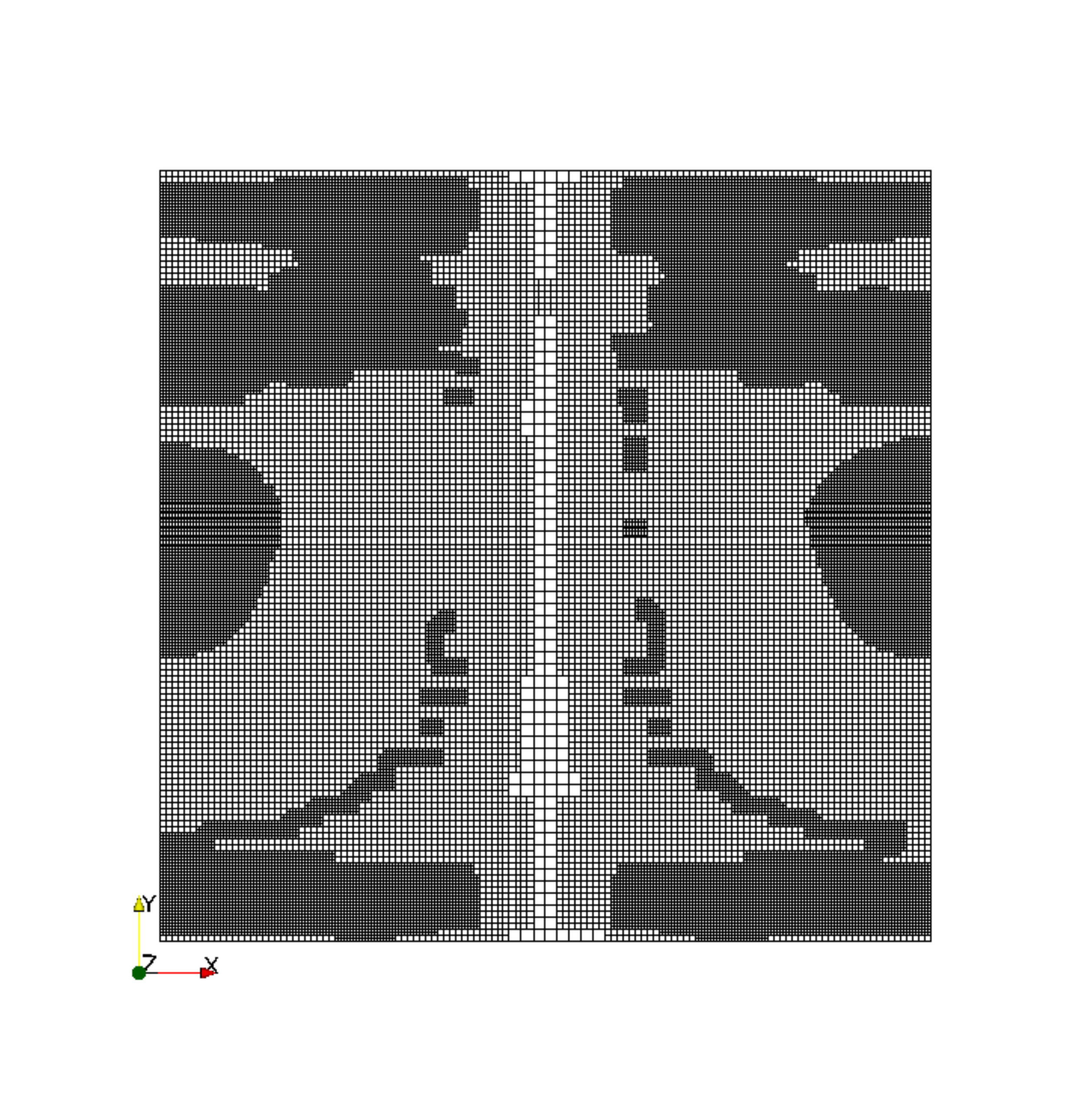}}}
	}
	\!\!\!\!
	\subfloat[ref. 7]{
	 \LMRspacetimeaxis{
	\includegraphics[width=3.7cm, trim={3.7cm 3.7cm 3.7cm 3.7cm}, clip]{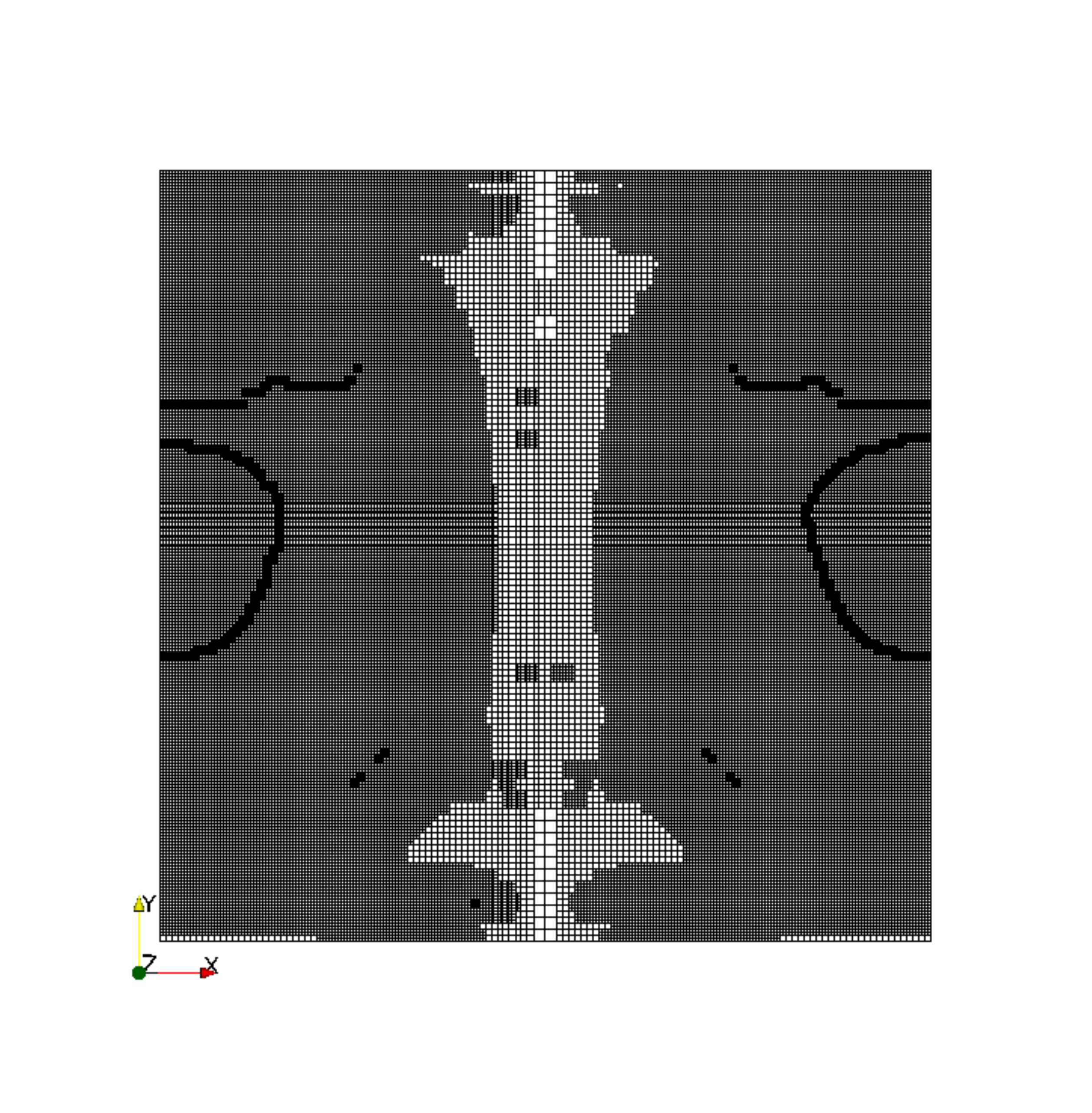}}
	} \\[-7pt]
	\subfloat[ref. 5]{
	 \LMRspacetimeaxis{
	\includegraphics[width=3.7cm, trim={5cm 3.2cm 5cm 3.2cm}, clip]{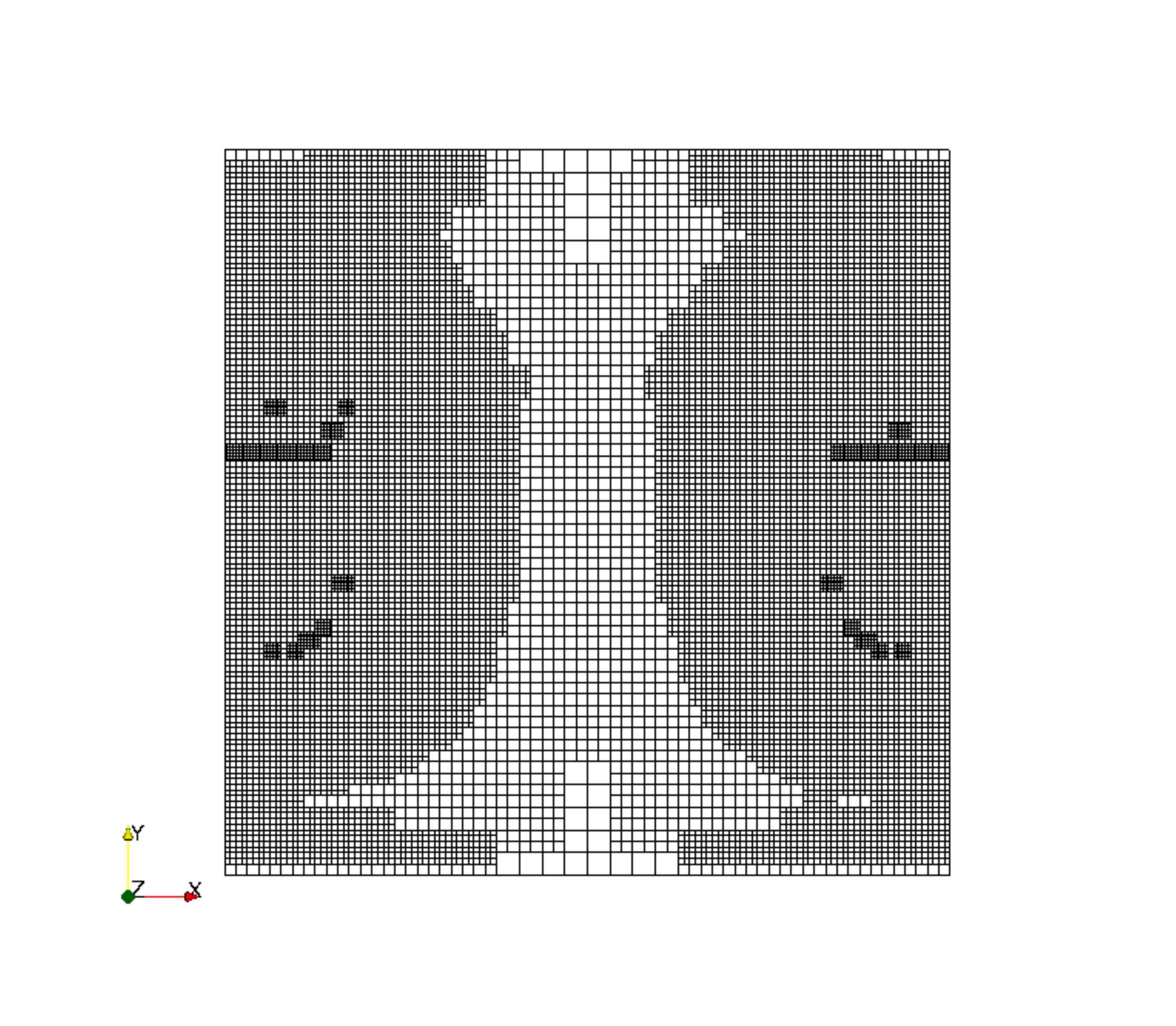}
	}
	} 
	\!\!\!\!
	\subfloat[ref. 6]{
	 \LMRspacetimeaxis{
	{
	\includegraphics[width=3.7cm, trim={5cm 3.2cm 5cm 3.2cm}, clip]{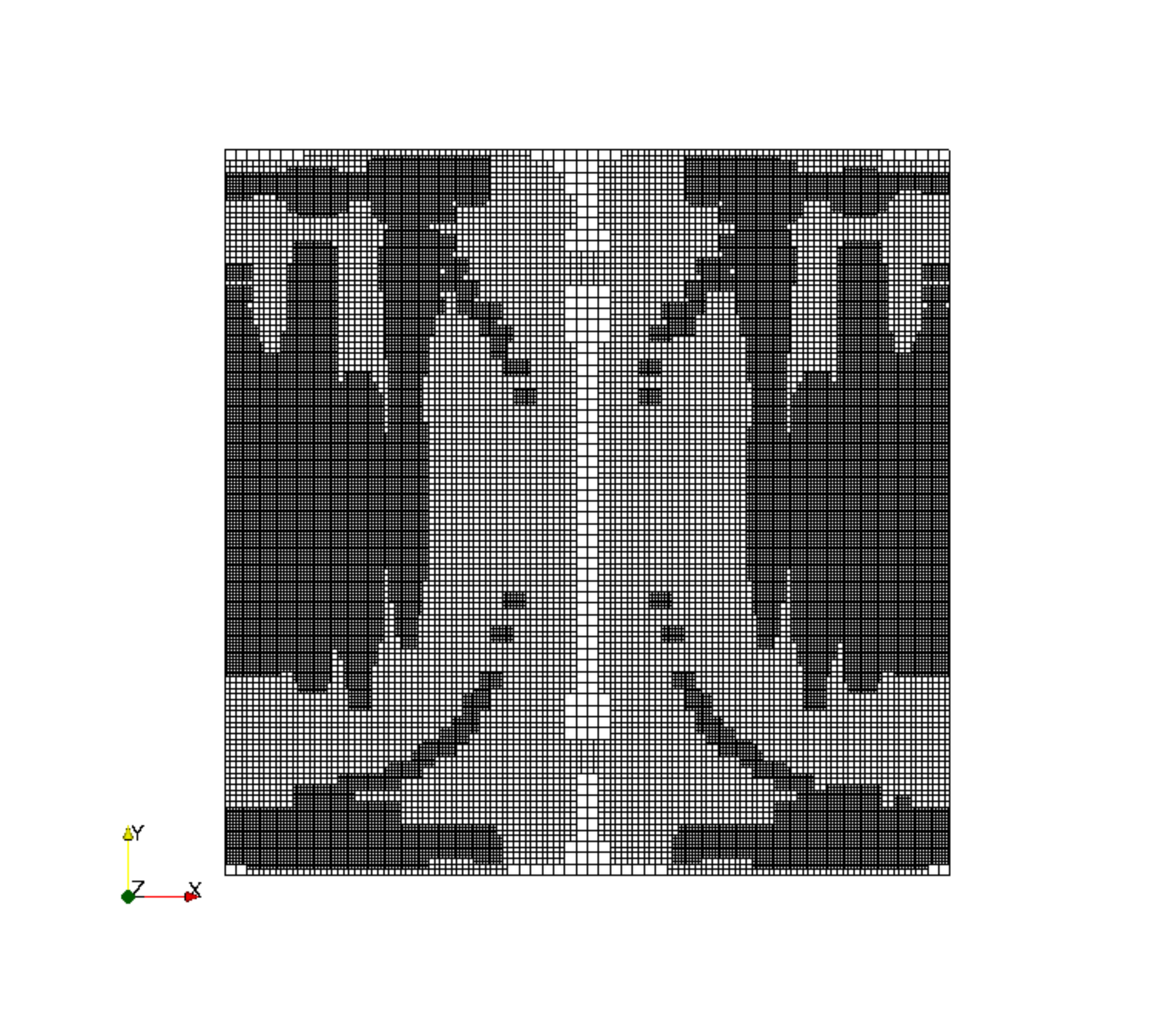}}}
	}
	\!\!\!\!
	\subfloat[ref. 7]{
	 \LMRspacetimeaxis{
	\includegraphics[width=3.7cm, trim={5cm 3.2cm 5cm 3.2cm}, clip]{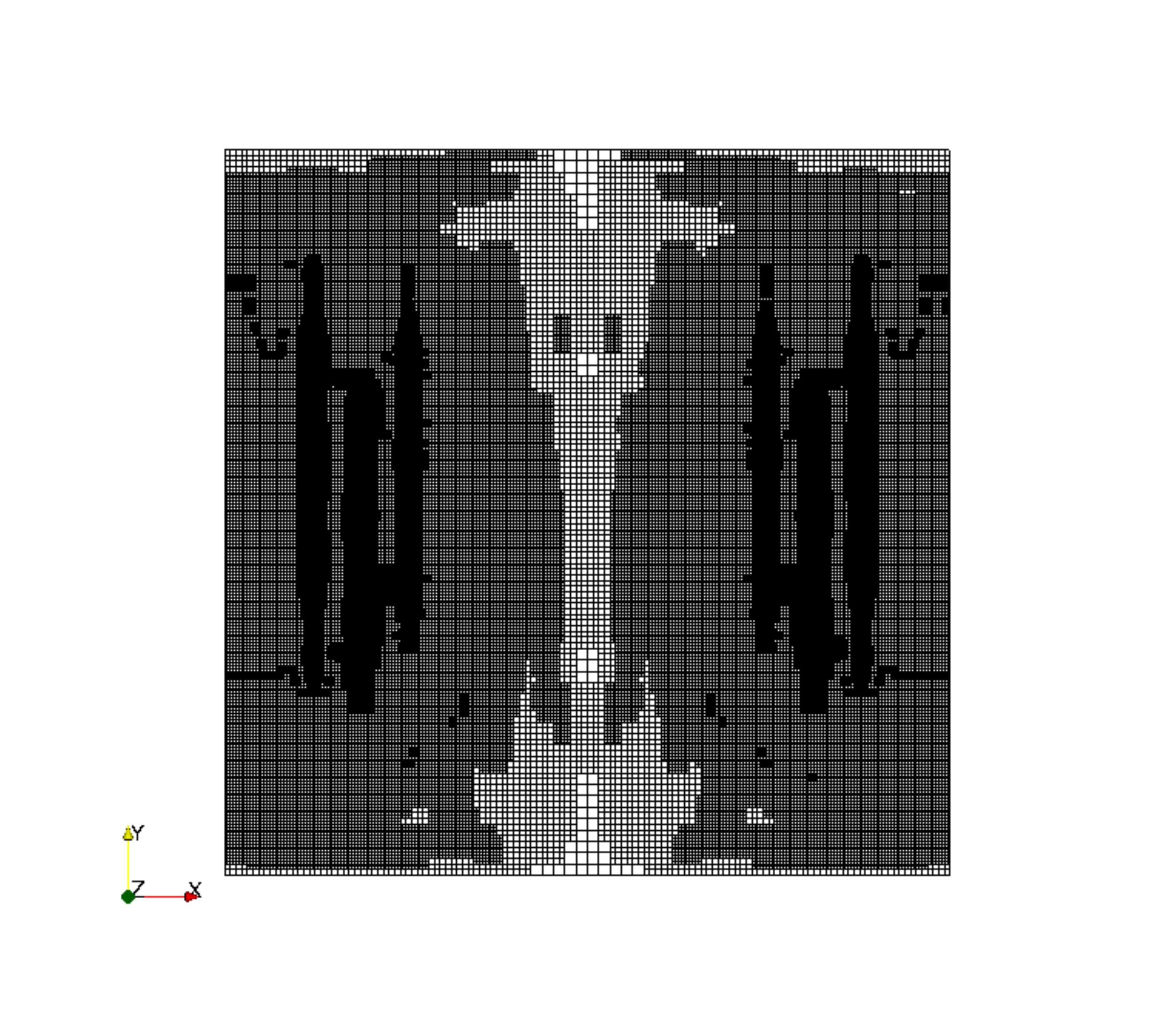}}
	} 
	\caption{\textit{Example 2-1}. Meshes obtained in the adaptive procedure 
	based on indicator $\overline{\rm m}^{\rm I}_{\rm d, K}$ (top row) and 
	based on the exact error $\| \nabla_x e \|_K$ (bottom row) w.r.t. refinement steps 5--7.}
	\label{fig:unit-domain-example-3-mesh-ref}
\end{figure}

Next, we set parameters $k_1 = 3$ and $k_2 = 6$. In this case, auxiliary variables are approximated by
$\boldsymbol{y}_h \in \oplus^2 S_{5h}^{7}$ and $w_h \in S^{7}_{5h}$. Figure 
\ref{fig:example-4-eoc-b} illustrates the order of convergence of errors and corresponding majorants 
(for two different marking strategy $\mathds{M}_{\rm BULK}(0.4)$ and $\mathds{M}_{\rm BULK}(0.6)$) 
and compares these results to the theoretical one $O(h^2)$. It is easy to see from the plot that efficiency 
of the majorant deteriorates on the first refinement steps, but it improves drastically on the last refinements. 
Tables \ref{tab:unit-domain-example-4-error-majorant-v-2-y-7-adapt-ref} and
\ref{tab:unit-domain-example-4-time-expenses-v-2-y-7-adapt-ref} compare numerical results obtained 
for different marking parameters $\sigma = 0.4$ (part (a)) and $\sigma = 0.6$ (part (b)). 
Finally, Figure \ref{fig:unit-domain-example-4-mesh-ref} demonstrates 
{ the evolution of meshes associated
with the refinement steps} 4--6 for the same cases. 

\begin{table}[!t]
\scriptsize
\centering
\newcolumntype{g}{>{\columncolor{gainsboro}}c} 	
\newcolumntype{k}{>{\columncolor{lightgray}}c} 	
\newcolumntype{s}{>{\columncolor{silver}}c} 
\newcolumntype{a}{>{\columncolor{ashgrey}}c}
\newcolumntype{b}{>{\columncolor{battleshipgrey}}c}
\begin{tabular}{c|cgg|c|cg|gc}
\parbox[c]{0.4cm}{\centering \# ref. } & 
\parbox[c]{1.2cm}{\centering  $\| \nabla_x e \|_Q$}   & 	  
\parbox[c]{1.0cm}{\centering $I_{\rm eff} (\overline{\rm M}^{\rm I})$ } & 
\parbox[c]{1.0cm}{\centering $I_{\rm eff} (\overline{\rm M}^{\rm I\!I})$ } & 
\parbox[c]{1.2cm}{\centering  $|\!|\!|  e |\!|\!|_{l\!o\!c\!,h}$ }   & 	  
\parbox[c]{1.2cm}{\centering  $|\!|\!|  e |\!|\!|_{\mathcal{L}}$ }   & 	  
\parbox[c]{1.0cm}{\centering$I_{\rm eff} ({{\rm E \!\!\! Id}})$ } & 
\parbox[c]{1.2cm}{\centering e.o.c. ($|\!|\!|  e |\!|\!|_{l\!o\!c\!,h}$)} & 
\parbox[c]{1.2cm}{\centering e.o.c. ($|\!|\!|  e |\!|\!|_{\mathcal{L}}$)} \\[3pt]
\bottomrule
\multicolumn{9}{l}{ \rule{0pt}{3ex}   
(a) $\sigma =0.4$} \\[3pt]
\toprule
   2 &     5.7161e-01 &         2.11 &         1.38 &     5.7163e-01 & 
   6.2371e+01 &         1.00 &     2.99 &     1.19 \\
   3 &     1.3927e-01 &         5.77 &         2.20 &     1.3928e-01 &  
   3.1026e+01 &         1.00 &     2.30 &     1.14 \\
   7 &     4.8735e-03 &         1.43 &         1.15 &     4.8736e-03 & 
   4.7350e+00 &         1.00 &     0.68 &     0.48 \\
   8 &     1.2298e-03 &         1.44 &         1.16 &     1.2298e-03 &  
   2.6917e+00 &         1.00 &     5.60 &     2.30 \\
\bottomrule
\multicolumn{9}{l}{ \rule{0pt}{3ex}   
(b) $\sigma =0.6$} \\[3pt]
\toprule
    2 &     5.7161e-01 &         2.11 &         1.38 &     5.7163e-01 &  
    6.2371e+01 &         1.00 &     2.99 &     1.19 \\
   3 &     1.7942e-01 &         4.69 &         1.96 &     1.7945e-01 &   
   3.2971e+01 &         1.00 &     2.18 &     1.20 \\
   7 &     6.8374e-03 &         1.32 &         1.12 &     6.8374e-03 &  
   5.8760e+00 &         1.00 &     1.18 &     0.72 \\
   8 &     2.7492e-03 &         1.44 &         1.15 &     2.7492e-03 &   
   4.0721e+00 &         1.00 &     4.75 &     1.91 \\
\end{tabular}
\caption{{\em Example 2-2}. 
Efficiency of $\overline{\rm M}^{\rm I}$, $\overline{\rm M}^{\rm I\!I}$, ${{\rm E \!\!\! Id}}$, and 
the order of convergence of $|\!|\!|  e |\!|\!|_{l\!o\!c\!,h}$ and $|\!|\!|  e |\!|\!|_{\mathcal{L}}$
for $u_h \in S^{2}_{h}$, $\boldsymbol{y}_h \in \oplus^2 S^{7}_{5h}$, and 
$w_h \in S^{7}_{5h}$ ($N_{\rm ref, 0}$ = 3).}
\label{tab:unit-domain-example-4-error-majorant-v-2-y-7-adapt-ref}
\end{table}

\begin{table}[!t]
\scriptsize
\centering
\newcolumntype{g}{>{\columncolor{gainsboro}}c} 	
\begin{tabular}{c|ccc|cgg|cgg|c}
& \multicolumn{3}{c|}{ d.o.f. } 
& \multicolumn{3}{c|}{ $t_{\rm as}$ }
& \multicolumn{3}{c|}{ $t_{\rm sol}$ } 
& $\tfrac{t_{\rm appr.}}{t_{\rm er.est.}}$ \\
\midrule
\parbox[c]{0.2cm}{\# ref. } & 
\parbox[c]{0.2cm}{\centering $u_h$ } &  
\parbox[c]{0.2cm}{\centering $\boldsymbol{y}_h$ } &  
\parbox[c]{0.2cm}{\centering $w_h$ } & 
\parbox[c]{0.8cm}{\centering $u_h$ } & 
\parbox[c]{0.8cm}{\centering $\boldsymbol{y}_h$ } & 
\parbox[c]{0.8cm}{\centering $w_h$ } & 
\parbox[c]{0.8cm}{\centering $u_h$ } & 
\parbox[c]{0.8cm}{\centering $\boldsymbol{y}_h$ } & 
\parbox[c]{0.8cm}{\centering $w_h$ } &
\\
\bottomrule
\multicolumn{10}{l}{ \rule{0pt}{3ex}   
(a) $\sigma = 0.4$} \\[2pt]
\toprule
%
   5 &      11426 &        450 &        225 &   1.50e+01 &   2.28e+00 &   2.92e+00 &         1.20e+00 &         7.89e-03 &         3.37e-03 &             3.11 \\
   6 &      30101 &        450 &        225 &   5.99e+01 &   2.29e+00 &   2.92e+00 &         3.57e+00 &         8.52e-03 &         4.33e-03 &            12.14 \\
   7 &      86849 &       1058 &        529 &   3.57e+02 &   9.30e+00 &   9.41e+00 &         1.11e+01 &         5.19e-02 &         3.47e-02 &            19.58 \\
   8 &     141987 &       2850 &       1425 &   6.36e+02 &   6.50e+01 &   5.91e+01 &         2.56e+01 &         3.00e-01 &         1.29e-01 &             5.31 \\
   \midrule
    &       &         &    &
    \multicolumn{3}{c|}{ $t_{\rm as} (u_h)$ \;:\; $t_{\rm as} (\boldsymbol{y}_h)$ \;:\; $t_{\rm as} (w_h)$ } &      
    \multicolumn{3}{c|}{\; $t_{\rm sol} (u_h)$ \;:\; $t_{\rm sol} (\boldsymbol{y}_h)$  \;:\;  $t_{\rm sol} (w_h)$\;} & \\
 \midrule
	 & 	 & 	 & 	 &      10.76 &       1.10 &       1.00 &           198.84 &             2.32 &             1.00 &             \\
\bottomrule
\multicolumn{10}{l}{ \rule{0pt}{3ex}   
(b) $\sigma = 0.6$} \\[2pt]
\toprule
   5 &       6320 &        450 &        225 &   9.30e+00 &   3.17e+00 &   2.57e+00 &         3.95e-01 &         9.80e-03 &         4.51e-03 &             3.04 \\
   6 &      15436 &        450 &        225 &   2.61e+01 &   2.36e+00 &   2.41e+00 &         1.77e+00 &         1.45e-02 &         3.12e-03 &             11.73 \\
   7 &      35745 &       1058 &        529 &   8.99e+01 &   9.86e+00 &   1.01e+01 &         4.68e+00 &         7.06e-02 &         4.12e-02 &             9.52 \\
   8 &      52453 &       2498 &       1249 &   1.05e+02 &   8.03e+01 &   7.08e+01 &         7.38e+00 &         3.47e-01 &         1.66e-01 &             1.39 \\
   \midrule
    &       &         &    &
    \multicolumn{3}{c|}{ $t_{\rm as} (u_h)$ \;:\; $t_{\rm as} (\boldsymbol{y}_h)$ \;:\; $t_{\rm as} (w_h)$ } &      
    \multicolumn{3}{c|}{\; $t_{\rm sol} (u_h)$ \;:\; $t_{\rm sol} (\boldsymbol{y}_h)$  \;:\;  $t_{\rm sol} (w_h)$\;} & \\
 \midrule
 	 & 	 & 	 & 	 &       1.49 &       1.13 &       1.00 &            44.46 &             2.09 &             1.00 &              \\
\end{tabular}
\caption{{\em Example 2-2}. 
Assembling and solving time (in seconds) spent for the systems generating d.o.f. of 
$u_h \in S^{2}_{h}$, $\boldsymbol{y}_h \in \oplus^2 S^{7}_{5h}$, and 
$w_h \in S^{7}_{5h}$ ($N_{\rm ref, 0}$ = 3).}
\label{tab:unit-domain-example-4-time-expenses-v-2-y-7-adapt-ref}
\end{table}

\begin{figure}[!t]
	\subfloat[ref. 4]{
	 \LMRspacetimeaxis{\includegraphics[width=3.7cm, trim={3.7cm 3.7cm 3.7cm 3.7cm}, clip]{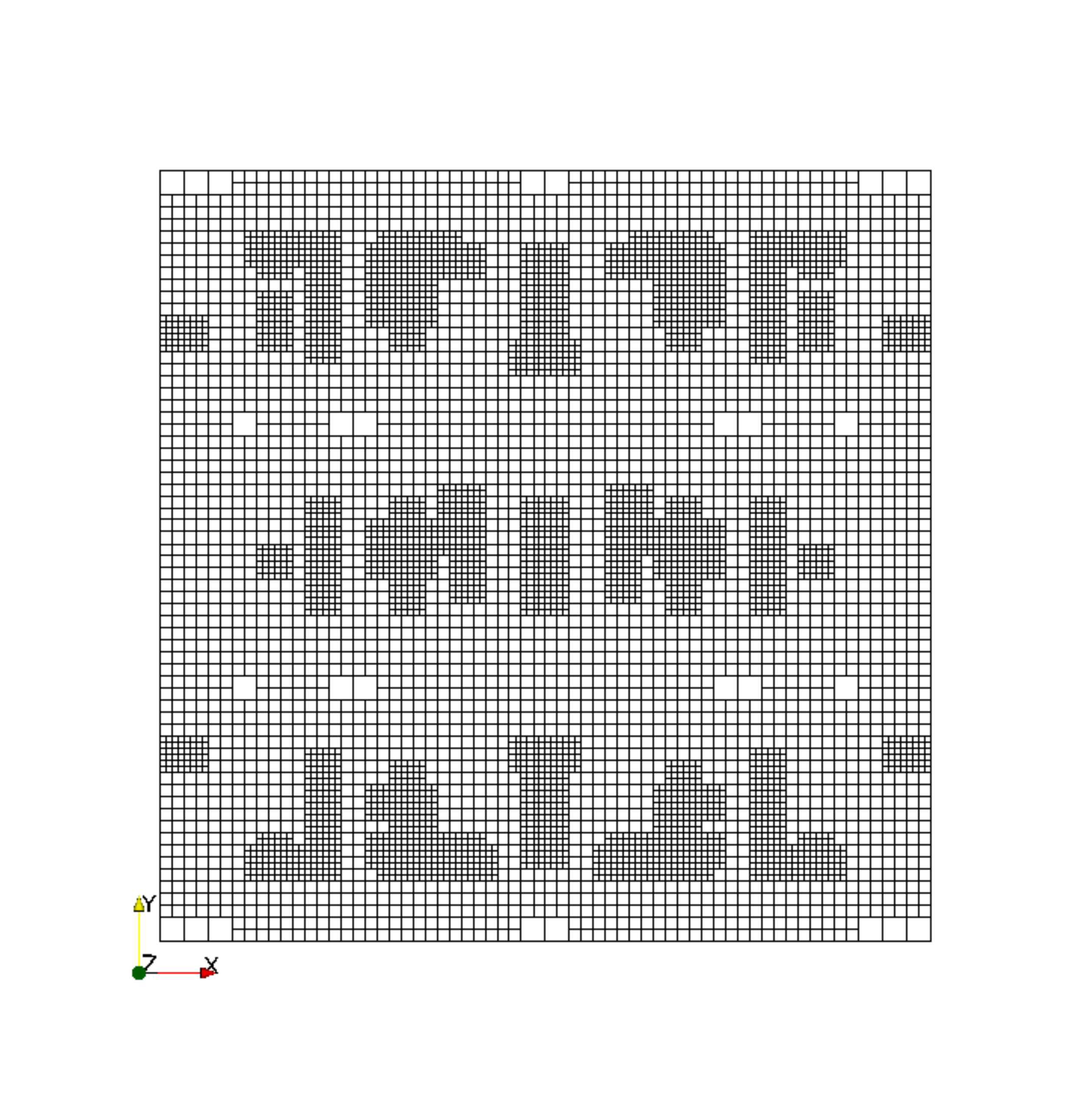}
	}
	}\!\!\!\!
	\subfloat[ref. 5]{
	 \LMRspacetimeaxis{\includegraphics[width=3.7cm, trim={3.7cm 3.7cm 3.7cm 3.7cm}, clip]{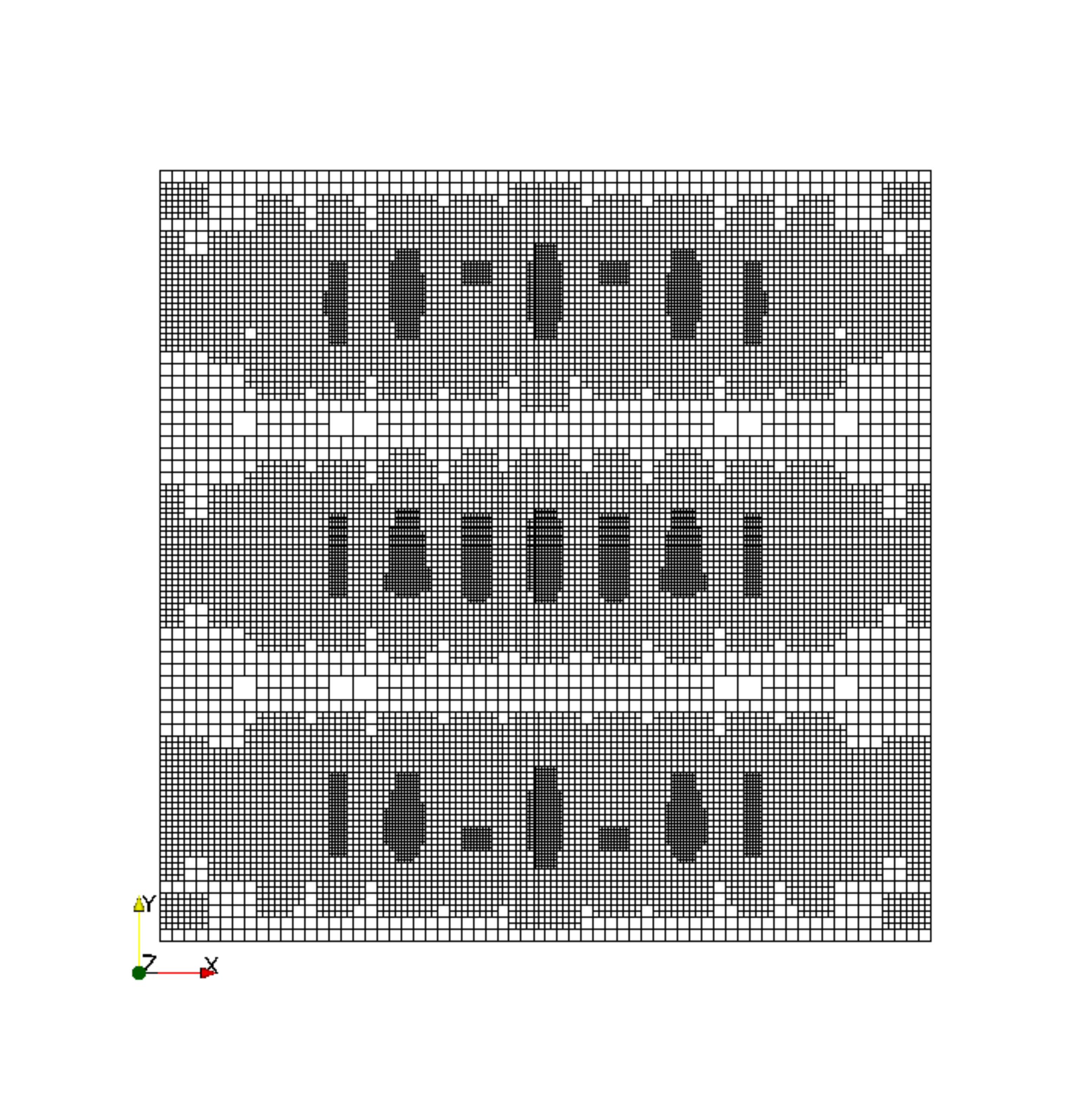}
	}
	}\!\!\!\!
	\subfloat[ref. 6]{
	 \LMRspacetimeaxis{\includegraphics[width=3.7cm, trim={3.7cm 3.7cm 3.7cm 3.7cm}, clip]{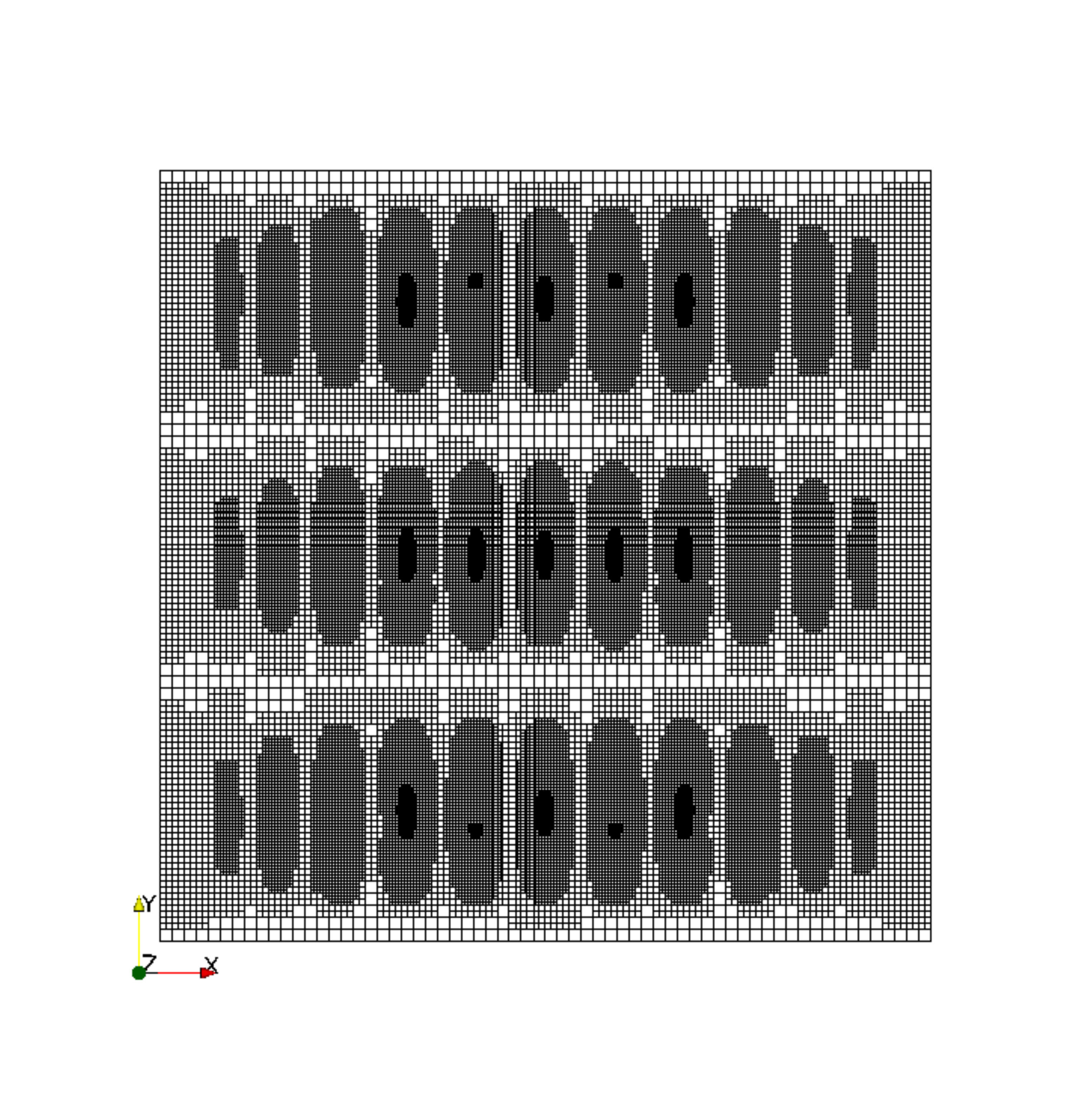}
	}
	}\!\!\!\!
	\\[-7pt]
	\subfloat[ref. 4]{
	 \LMRspacetimeaxis{\includegraphics[width=3.7cm, trim={3.7cm 3.7cm 3.7cm 3.7cm}, clip]{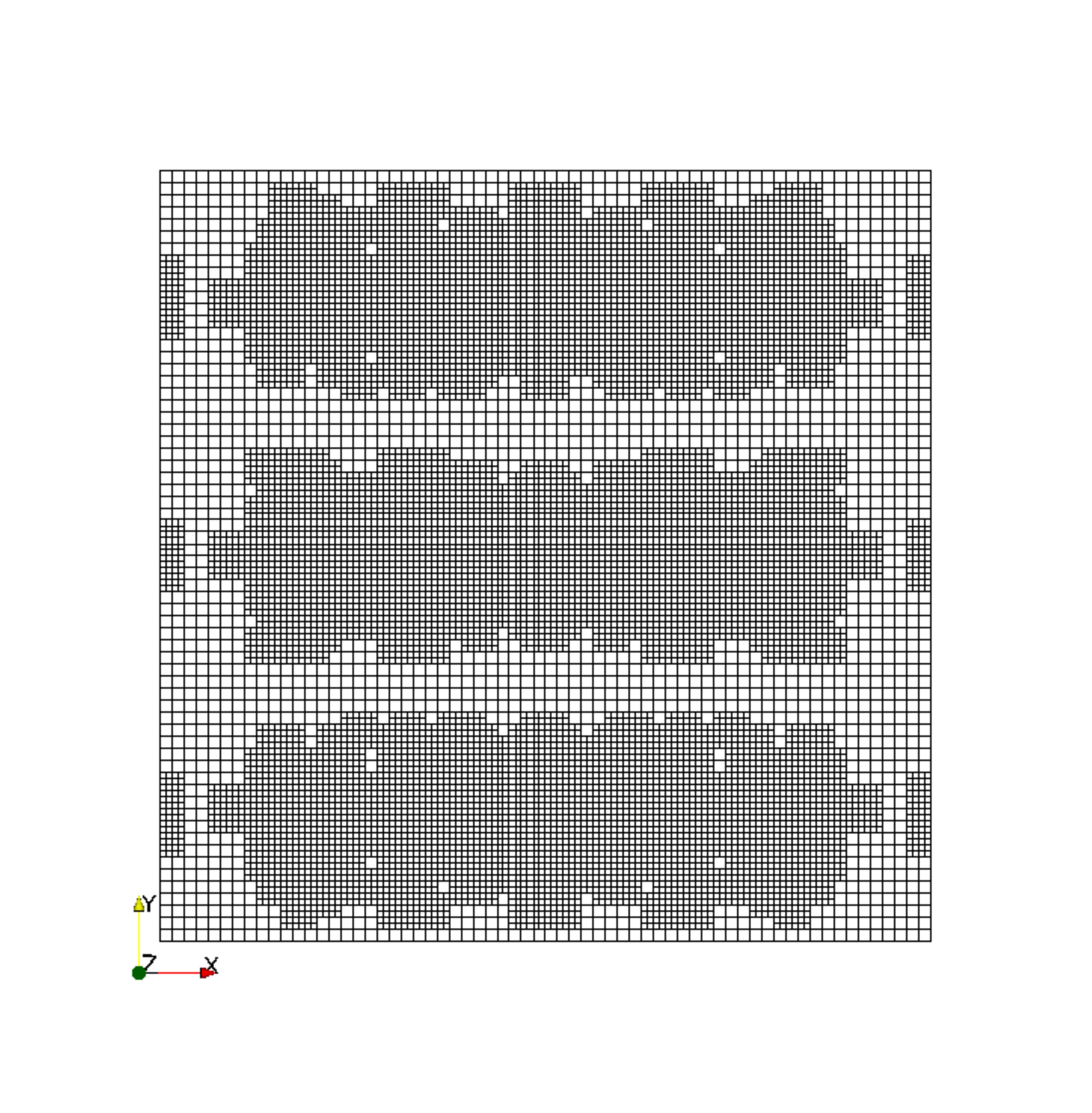}
	}
	}\!\!\!\!
	\subfloat[ref. 5]{
	 \LMRspacetimeaxis{\includegraphics[width=3.7cm, trim={3.7cm 3.7cm 3.7cm 3.7cm}, clip]{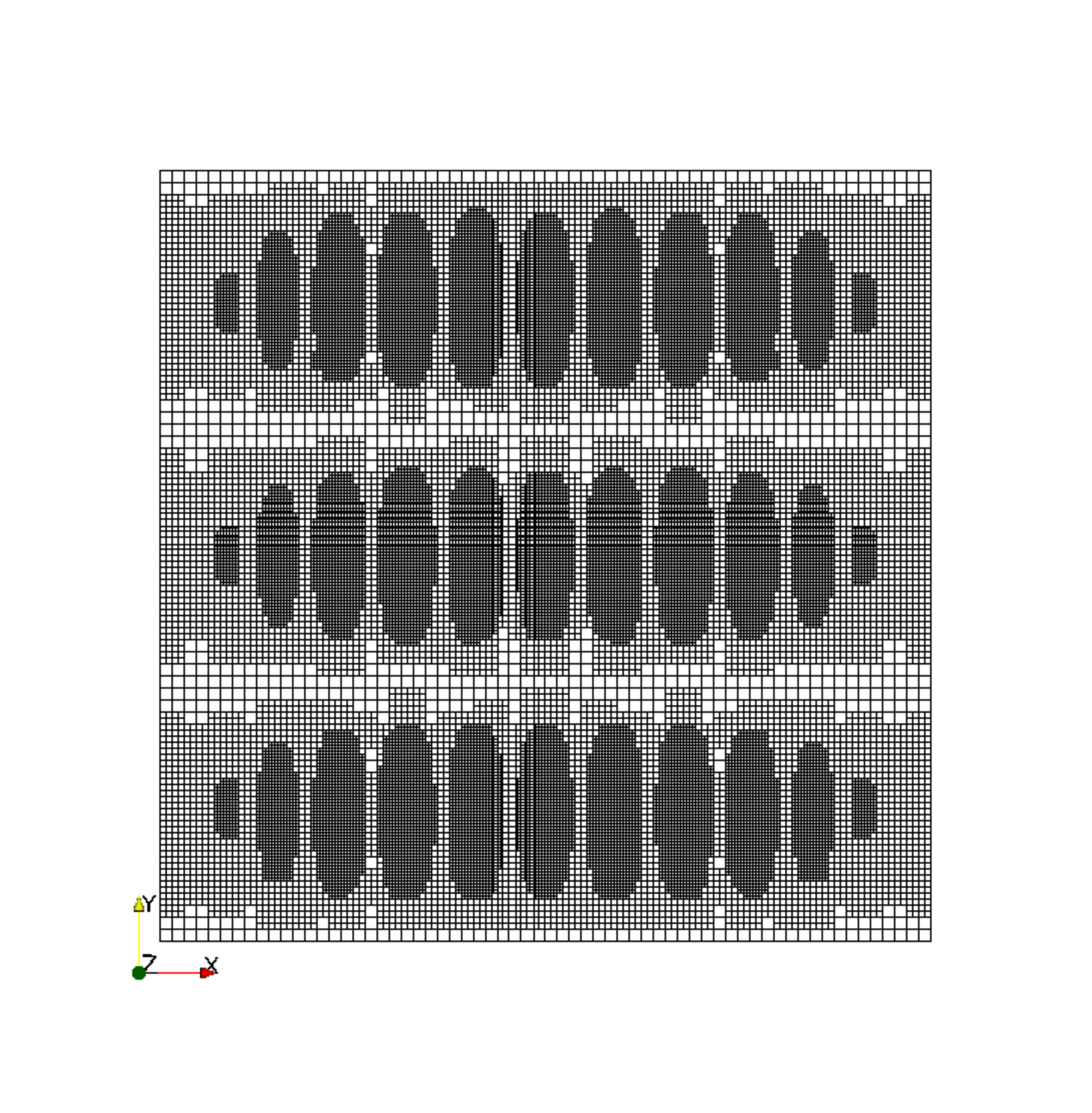}
	}
	}\!\!\!\!
	\subfloat[ref. 6]{
	 \LMRspacetimeaxis{\includegraphics[width=3.7cm, trim={3.7cm 3.7cm 3.7cm 3.7cm}, clip]{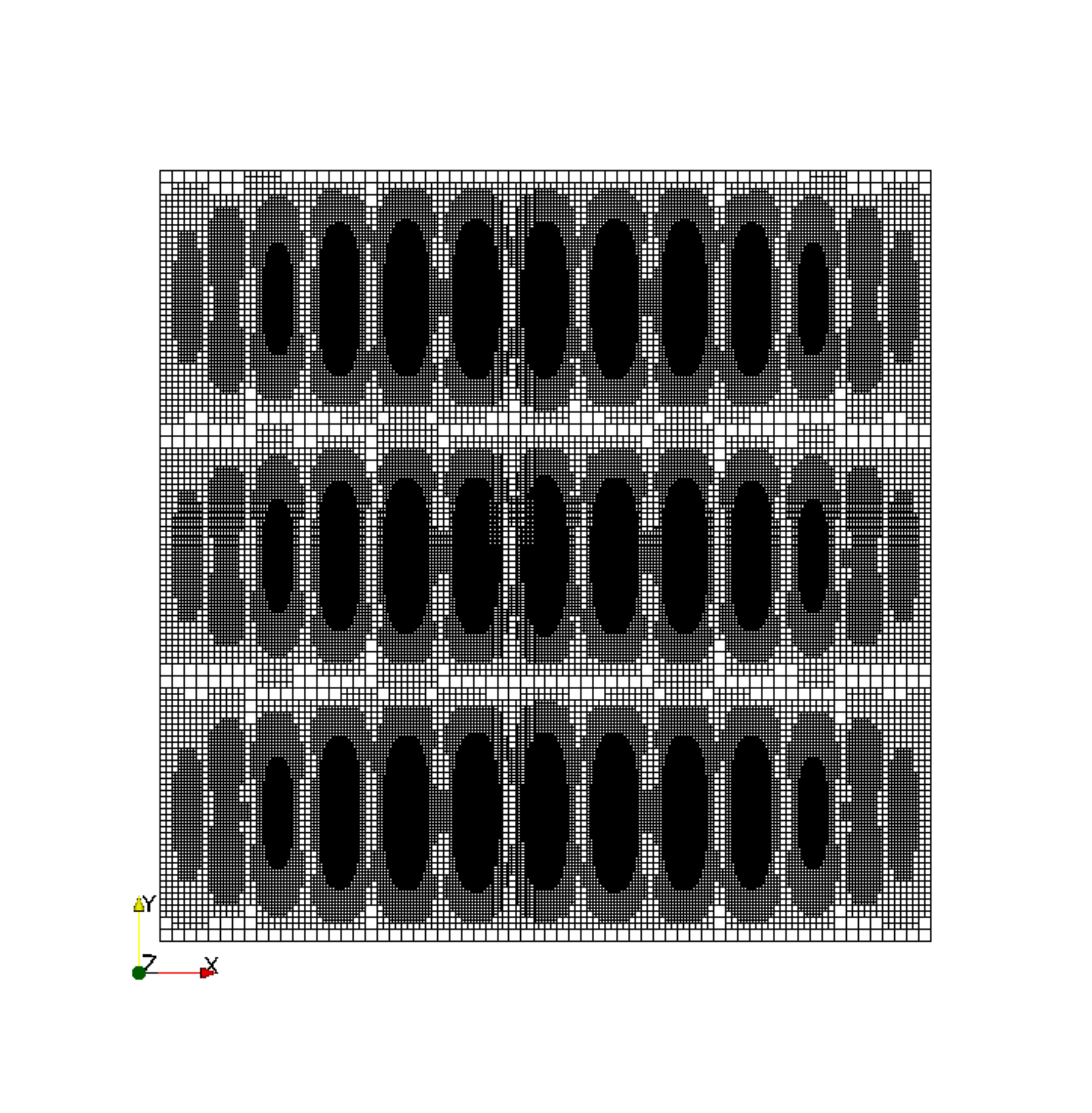}
	}
	}
	\caption{{\em Example 2-2}. 
	Meshes obtained by marking criteria $\mathds{M}_{\rm BULK}(0.6)$ (top row) and 
	$\mathds{M}_{\rm BULK}(\sigma =0.4)$ (bottom row) w.r.t. the refinement steps 4--6.}
\label{fig:unit-domain-example-4-mesh-ref}
\end{figure}

\subsection{Example 3: Gaussian distribution}

As the next test case, we consider the exact solution defined by a sharp local Gaussian distribution
%
$u(x, t) = (x^2 - x) \, (t^2 - t) \, e^{-100 \,|(x, t) - (0.8, 0.05)|}$,  
$(x, t) \in \overline{Q} := [0, 1]^2,$
%
where the peak is located in the point $(x, t) = (0.8, 0.05)$. Then $f$ is computed by substituting $u$ 
into \eqref{eq:equation}. The Dirichlet boundary condition is obviously homogeneous.

\begin{figure}[!t]
	\centering
	\includegraphics[scale=0.6]{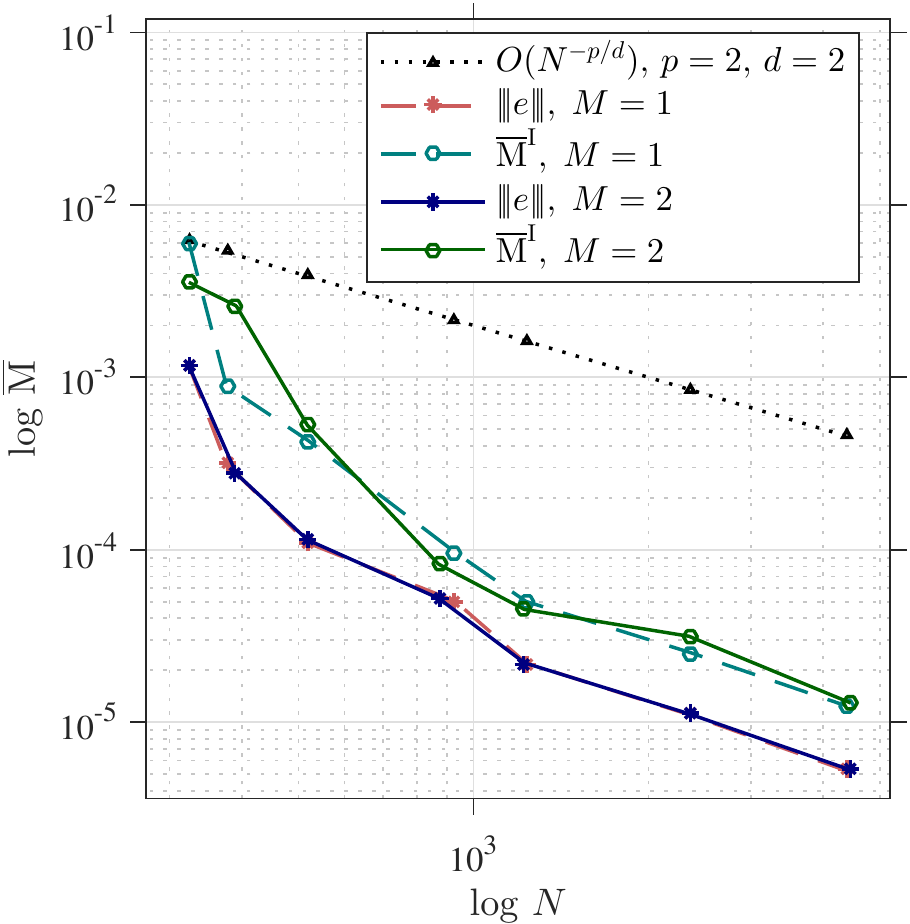}
	\caption{{\em Example 3}. The majorant and e.o.c. for $u_h \in S^{2}_{h}$ and 
	two different setting of auxiliary functions 
	(a) $\boldsymbol{y}_h \in \oplus^2 S^{3}_{h}$, and $w_h \in S^{3}_{h}$ and
	(b) $\boldsymbol{y}_h \in \oplus^2 S^{6}_{2h}$, and $w_h \in S^{6}_{2h}$.}
	\label{fig:example-6-eoc}
\end{figure}

For the discretisation spaces, we 
use the standard configuration, i.e., $u_h \in S^{2}_h$ for the approximate solution, as well as 
$\boldsymbol{y}_h \in \oplus^2 S^{3}_h$ and $w_h \in S^{3}_h$ for the auxiliary functions. 
We start with four initial global refinements ($N_{\rm ref, 0} = 4$), and continue with 
seven adaptive steps ($N_{\rm ref} = 7$). As the marking criteria, we choose 
$\mathds{M}_{\rm BULK}(0.6)$. The error 
order of convergence is illustrated in Figure \ref{fig:example-6-eoc}. It confirms that 
majorants reconstructed with \linebreak $\boldsymbol{y}_h \in \oplus^2 S^{6}_{2h}$, and $w_h \in S^{6}_{2h}$ 
are as efficient as the one reconstructed with $\boldsymbol{y}_h \in \oplus^2 S^{3}_{h}$, and 
$w_h \in S^{3}_{h}$. They also drastically improve the convergence order on the first refinement steps.

{
Numbers exposed in Table \ref{tab:unit-domain-example-6-error-majorant-adapt-ref} demonstrate  the 
efficiency of the majorants and the error identity in terms of error estimation and show that}
$\overline{\rm M}^{\rm I\!I}$ is twice sharper than $\overline{\rm M}^{\rm I}$, whereas 
the error identity, as expected, reflects the error $|\!|\!|  e |\!|\!|_{\mathcal{L}}$ exactly. 
In Table \ref{tab:unit-domain-example-6-time-expenses-adapt-ref}, we see that the assembling of 
matrices for the $\boldsymbol{y}_h$ and $w_h$ is 3 times more time-consuming in comparison to 
assembling the system for $u_h$. 

\begin{table}[!t]
\scriptsize
\centering
\newcolumntype{g}{>{\columncolor{gainsboro}}c} 	
\newcolumntype{k}{>{\columncolor{lightgray}}c} 	
\newcolumntype{s}{>{\columncolor{silver}}c} 
\newcolumntype{a}{>{\columncolor{ashgrey}}c}
\newcolumntype{b}{>{\columncolor{battleshipgrey}}c}
\begin{tabular}{c|cga|ck|cb|cc}
\parbox[c]{0.4cm}{\centering \# ref. } & 
\parbox[c]{1.2cm}{\centering  $\| \nabla_x e \|_Q$}   & 	  
\parbox[c]{1.0cm}{\centering $I_{\rm eff} (\overline{\rm M}^{\rm I})$ } & 
\parbox[c]{1.0cm}{\centering $I_{\rm eff} (\overline{\rm M}^{\rm I\!I})$ } & 
\parbox[c]{1.2cm}{\centering  $|\!|\!|  e |\!|\!|_{l\!o\!c\!,h}$ }   & 	  
\parbox[c]{1.2cm}{\centering  $|\!|\!|  e |\!|\!|_{\mathcal{L}}$ }   & 	  
\parbox[c]{1.0cm}{\centering$I_{\rm eff} ({{\rm E \!\!\! Id}})$ } & 
\parbox[c]{1.2cm}{\centering e.o.c. ($|\!|\!|  e |\!|\!|_{l\!o\!c\!,h}$)} & 
\parbox[c]{1.2cm}{\centering e.o.c. ($|\!|\!|  e |\!|\!|_{\mathcal{L}}$)} \\[3pt]
\toprule
   2 &     3.1311e-04 &         2.85 &         1.55 &     3.1335e-04 & 
   5.6510e-02 &         1.00 &    17.71 &     8.64 \\
   3 &     1.0915e-04 &         3.93 &         1.73 &     1.0944e-04 & 
   3.1506e-02 &         1.00 &     6.49 &     3.60 \\
   5 &     2.2033e-05 &         2.27 &         1.36 &     2.2042e-05 & 
   1.4796e-02 &         1.00 &     5.87 &     3.59 \\
   7 &     5.2517e-06 &         2.38 &         1.22 &     5.2526e-06 &  
   7.2473e-03 &         1.00 &     2.41 &     1.27 \\
\end{tabular}
\caption{{\em Example 3}. 
Efficiency of $\overline{\rm M}^{\rm I}$, $\overline{\rm M}^{\rm I\!I}$, 
$\overline{\rm M}^{\rm I}_{h}$, ${{\rm E \!\!\! Id}}$, and 
{ the order of convergence} of $|\!|\!|  e |\!|\!|_{l\!o\!c\!,h}$ and $|\!|\!|  e |\!|\!|_{\mathcal{L}}$
with marking criterion $\mathds{M}_{\rm BULK}(0.6)$ for 
$u_h \in S^{2}_{h}$, 
$\boldsymbol{y}_h \in \oplus^2 S^{3}_{h}$, and 
$w_h \in S^{3}_{h}$ ($N_{\rm ref, 0}$ = 4).}
\label{tab:unit-domain-example-6-error-majorant-adapt-ref}
\end{table}
   
\begin{table}[!t]
\scriptsize
\centering
\newcolumntype{g}{>{\columncolor{gainsboro}}c} 	
\begin{tabular}{c|ccc|cgg|cgg|c}
& \multicolumn{3}{c|}{ d.o.f. } 
& \multicolumn{3}{c|}{ $t_{\rm as}$ }
& \multicolumn{3}{c|}{ $t_{\rm sol}$ } 
& $\tfrac{t_{\rm appr.}}{t_{\rm er.est.}}$ \\
\midrule
\parbox[c]{0.2cm}{\# ref. } & 
\parbox[c]{0.4cm}{\centering $u_h$ } &  
\parbox[c]{0.4cm}{\centering $\boldsymbol{y}_h$ } &  
\parbox[c]{0.4cm}{\centering $w_h$ } & 
\parbox[c]{0.8cm}{\centering $u_h$ } & 
\parbox[c]{0.8cm}{\centering $\boldsymbol{y}_h$ } & 
\parbox[c]{0.8cm}{\centering $w_h$ } & 
\parbox[c]{0.8cm}{\centering $u_h$ } & 
\parbox[c]{0.8cm}{\centering $\boldsymbol{y}_h$ } & 
\parbox[c]{0.8cm}{\centering $w_h$ } & \\
\toprule
   3 &        520 &       1088 &        544 &   6.50e-01 &   2.39e+00 &   1.80e+00 &         1.09e-02 &         8.60e-03 &         1.91e-02 &             0.27 \\
   5 &       1232 &       2474 &       1237 &   1.72e+00 &   6.26e+00 &   4.71e+00 &         6.12e-02 &         8.25e-02 &         1.26e-01 &             0.28 \\
   7 &       4368 &       8492 &       4246 &   6.20e+00 &   3.04e+01 &   1.87e+01 &         6.38e-01 &         6.05e-01 &         5.26e-01 &             0.22\\
   \midrule
    &       &         &    &
    \multicolumn{3}{c|}{ $t_{\rm as} (u_h)$ \;:\; $t_{\rm as} (\boldsymbol{y}_h)$ \;:\; $t_{\rm as} (w_h)$ } &      
    \multicolumn{3}{c|}{\; $t_{\rm sol} (u_h)$ \;:\; $t_{\rm sol} (\boldsymbol{y}_h)$  \;:\;  $t_{\rm sol} (w_h)$\;} & \\
 \midrule
	 & 	 & 	 & 	 &       0.33 &       1.62 &       1.00 &             1.21 &             1.15 &             1.00 &            \\
\end{tabular}
\caption{{\em Example 3}. 
Assembling and solving time (in seconds) spent for the systems generating d.o.f. of 
$u_h \in S^{2}_{h}$, $\boldsymbol{y}_h \in \oplus^2 S^{3}_{h}$, and $w_h \in S^{3}_{h}$ 
for the bulk marking parameter $\sigma =0.6$ for  ($N_{\rm ref, 0}$ = 4).}
\label{tab:unit-domain-example-6-time-expenses-adapt-ref}
\end{table}

\begin{figure}[!t]
	\subfloat[ref. 4]{
	 \LMRspacetimeaxis{
	\includegraphics[width=3.7cm, trim={3.7cm 3.7cm 3.7cm 3.7cm}, clip]{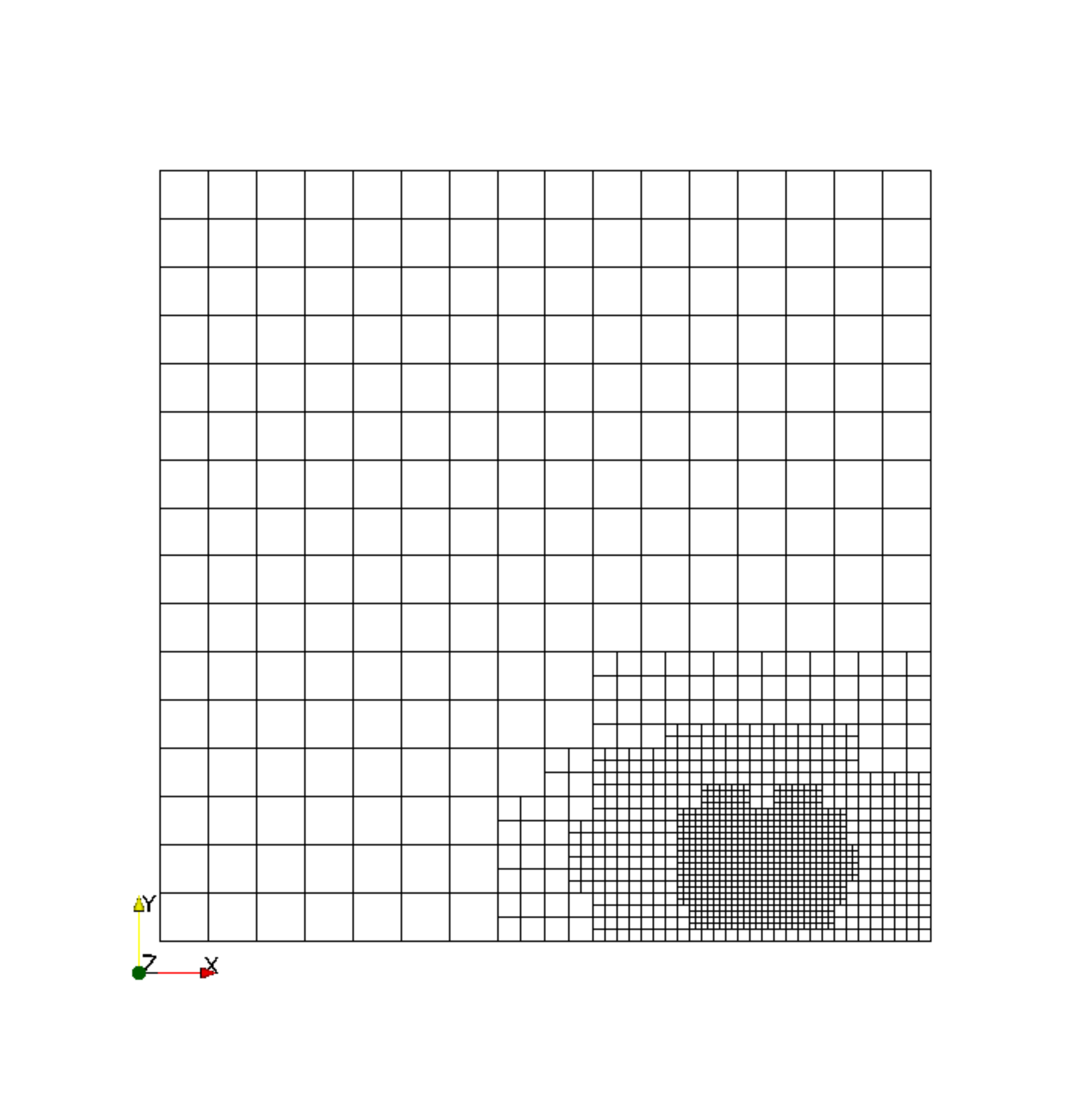}
	}
	}\!\!\!\!
	\subfloat[ref. 5]{
	 \LMRspacetimeaxis{
	\includegraphics[width=3.7cm, trim={3.7cm 3.7cm 3.7cm 3.7cm}, clip]{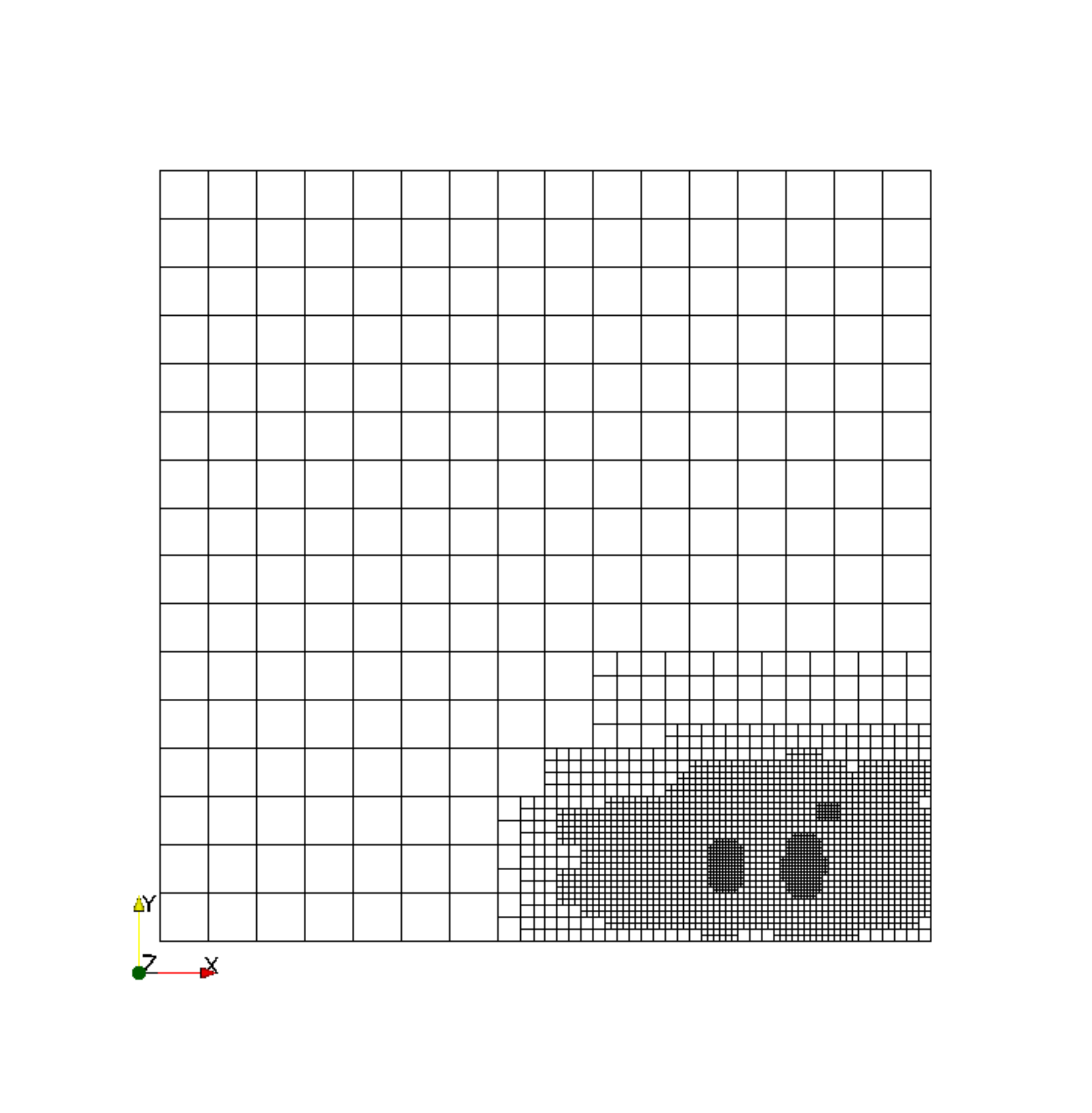}
	}
	}\!\!\!\!
	\subfloat[ref. 6]{
	 \LMRspacetimeaxis{
	\includegraphics[width=3.7cm, trim={3.7cm 3.7cm 3.7cm 3.7cm}, clip]{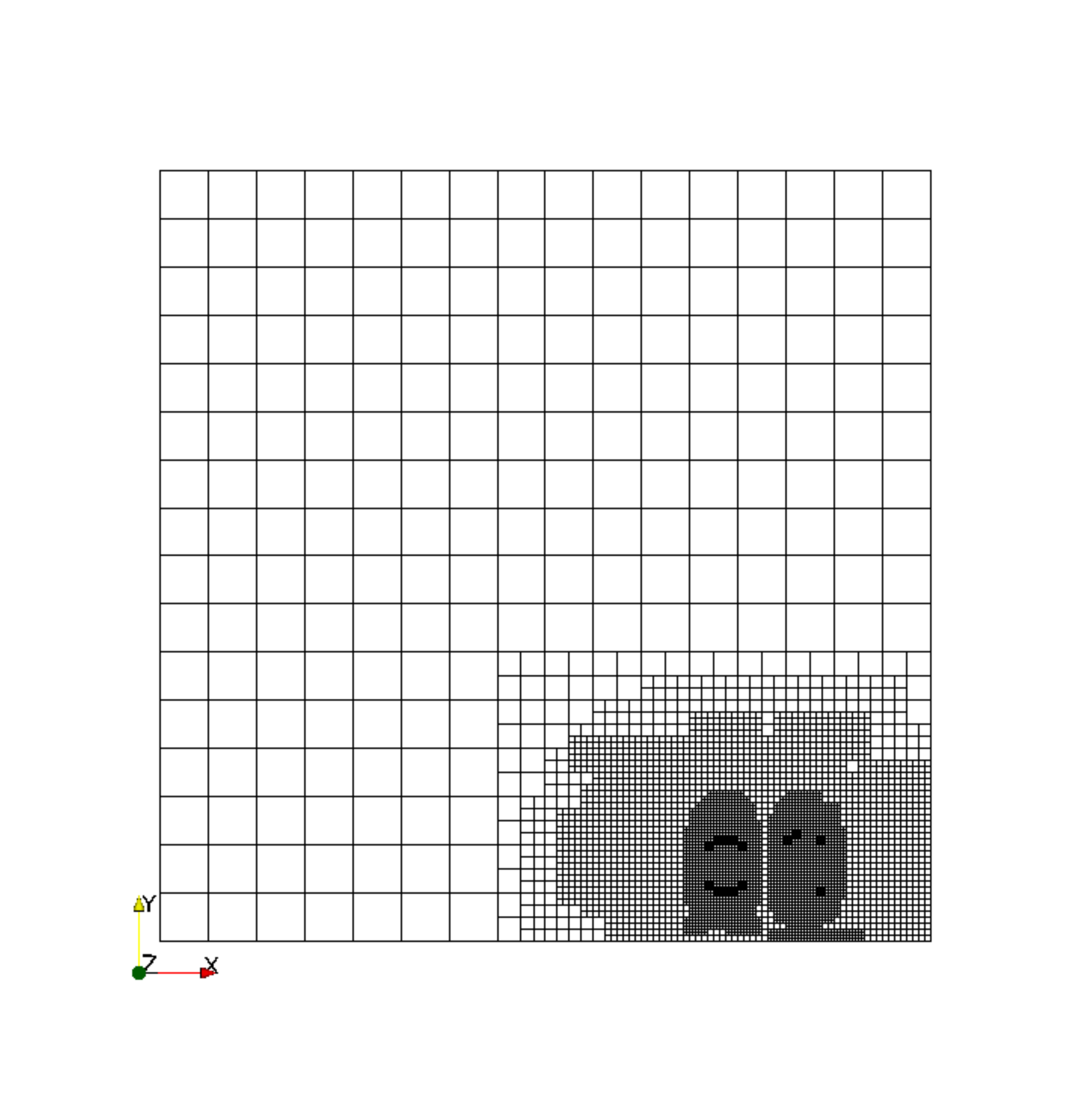}
	}
	}
	\caption{{\em Example 3}. Meshes obtained on the refinement steps 3--6, 
	$\sigma =0.6$ ($N_{\rm ref, 0}$ = 4) for $u_h \in S^{2}_{h}$, 
	$\boldsymbol{y}_h \in \oplus^2 S^{3}_{h}$, and $w_h \in S^{3}_{h}$.}
\end{figure}

\subsection{Example 4: solution with singularity w.r.t. $t$-coordinate}
\label{ex:example-4-time-singularity}
\rm
{
For Example 4, we consider the solution with the singularity w.r.t. time} 
coordinate, i.e., we take
$$u(x, t) = \sin \pi x \, (1 - t)^\lambda, 
\quad (x, t) \in \overline{Q} = (0, 1) \times (0, 2),$$
where parameter $\lambda = \Big\{ \tfrac{3}{2}, 1, \frac{1}{2} \Big\}$ (see Figure 
\ref{fig:time-singularity-1d-t-example-6} with $u$ for different $\lambda$). The RHS $f(x, t)$ follows 
from the substitution of $u$ into \eqref{eq:equation}, and the Dirichlet boundary condition is defined as 
{$u_D = u$ on $\Sigma$}.%

\begin{figure}[!t]
	\centering
	\centering
	\subfloat[$\lambda = \tfrac{3}{2}$]{
	\includegraphics[scale=0.45]{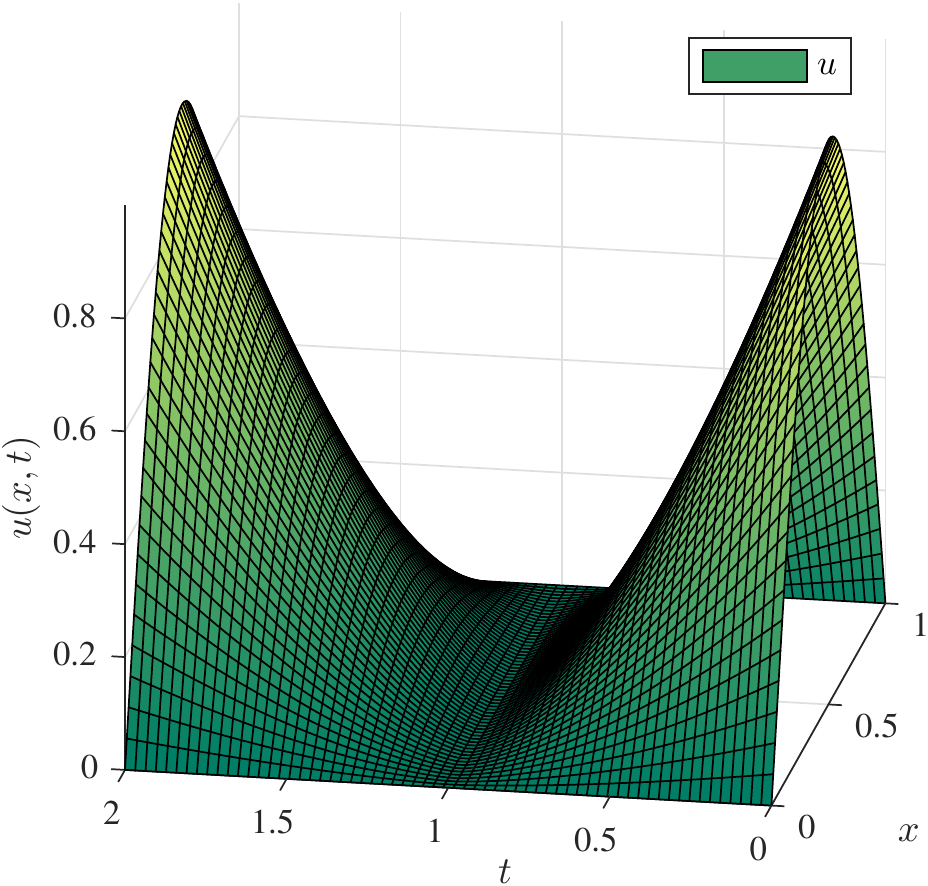}
	\label{fig:example-35-exact-solution}
	}
	\subfloat[$\lambda = 1$]{
	\includegraphics[scale=0.45]{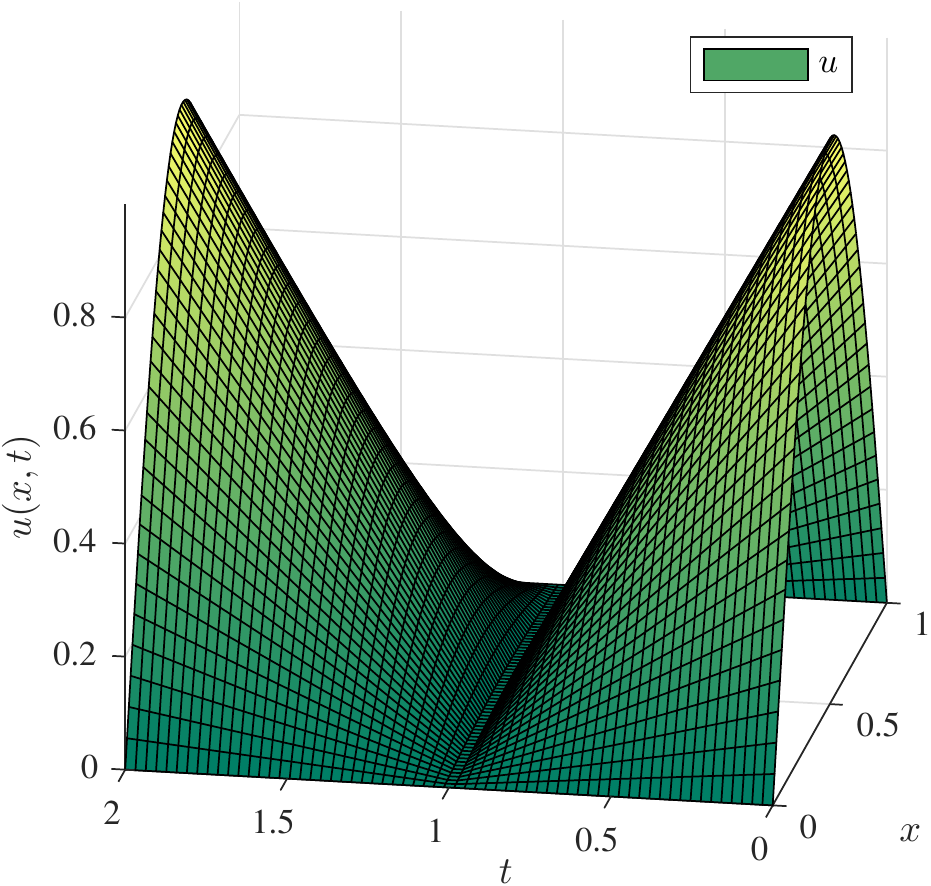}
	\label{fig:example-33-exact-solution}
	}	
	\subfloat[$\lambda = \tfrac{1}{2}$]{
	\includegraphics[scale=0.45]{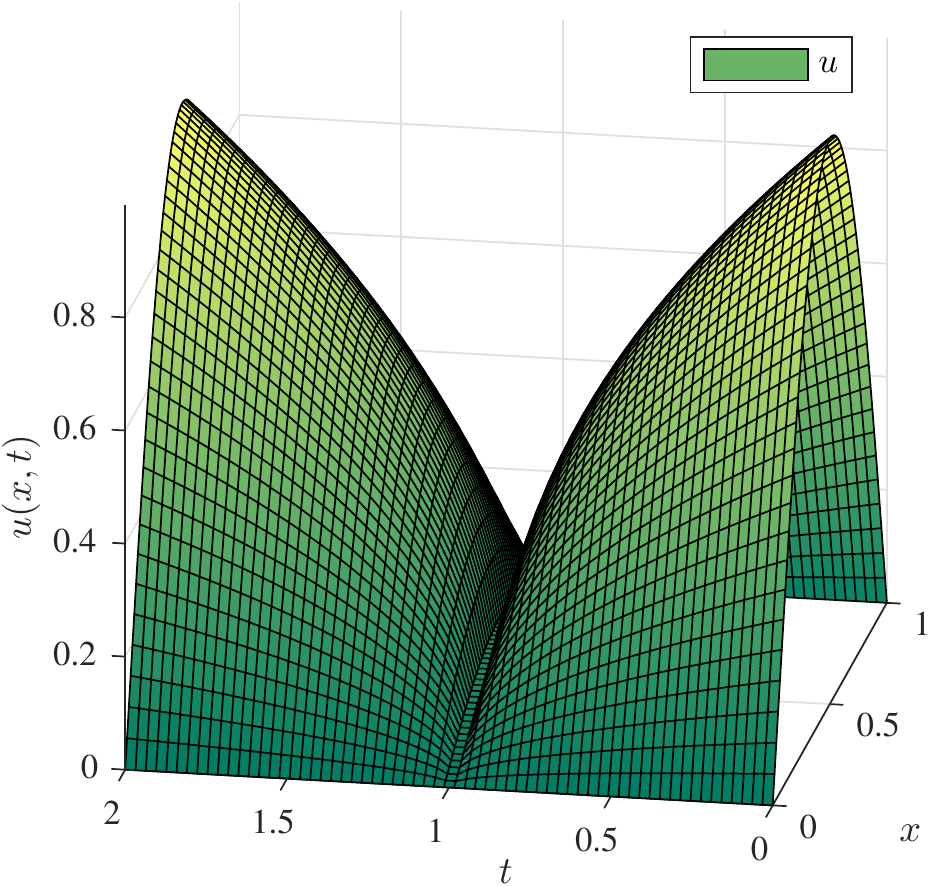}
	\label{fig:example-34-exact-solution}
	}	
	\caption{{\em Example 4}. 
	Exact solution $u(x, t) = \sin \pi x \, (1 - t)^{\lambda}$.}
	\label{fig:time-singularity-1d-t-example-6}
\end{figure}

\begin{figure}[!t]
	\centering
	\centering
	\subfloat[$\lambda = \tfrac{3}{2}$]{
	\includegraphics[scale=0.45]{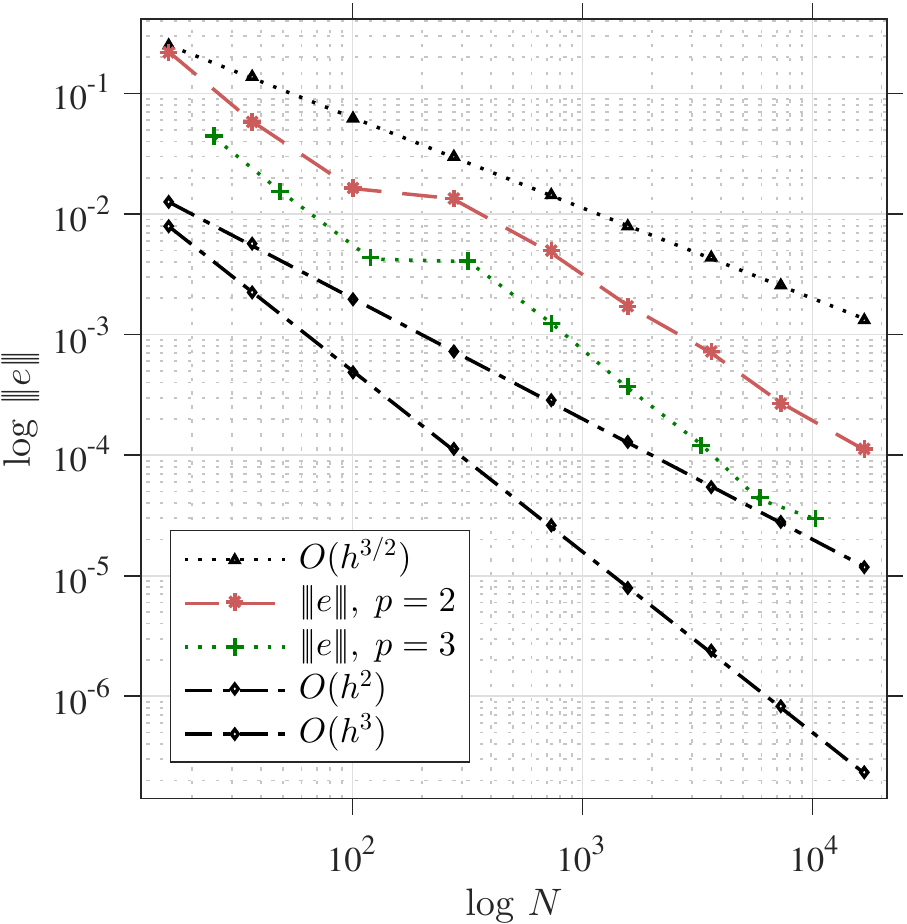}
	\label{fig:example-35-exact-solution-eoc}
	}
	\subfloat[$\lambda = 1$]{
	\includegraphics[scale=0.45]{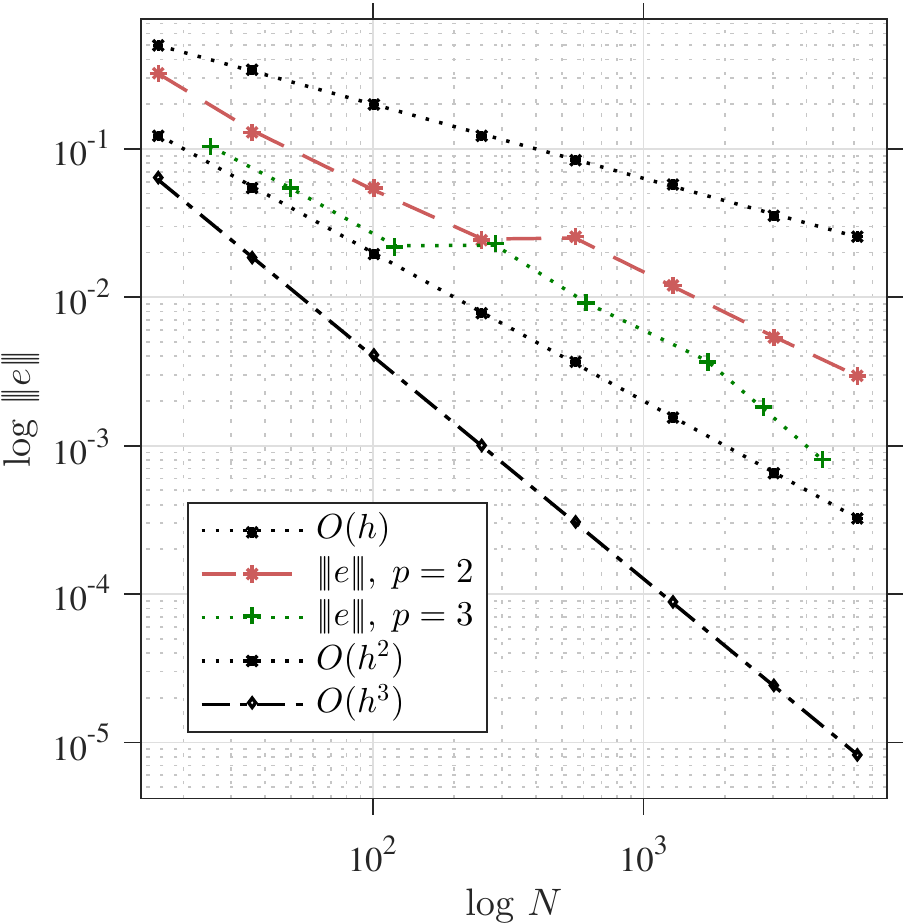}
	\label{fig:example-33-exact-solution-eoc}
	}	
	\subfloat[$\lambda = \tfrac{1}{2}$]{
	\includegraphics[scale=0.45]{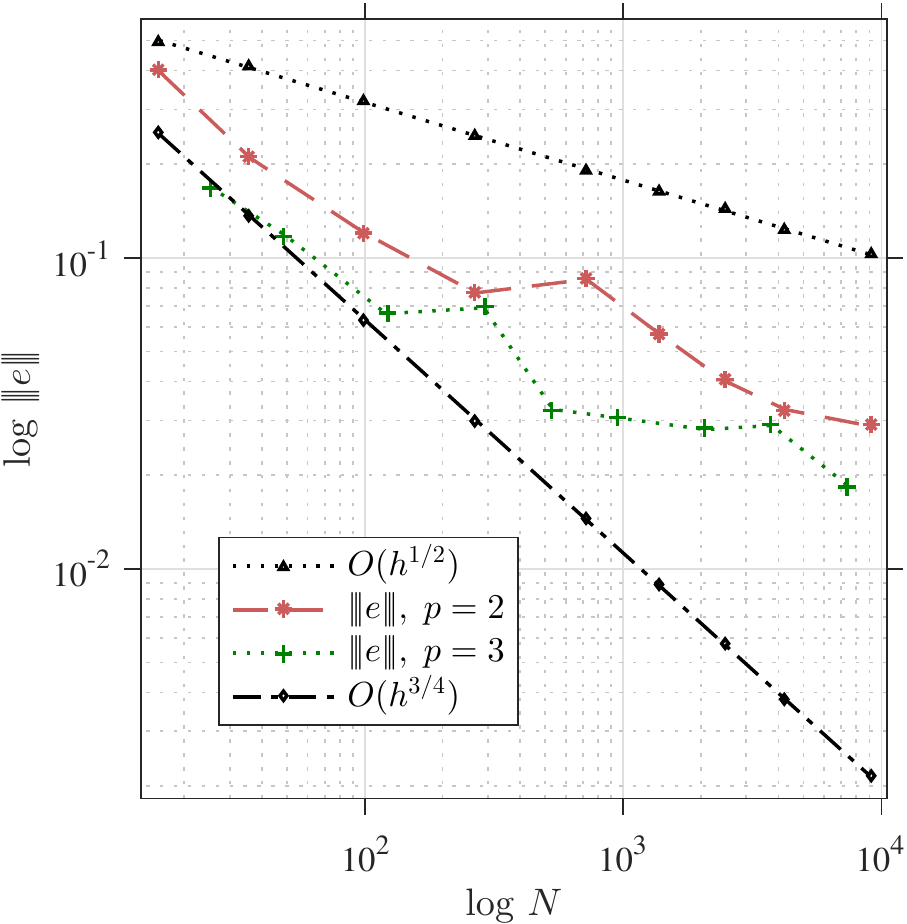}
	\label{fig:example-34-exact-solution-eoc}
	}	
	\caption{{\em Example 4}. { The error order of convergence} for approximations 
	with $u \in S_h^2$ and $S_h^3$:
	(a) $\lambda = \tfrac{3}{2}$,
	(b) $\lambda = 1$, 
	(c) $\lambda = \tfrac{1}{2}$.}
	\label{fig:time-singularity-1d-t-eoc-example-6}
\end{figure}

The solution $u(x, t)$ is smooth w.r.t. to spatial coordinates, while the regularity in time depends on the 
parameter $\lambda$. In particular, $p$ satisfies the following inequality 
$\lambda = p - \tfrac{1}{2} + \varepsilon$, where $\varepsilon >0$ is considerably small number. 
Then, the expected convergence in the term $h^{{1}\!/{2}} \, \| \partial_t (u - u_h) \|_{Q}$ is 
$O(h^{p - 1}) \cdot O(h^{{1}\!/{2}})$ (see \cite{LMRHoferLangerNeumuellerToulopoulos2017}).

Theoretical convergence for each $\lambda$ were tested in \cite{LMRLangerMatculevichRepinArxiv2017} (for $p = 2$). 
Table \ref{tab:example-4-time-singularity-error-majorant-adapt-ref} illustrates an improved 
{ error order of convergence} for $\lambda = \Big\{ \tfrac{1}{2}, 1, \tfrac{3}{2} \Big\}$. 
The same behaviour can be observed from Figure \ref{fig:time-singularity-1d-t-example-6} for different 
parameters. It presents 
meshes obtained on the adaptive refinement steps 5--7 and re-confirms that functional error estimates 
detect the local singularities rather efficiently. {For the case $\lambda = \tfrac{1}{2}$, 
we leave the class of solutions $V^{\Delta_x, 1}_{u_D} := \big\{u \in V^{\Delta_x, 1} :
\, u = u_D \; \mbox{on} \; {\Sigma} \big\}$, and, as consequence, 
are not able to recover the theoretical error order of convergence, since $\| \partial_t u\|_{L_2}$ explodes.
Nevertheless, the singularity at $t =1$ is captured and very well represented by the error indicator and 
resulting adaptive mesh.} Moreover, for the
rest of the times, i.e., $(0, 1) \cup (1, 2)$, where the solution is smooth, the mesh is not over-refined.

\begin{table}[!t]
\scriptsize
\centering
\newcolumntype{g}{>{\columncolor{gainsboro}}c} 	
\newcolumntype{k}{>{\columncolor{lightgray}}c} 	
\newcolumntype{s}{>{\columncolor{silver}}c} 
\newcolumntype{a}{>{\columncolor{ashgrey}}c}
\newcolumntype{b}{>{\columncolor{battleshipgrey}}c}
\begin{tabular}{c|cga|cc}
\parbox[c]{0.6cm}{\centering \# ref. } & 
\parbox[c]{1.2cm}{\centering  $\| \nabla_x e \|_Q$}   & 	  
\parbox[c]{1.0cm}{\centering $I_{\rm eff} (\overline{\rm M}^{\rm I})$ } & 
\parbox[c]{1.2cm}{\centering $I_{\rm eff} (\overline{\rm M}^{\rm I\!I})$ } & 
\parbox[c]{1.2cm}{\centering  $|\!|\!|  e |\!|\!|_{l\!o\!c\!,h}$ }   & 	  
\parbox[c]{0.8cm}{\centering e.o.c. ($|\!|\!|  e |\!|\!|_{\mathcal{L}}$)} \\[3pt]
\bottomrule
\multicolumn{6}{l}{ \rule{0pt}{3ex} 
(a) $\lambda = \tfrac{1}{2}$} \\[3pt]
\bottomrule
\multicolumn{6}{l}{ \rule{0pt}{3ex}   
\qquad 
$u_h \in S^{2}_{h}$, 
$\boldsymbol{y}_h \in \oplus^2 S^{4}_{h}$, and 
$w_h \in S^{4}_{h}$, 
\quad \mbox{theoretical e.o.c.} \;$O \Big(h^{{1}\!/{2}}\Big)$}\\[3pt]
\toprule   
   6 &     5.7560e-02 &         5.94 &         1.42 &     5.8807e-02 &   
   1.32 \\ 
   7 &     4.0317e-02 &         8.66 &         1.70 &     4.0749e-02 &    
   1.23 \\ 
   8 &     3.2498e-02 &        10.80 &         1.94 &     3.2703e-02 &     
   0.80 \\ 
\bottomrule
\multicolumn{6}{l}{ \rule{0pt}{3ex}   
\qquad 
$u_h \in S^{3}_{h}$, 
$\boldsymbol{y}_h \in \oplus^2 S^{5}_{h}$, and 
$w_h \in S^{5}_{h}$, 
\quad \mbox{theoretical e.o.c.} \;$O \Big(h^{{1}\!/{2}}\Big)$ }\\[3pt]
\toprule   
   6 &     3.0549e-02 &         9.78 &         1.83 &     3.3397e-02 &   
   0.40 \\
   7 &     2.8038e-02 &        10.51 &         1.85 &     2.9693e-02 &   
   0.30 \\
   9 &     1.8447e-02 &        15.74 &         2.17 &     1.8700e-02 &   
   1.37 \\
\bottomrule
\multicolumn{6}{l}{ \rule{0pt}{3ex} 
(b) $\lambda = 1$}\\[3pt]
\bottomrule
\multicolumn{6}{l}{ \rule{0pt}{3ex}   
\qquad $u_h \in S^{2}_{h}$, 
$\boldsymbol{y}_h \in \oplus^2 S^{3}_{h}$, and 
$w_h \in S^{3}_{h}$, 
\quad \mbox{theoretical e.o.c.} \;$O(h)$}\\[3pt]
\toprule
   6 &     1.1955e-02 &         4.47 &         1.44 &     1.2104e-02 &   
   1.85 \\
   7 &     5.3797e-03 &         6.84 &         1.70 &     5.4167e-03 &  
   1.83 \\ 
   8 &     2.9478e-03 &         9.37 &         2.01 &     2.9602e-03 &  
   1.73 \\ 
\bottomrule
\multicolumn{6}{l}{ 
\rule{0pt}{3ex}   
\qquad $u_h \in S^{3}_{h}$, 
$\boldsymbol{y}_h \in \oplus^2 S^{4}_{h}$, and 
$w_h \in S^{4}_{h}$, 
\quad \mbox{theoretical e.o.c.} \; $O\Big(h\Big)$}\\[3pt]
\toprule
   6 &     3.7397e-03 &         7.95 &         1.58 &     3.9142e-03 & 
   1.83 \\ 
   7 &     1.8031e-03 &        11.66 &         2.49 &     1.8454e-03 & 
   3.20 \\
   8 &     7.9328e-04 &        20.28 &         3.28 &     8.1331e-04 & 
   3.14 \\
\bottomrule
\multicolumn{6}{l}{ \rule{0pt}{3ex} 
(c) $\lambda = \tfrac{3}{2}$}\\[3pt]
\bottomrule
\multicolumn{6}{l}{ \rule{0pt}{3ex}   
\qquad $u_h \in S^{2}_{h}$, 
$\boldsymbol{y}_h \in \oplus^2 S^{4}_{h}$, and 
$w_h \in S^{4}_{h}$, 
\quad \mbox{theoretical e.o.c.} \;$O\Big(h^{{3}\!/_{2}}\Big)$}\\[3pt]
\toprule
   6 &     1.7201e-03 &         3.89 &         1.28 &     1.7489e-03 &   
   2.66 \\ 
   7 &     7.1799e-04 &         4.58 &         1.51 &     7.2230e-04 &   
   2.18 \\
   8 &     2.7180e-04 &         6.44 &         1.70 &     2.7294e-04 &   
   2.73 \\
   9 &     1.1070e-04 &         8.80 &         1.90 &     1.1088e-04 &   
   2.16 \\
\bottomrule
\multicolumn{6}{l}{ \rule{0pt}{3ex}   
\qquad 
$u_h \in S^{3}_{h}$, 
$\boldsymbol{y}_h \in \oplus^2 S^{5}_{h}$, and 
$w_h \in S^{5}_{h}$, 
\quad \mbox{theoretical e.o.c.} \;$O \Big(h^{{3}\!/_{2}}\Big)$ }\\[3pt]
\toprule 
   6 &     3.6941e-04 &         7.48 &         5.84 &     3.8382e-04 & 
   3.30 \\
   7 &     1.2426e-04 &         8.76 &         2.09 &     1.2580e-04 &  
   3.00 \\
   8 &     4.3053e-05 &        14.49 &         3.80 &     4.3483e-05 & 
   3.65 \\
   9 &     2.9692e-05 &        12.48 &         3.14 &     2.9790e-05 &  
   1.35 \\
\end{tabular}
\caption{{\em Example 4}. 
Efficiency of $\overline{\rm M}^{\rm I}$, $\overline{\rm M}^{\rm I\!I}$, ${{\rm E \!\!\! Id}}$, and 
{ order of convergence} of $|\!|\!|  e |\!|\!|_{l\!o\!c\!,h}$ and $|\!|\!|  e |\!|\!|_{\mathcal{L}}$ for 
$\sigma =0.4$ ($N_{\rm ref, 0}$ = 1).}
\label{tab:example-4-time-singularity-error-majorant-adapt-ref}
\end{table}

\begin{figure}[!t]
	\centering
	\captionsetup[subfigure]{oneside}
	\subfloat[ref. 5]{
	\includegraphics[width=4cm, trim={9cm 0.5cm 9cm 0.5cm}, clip]{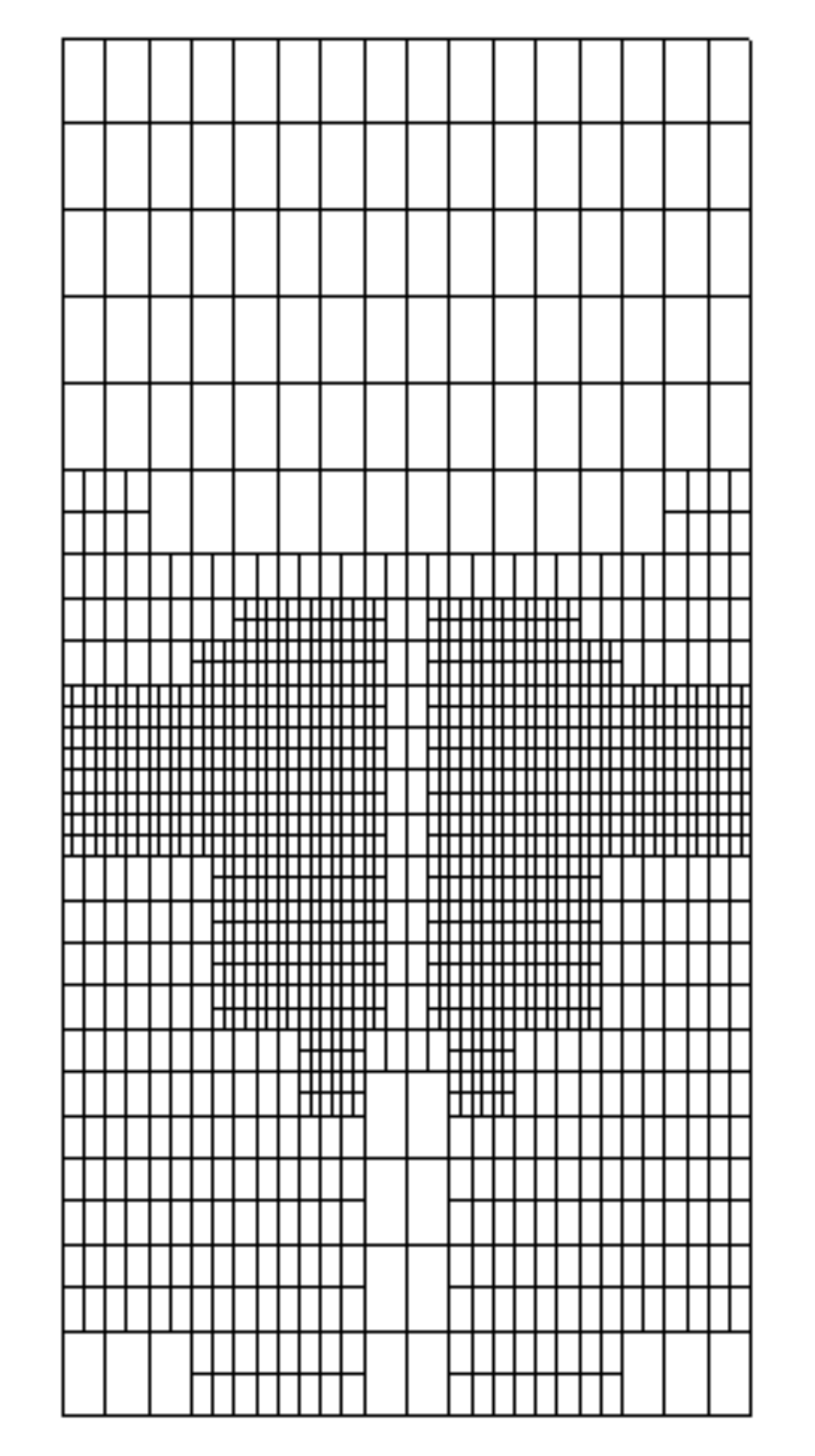}
	}
	\subfloat[ref. 6]{
	\includegraphics[width=4cm, trim={9cm 0.5cm 9cm 0.5cm}, clip]{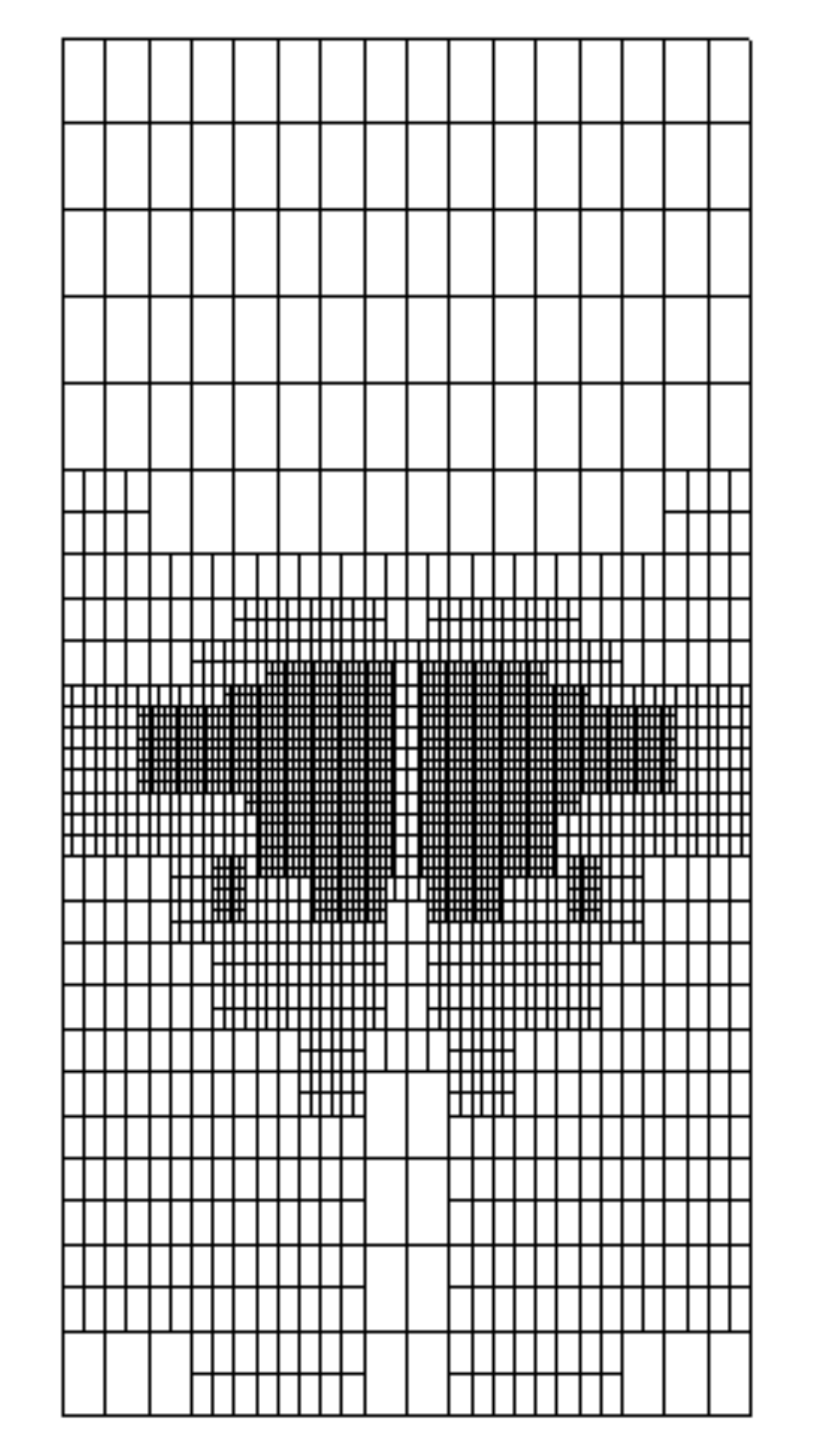}
	}
	\subfloat[ref. 7]{
	\includegraphics[width=4cm, trim={9cm 0.5cm 9cm 0.5cm}, clip]{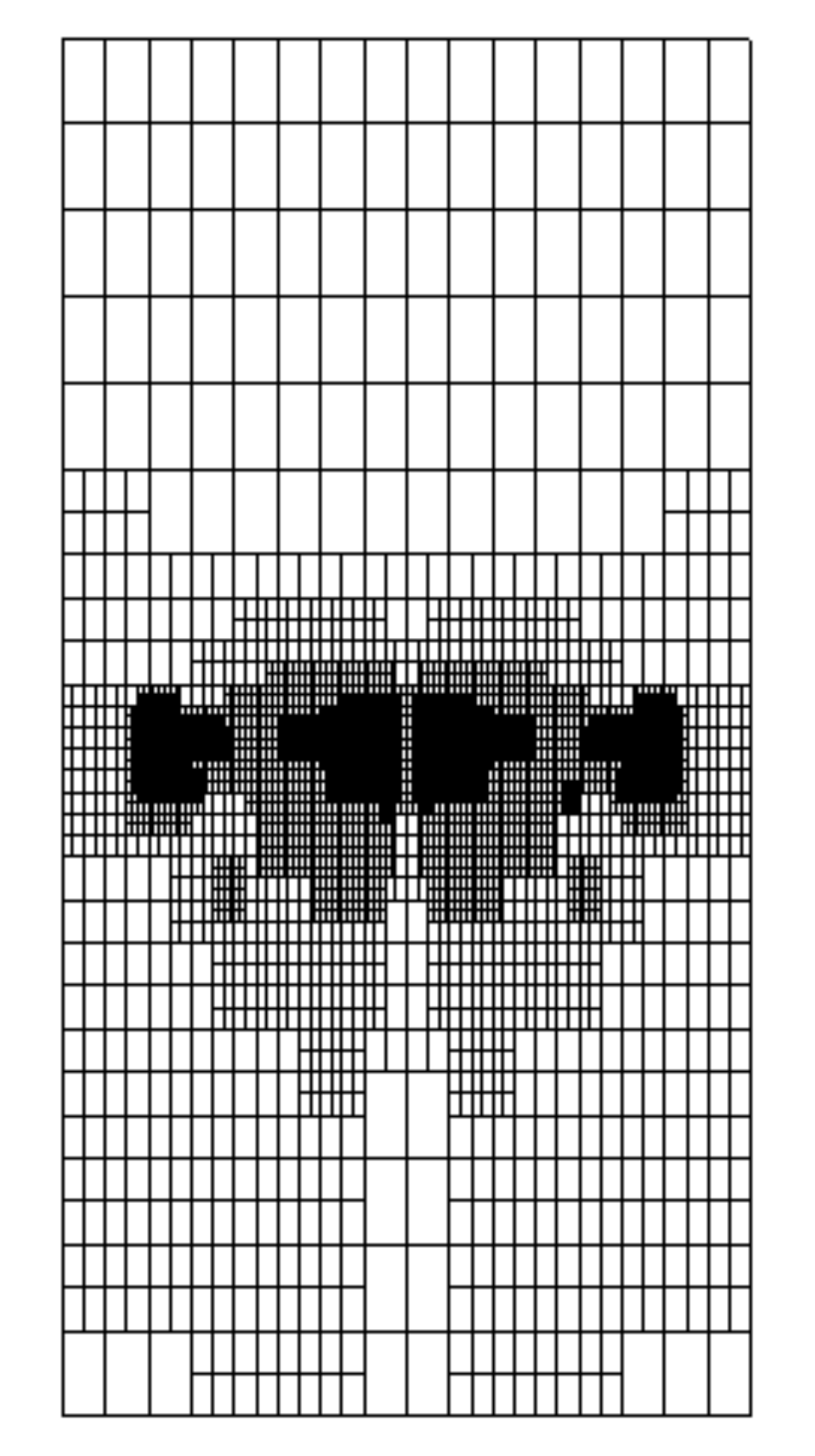}
	} 
	\caption{{\em Example 3 (case (a): $\lambda = \tfrac{1}{2}$)}. 
	Meshes obtained on the refinement steps 5--7 for 
	$u_h \in S^{2}_{h}$, $\boldsymbol{y}_h \in \oplus^2 S^{3}_{h}$, and $w_h \in S^{3}_{h}$.
}
\label{fig:unit-domain-example-34-ed-md-distribution}
\end{figure}

\begin{figure}[!t]
	\centering
	\captionsetup[subfigure]{oneside}
	\subfloat[ref. 5]{
	\includegraphics[width=4cm, trim={9cm 0.5cm 9cm 0.5cm}, clip]{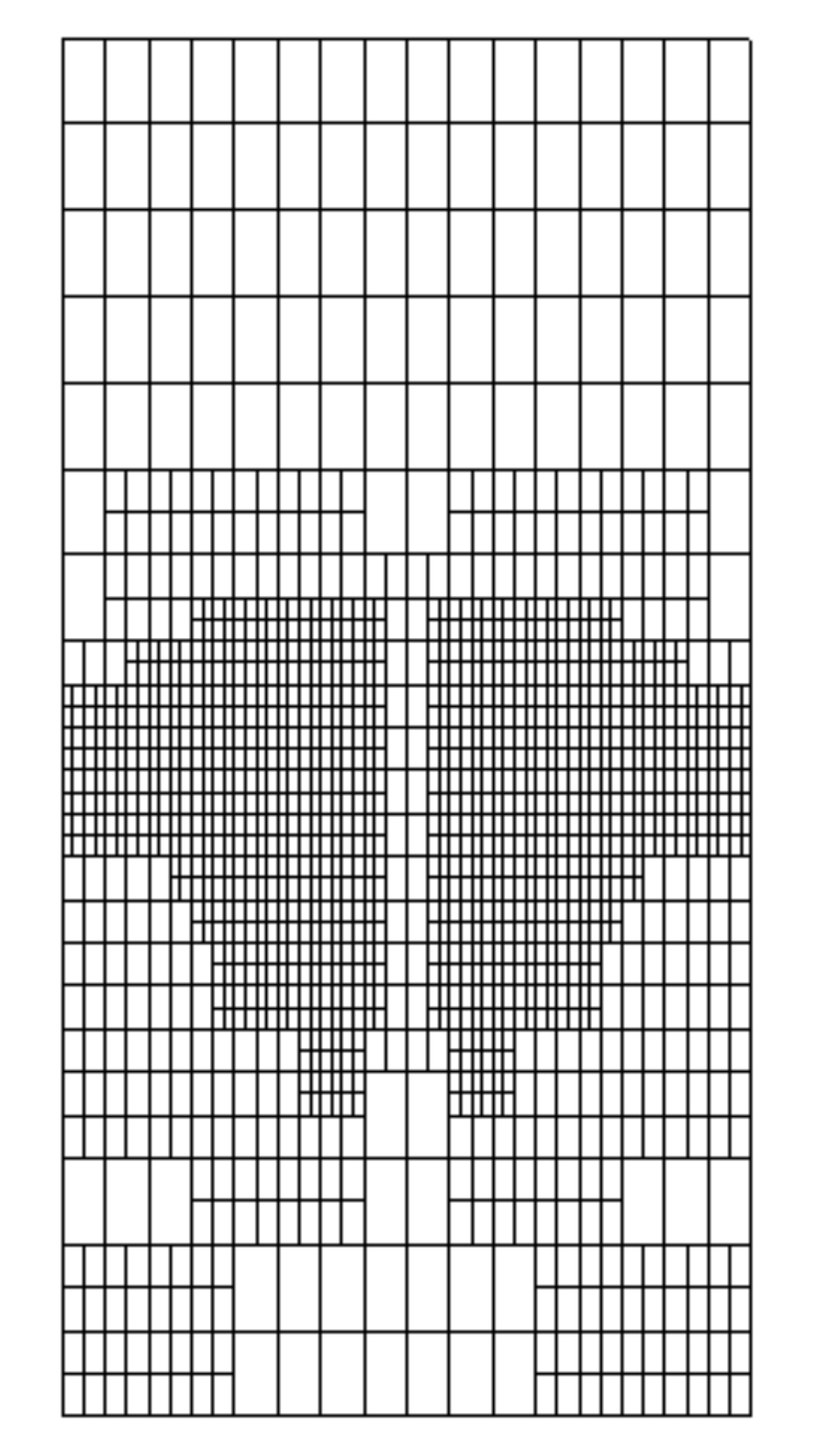}
	}
	\subfloat[ref. 6]{
	\includegraphics[width=4cm, trim={9cm 0.5cm 9cm 0.5cm}, clip]{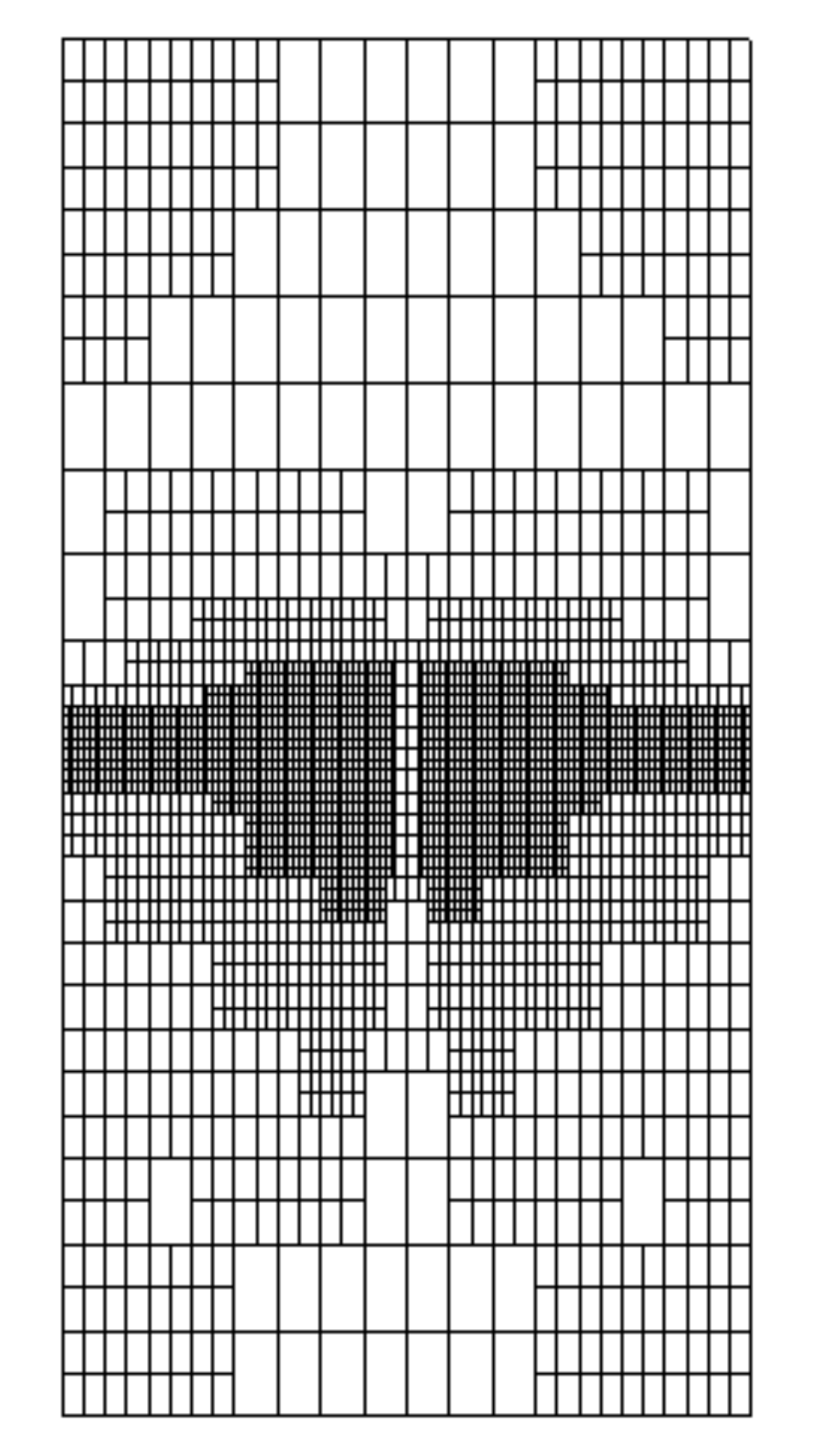}
	}
	\subfloat[ref. 7]{
	\includegraphics[width=4cm, trim={9cm 0.5cm 9cm 0.5cm}, clip]{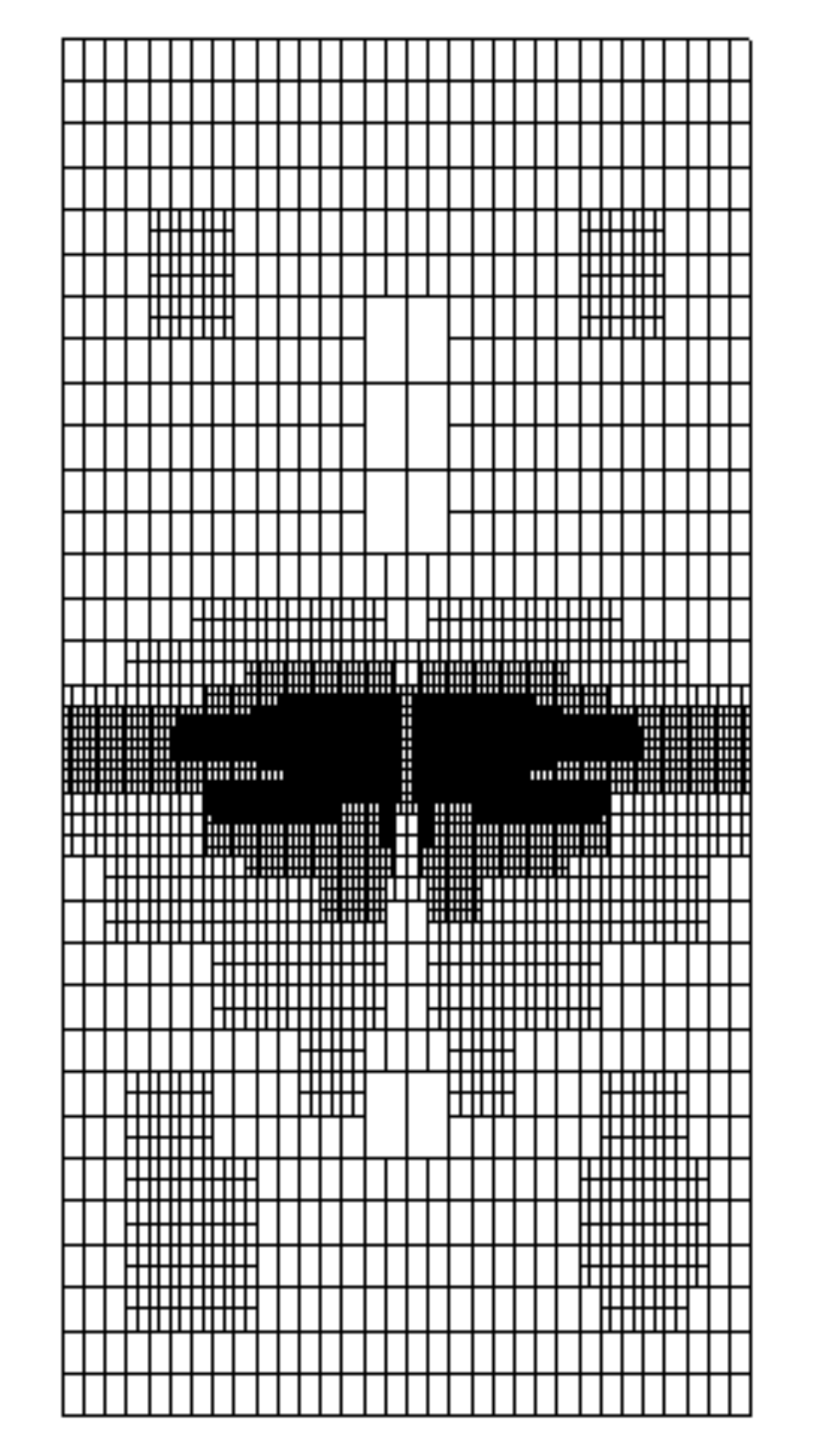}
	} 
	\caption{{\em Example 3 (case (b): $\lambda = 1$)}. 
	Meshes obtained on the refinement steps 5--7 for 
	$u_h \in S^{2}_{h}$, $\boldsymbol{y}_h \in \oplus^2 S^{3}_{h}$, and $w_h \in S^{3}_{h}$.
	}
\label{fig:unit-domain-example-33-ed-md-distribution}
\end{figure}

\begin{figure}[!t]
	\centering
	\captionsetup[subfigure]{oneside}
	\subfloat[ref. 5]{
	\includegraphics[width=4cm, trim={9cm 0.5cm 9cm 0.5cm}, clip]{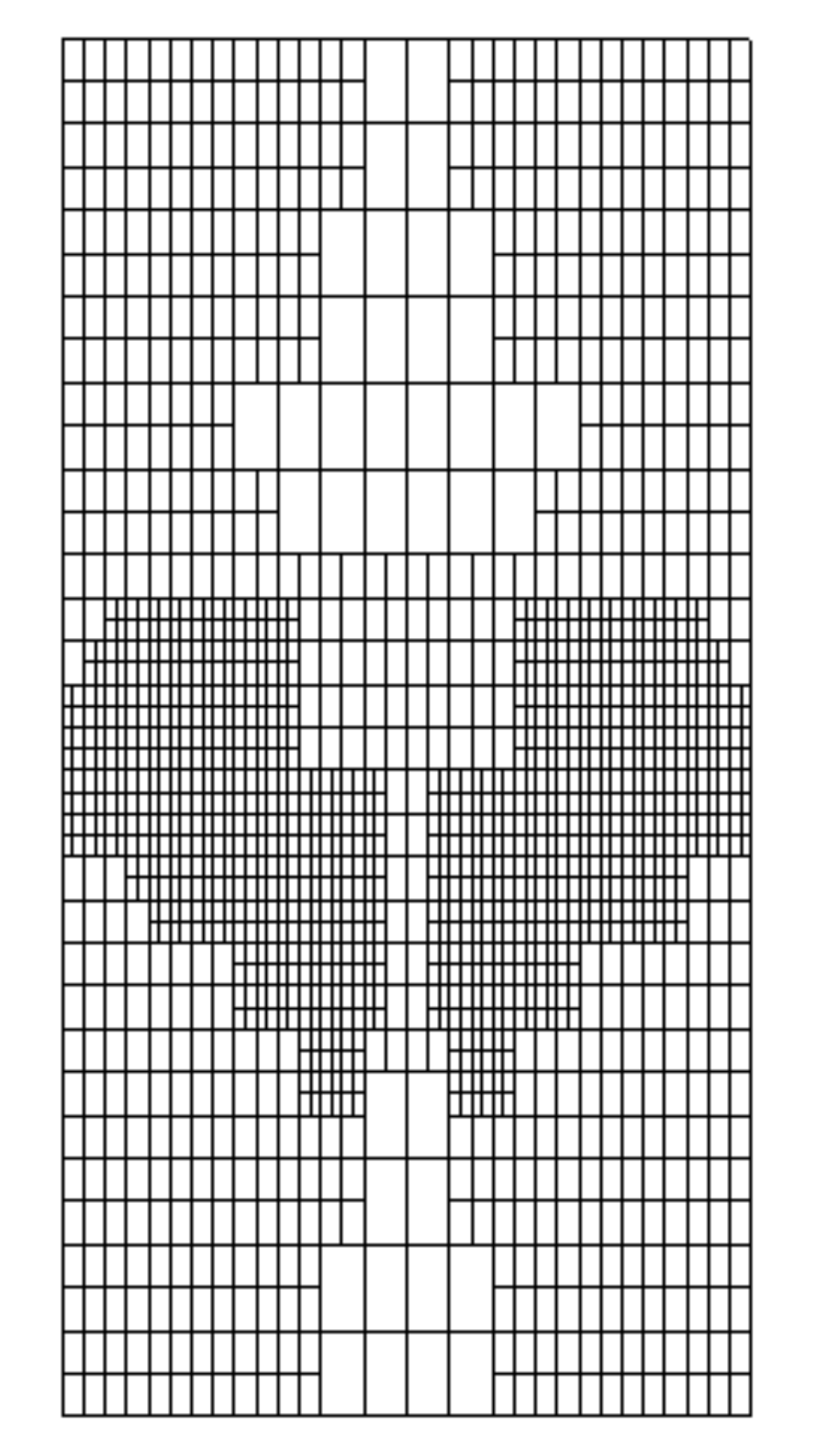}
	}
	\subfloat[ref. 6]{
	\includegraphics[width=4cm, trim={9cm 0.5cm 9cm 0.5cm}, clip]{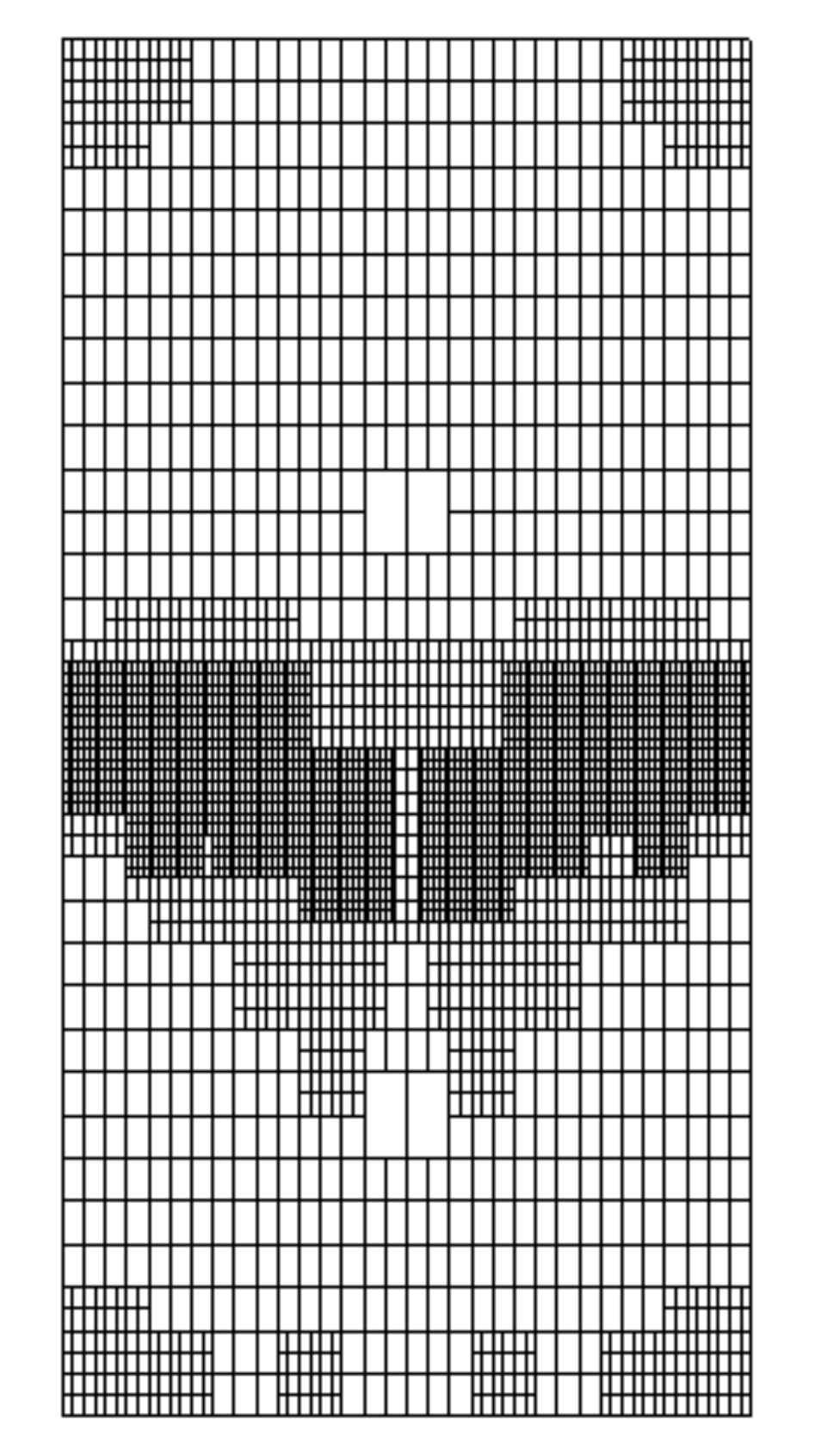}
	}
	\subfloat[ref. 7]{
	\includegraphics[width=4cm, trim={9cm 0.5cm 9cm 0.5cm}, clip]{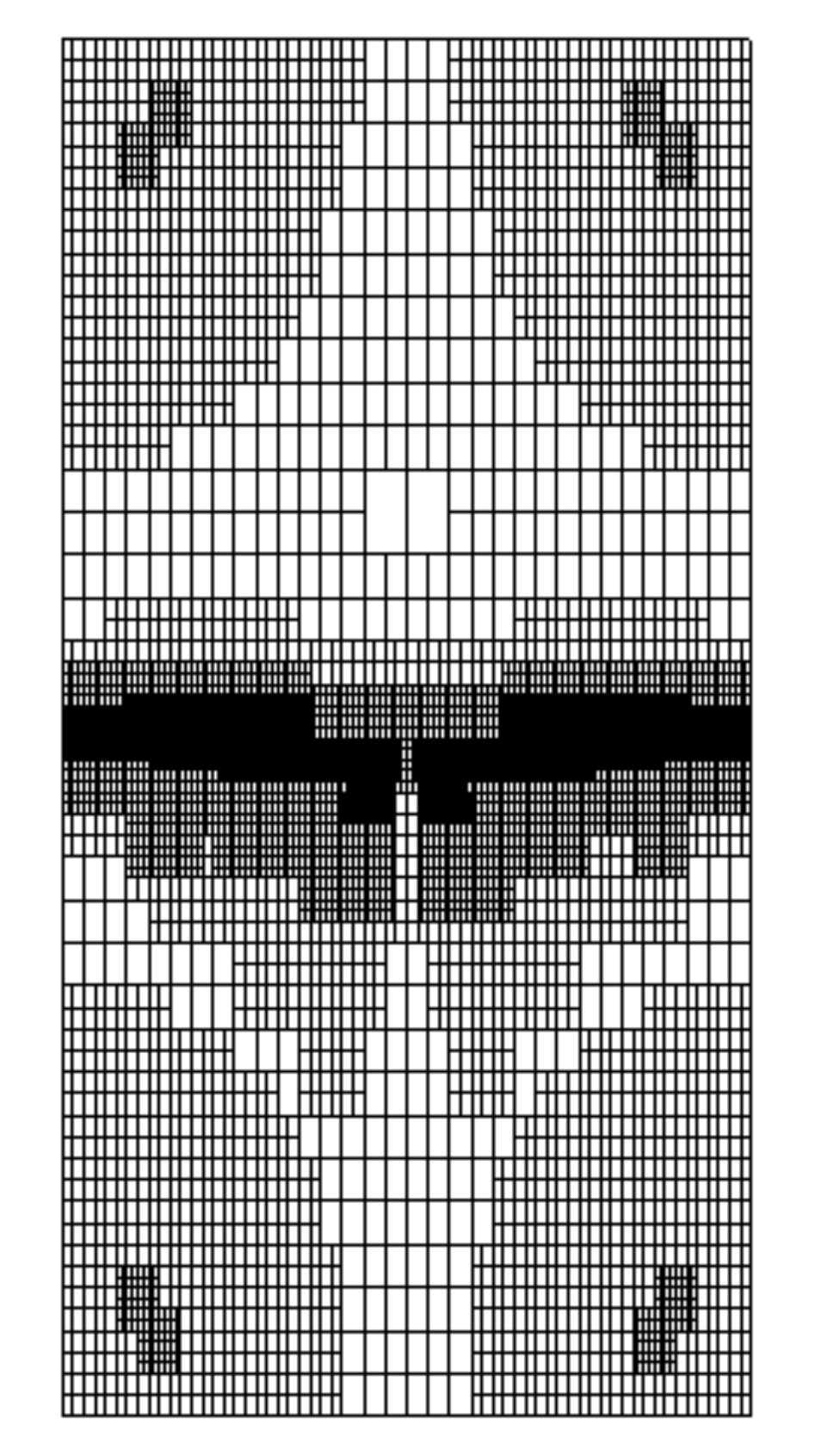}
	} 
	\caption{{\em Example 3 (case (c): $\lambda = \tfrac{3}{2}$)}. 
	Meshes obtained on the refinement steps 5--7 for 
	$u_h \in S^{2}_{h}$, $\boldsymbol{y}_h \in \oplus^2 S^{3}_{h}$, and $w_h \in S^{3}_{h}$.
}
\label{fig:unit-domain-example-35-ed-md-distribution}
\end{figure}

\subsection{Example 5: quarter-annulus domain extended in time}
\label{ex:example-5-quarter-annulus}

In the last example, we test the problem defined in the three-dimensional 
{ space-time cylinder} $Q = \Omega_{ \,
\begin{tikzpicture}[scale=0.1]
\draw (0:1cm) -- (0:2cm)
arc (0:90:2cm) -- (90:1cm)
arc (90:0:1cm) -- cycle;
\end{tikzpicture}} \times (0, T)$, where  
$\Omega_{ 
\begin{tikzpicture}[scale=0.1]
\draw (0:1cm) -- (0:2cm)
arc (0:90:2cm) -- (90:1cm)
arc (90:0:1cm) -- cycle;
\end{tikzpicture}}
$ 
is represented by a quarter-annulus, which extended form $t = 0$ till $t = T = 1$. 
The exact solution is defined by
\begin{alignat*}{4}
u(x, y, t) 	& = (1 - x) \, x^2 \, (1 - y) \, y^2 \, (1 - t) \, t^2, 
		& \quad  (x, y, t) & \in 
\overline{Q} := \overline{\Omega}_{ 
\begin{tikzpicture}[scale=0.1]
\draw (0:1cm) -- (0:2cm)
arc (0:90:2cm) -- (90:1cm)
arc (90:0:1cm) -- cycle;
\end{tikzpicture}}
\times [0, 1].
\end{alignat*}
\begin{figure}[!t]
	\centering
	\includegraphics[scale=0.5]{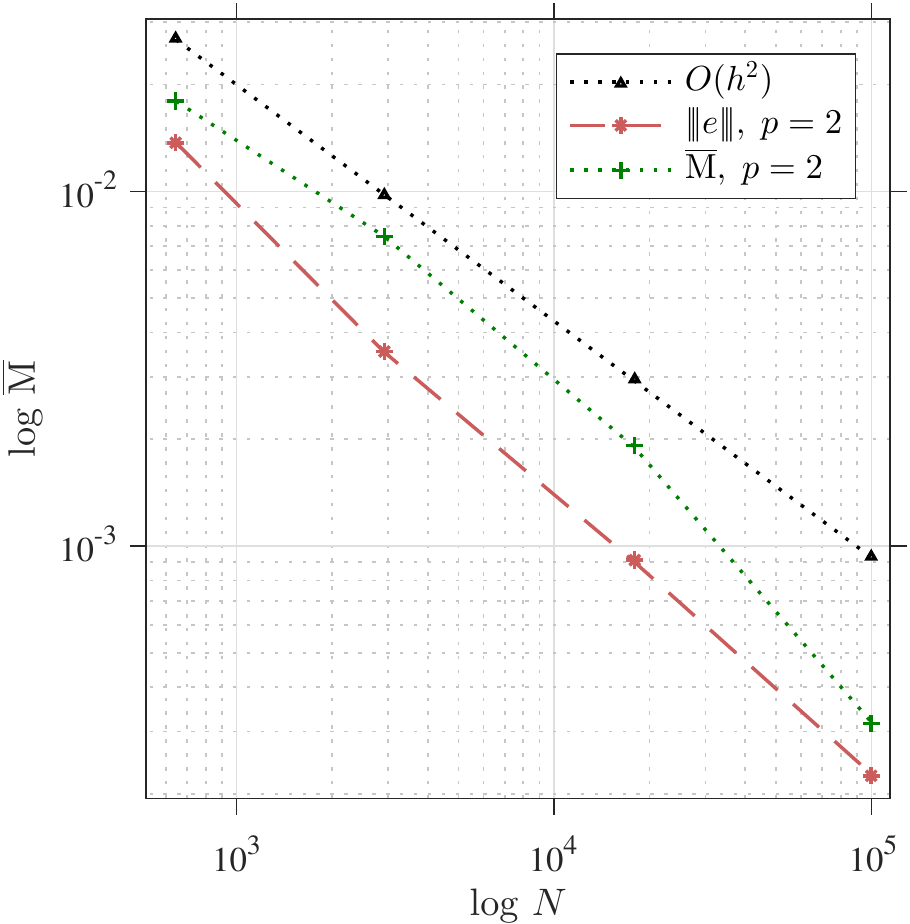}
	\label{fig:example-35-exact-solution}
	\caption{{\em Example 5}. 
	{ The error order of convergence} for $u \in S_h^2$.}
	\label{ex:example-23-quarter-annulus}
\end{figure}
The RHS $f(x, y, t)$, $(x, y, t) \in 
Q := {\Omega}_{ 
\begin{tikzpicture}[scale=0.1]
\draw (0:1cm) -- (0:2cm)
arc (0:90:2cm) -- (90:1cm)
arc (90:0:1cm) -- cycle;
\end{tikzpicture}} \times (0, 1)$, is computed based on the substitution of $u$ into the equation 
\eqref{eq:equation} and the Dirichlet boundary condition is defined as {$u_D = u$ on $\Sigma$.}

The initial mesh for the test is generated by one uniform refinement $N_{\rm ref, 0} = 1$, 
the bulk marking parameter is set to $\sigma =0.4$. The error order of convergence is illustrated in 
Figure \ref{ex:example-23-quarter-annulus}, which corresponds to the theoretical expectation.
We start the analysis from Table \ref{tab:unit-domain-example-23-error-majorant-adapt-ref}. 
It is easy to see that all majorants have adequate performance, taking into account 
that the auxiliary functions $\boldsymbol{y}_h  \in \oplus^2 S_{3h}^{5}$ and $w_h \in S^{5}_{3h}$. 
Table \ref{tab:unit-domain-example-23-time-expenses-adapt-ref} confirms that 
assembling and solving of the systems reconstructing d.o.f. of $u_h$ requires more time than assembling 
and solving routines for the systems generating $\boldsymbol{y}_h $ and $w_h$.

Figure \ref{fig:time-singularity-1d-t-example-23} presents an evolution of the adaptive meshes 
discretising expanded in time quarter-annulus $Q$. From the plots presented, we can see 
that the refinement is localised in the area close to the lateral surface of the quarter-annulus with the 
radius two. This happens due to fast changes in the solution appearing close to this
`outer' surface. 

\begin{table}[!t]
\scriptsize
\centering
\newcolumntype{g}{>{\columncolor{gainsboro}}c} 	
\newcolumntype{k}{>{\columncolor{lightgray}}c} 	
\newcolumntype{s}{>{\columncolor{silver}}c} 
\newcolumntype{a}{>{\columncolor{ashgrey}}c}
\newcolumntype{b}{>{\columncolor{battleshipgrey}}c}
\begin{tabular}{c|cgg|g|cc|gc}
\parbox[c]{0.4cm}{\centering \# ref. } & 
\parbox[c]{1.2cm}{\centering  $\| \nabla_x e \|_Q$}   & 	  
\parbox[c]{1.0cm}{\centering $I_{\rm eff} (\overline{\rm M}^{\rm I})$ } & 
\parbox[c]{1.0cm}{\centering $I_{\rm eff} (\overline{\rm M}^{\rm I\!I})$ } & 
\parbox[c]{1.2cm}{\centering  $|\!|\!|  e |\!|\!|_{l\!o\!c\!,h}$ }   & 	  
\parbox[c]{1.2cm}{\centering  $|\!|\!|  e |\!|\!|_{\mathcal{L}}$ }   & 	  
\parbox[c]{1.0cm}{\centering$I_{\rm eff} ({{\rm E \!\!\! Id}})$ } & 
\parbox[c]{1.2cm}{\centering e.o.c. ($|\!|\!|  e |\!|\!|_{l\!o\!c\!,h}$)} & 
\parbox[c]{1.2cm}{\centering e.o.c. ($|\!|\!|  e |\!|\!|_{\mathcal{L}}$)} \\[3pt]
\midrule
   3 &     1.3711e-02 &         1.31 &         1.20 &     1.3722e-02 & 
   2.5548e-01 &         1.00 &     4.66 &     2.14 \\
   4 &     3.5322e-03 &         2.12 &         1.74 &     3.5331e-03 & 
   1.2719e-01 &         1.00 &     2.70 &     1.39 \\
   5 &     9.0289e-04 &         2.11 &         1.90 &     9.0425e-04 & 
   5.9632e-02 &         1.00 &     2.25 &     1.25 \\
   6 &     2.2747e-04 &         1.40 &         1.69 &     2.2749e-04 & 
   3.1509e-02 &         1.00 &     2.41 &     1.11 \\
\end{tabular}
\caption{{\em Example 5}. 
Efficiency of $\overline{\rm M}^{\rm I}$, $\overline{\rm M}^{\rm I\!I}$, and ${{\rm E \!\!\! Id}}$ 
for the bulk marking parameter $\sigma =0.4$ for
$u_h \in S^{2}_{h}$, 
$\boldsymbol{y}_h \in \oplus^3 S^{3}_{h}$, and 
$w_h \in S^{5}_{3h}$ ($N_{\rm ref, 0}$ = 4).}
\label{tab:unit-domain-example-23-error-majorant-adapt-ref}
\end{table}
   
\begin{table}[!t]
\scriptsize
\centering
\newcolumntype{g}{>{\columncolor{gainsboro}}c} 	
\begin{tabular}{c|ccc|cgg|cgg|c}
& \multicolumn{3}{c|}{ d.o.f. } 
& \multicolumn{3}{c|}{ $t_{\rm as}$ }
& \multicolumn{3}{c|}{ $t_{\rm sol}$ } 
& $\tfrac{t_{\rm appr.}}{t_{\rm er.est.}}$ \\
\midrule
\parbox[c]{0.2cm}{\# ref. } & 
\parbox[c]{0.4cm}{\centering $u_h$ } &  
\parbox[c]{0.4cm}{\centering $\boldsymbol{y}_h$ } &  
\parbox[c]{0.4cm}{\centering $w_h$ } & 
\parbox[c]{0.8cm}{\centering $u_h$ } & 
\parbox[c]{0.8cm}{\centering $\boldsymbol{y}_h$ } & 
\parbox[c]{0.8cm}{\centering $w_h$ } & 
\parbox[c]{0.8cm}{\centering $u_h$ } & 
\parbox[c]{0.8cm}{\centering $\boldsymbol{y}_h$ } & 
\parbox[c]{0.8cm}{\centering $w_h$ } &
\\
\bottomrule
\multicolumn{10}{l}{ \rule{0pt}{3ex}   
(a) $u_h \in S^{2}_{3h}$, 
$\boldsymbol{y}_h \in S^{5}_{3h} \oplus S^{3}_{h}$, and 
$w_h \in S^{5}_{3h}$} \\[2pt]
\toprule
   3 &        646 &       1029 &        343 &   7.03e+00 &   1.47e+01 &   6.17e+00 &         1.10e-02 &         5.83e-01 &         2.13e-03 &             0.33 \\
   4 &       2910 &       1029 &        343 &   4.04e+01 &   1.21e+01 &   5.74e+00 &         4.19e-01 &         5.45e-01 &         2.85e-03 &             2.22 \\
   5 &      17881 &       2187 &        729 &   2.75e+02 &   8.16e+01 &   4.08e+01 &         2.75e+01 &         4.00e+00 &         4.72e-02 &             2.39 \\
   6 &      99842 &       6210 &       2070 &   2.90e+03 &   2.33e+03 &   1.51e+03 &         1.26e+03 &         8.88e+01 &         3.42e-01 &             1.06 \\
   \midrule
    &       &         &    &
    \multicolumn{3}{c|}{ $t_{\rm as} (u_h)$ \;:\; $t_{\rm as} (\boldsymbol{y}_h)$ \;:\; $t_{\rm as} (w_h)$ } &      
    \multicolumn{3}{c|}{\; $t_{\rm sol} (u_h)$ \;:\; $t_{\rm sol} (\boldsymbol{y}_h)$  \;:\;  $t_{\rm sol} (w_h)$\;} & \\
 \midrule
 	 & 	 & 	 & 	 &       1.91 &       1.54 &       1.00 &          3683.81 &           259.30 &             1.00 &             1.06 \\
\end{tabular}
\caption{{\em Example 5}. 
Assembling and solving time (in seconds) spent for the systems generating d.o.f. of 
$u_h$, $\boldsymbol{y}_h$, and $w_h$ with bulk parameter $\sigma =0.4$ ($N_{\rm ref, 0}$ = 4).}
\label{tab:unit-domain-example-23-time-expenses-adapt-ref}
\end{table}

\begin{figure}[!t]
	\centering
	\subfloat[ref. 1]{
	\includegraphics[width=4cm, trim={2cm 2cm 3cm 3cm}, clip]{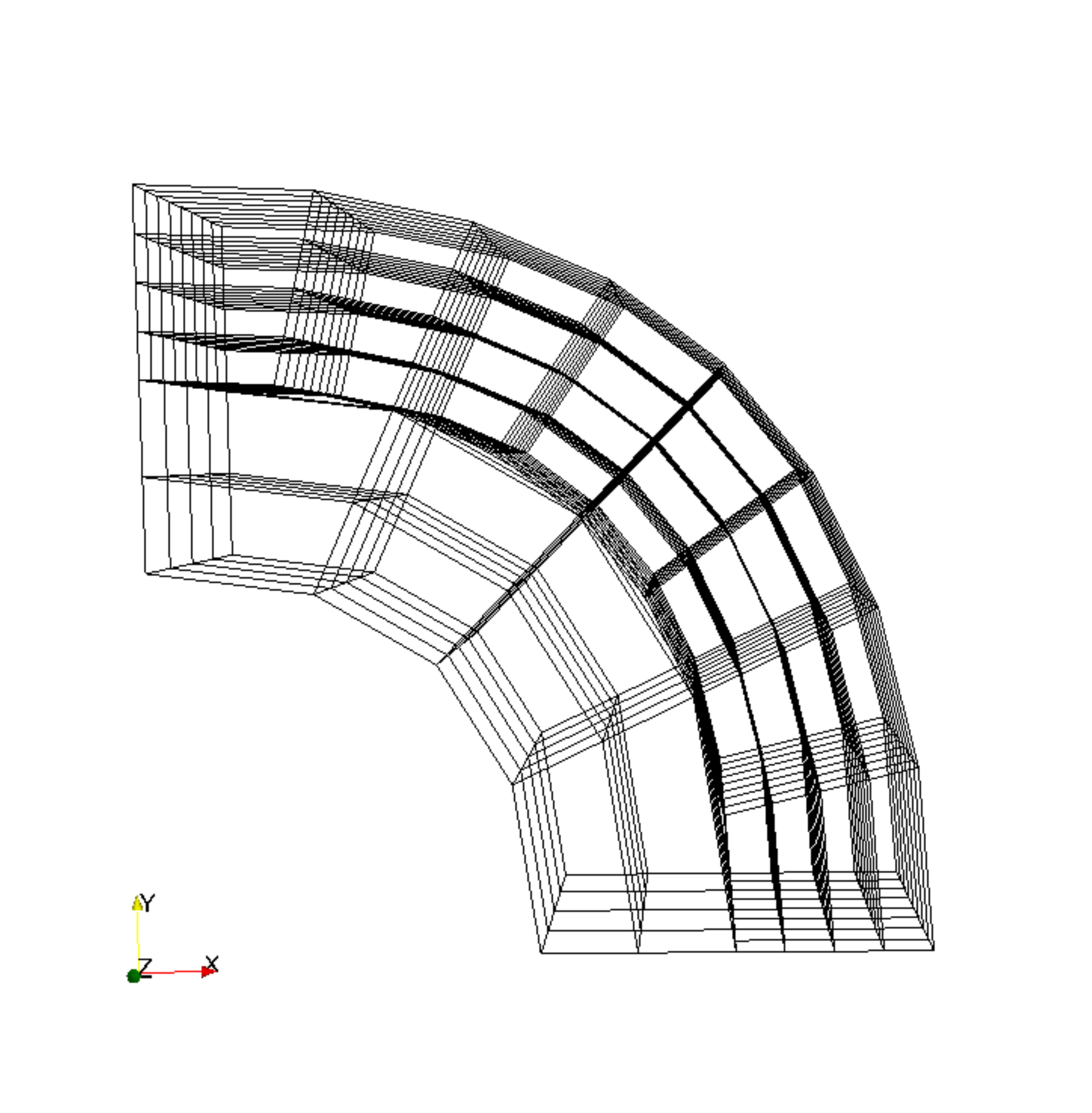}
	\label{fig:example-35-exact-solution}
	}
	\subfloat[ref. 2]{
	\includegraphics[width=4cm, trim={2cm 2cm 3cm 3cm}, clip]{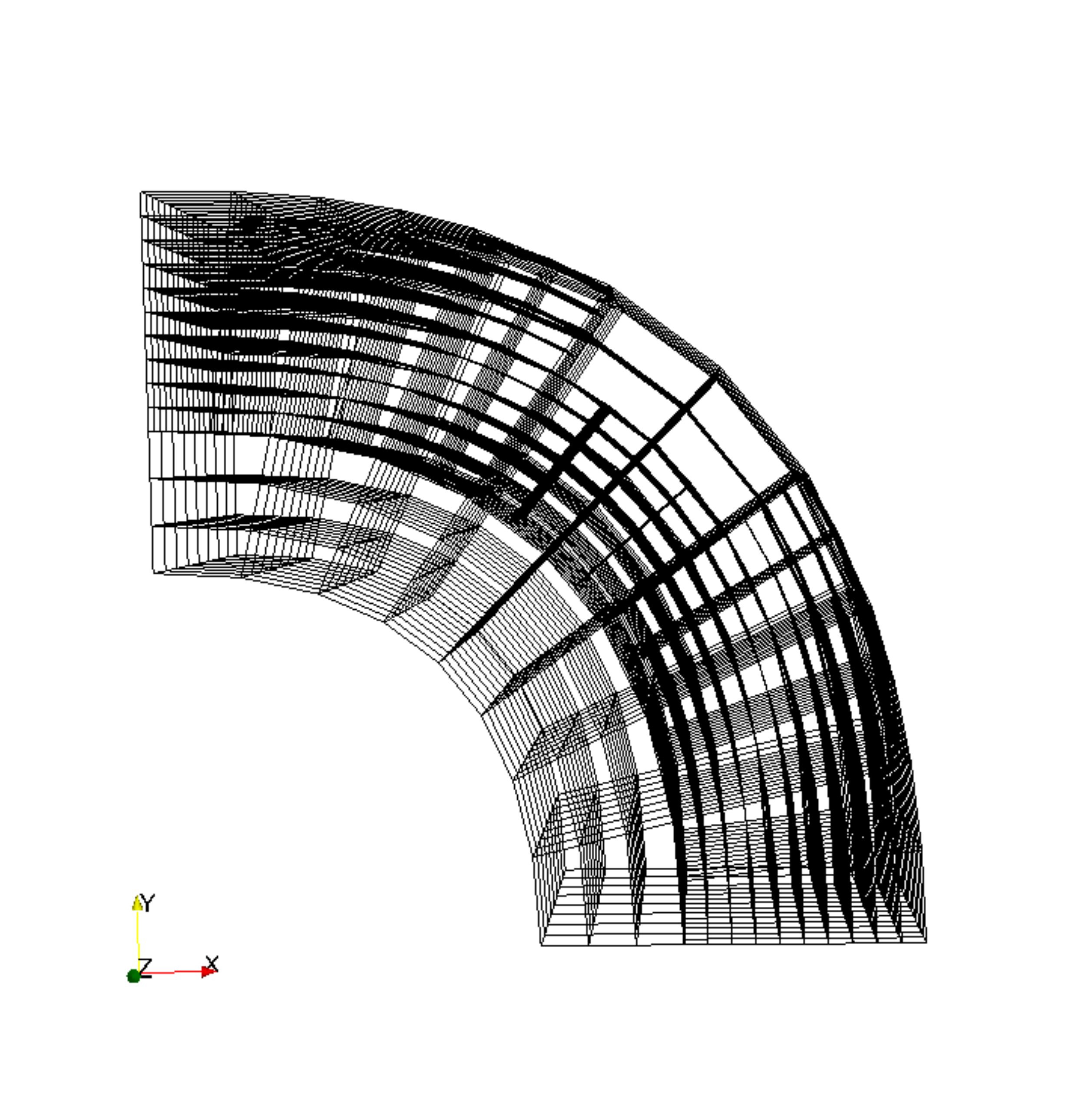}
	\label{fig:example-35-exact-solution}
	}
	\quad
	\subfloat[ref. 3]{
	\includegraphics[width=4cm, trim={2cm 2cm 3cm 3cm}, clip]{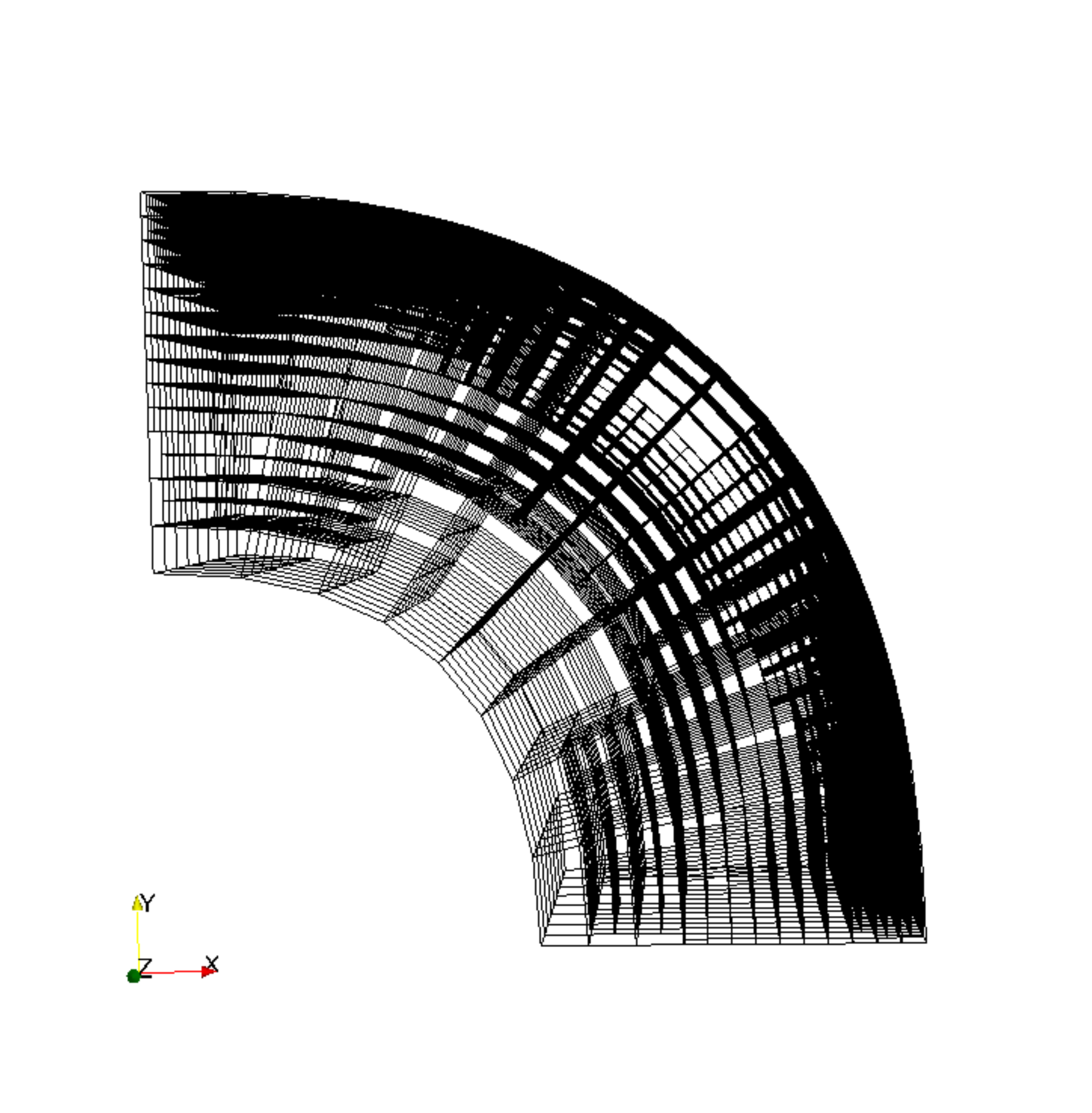}
	\label{fig:example-33-exact-solution}
	}		
	\caption{{\em Example 5}. 
	Mesh evolution for refinement steps 1--3 for marking parameter $\sigma = 0.6$.}
	\label{fig:time-singularity-1d-t-example-23}
\end{figure}

\section{Conclusions}
\label{sec:conclusion}
We derived a new locally stabilized space-time IgA schemes for parabolic I-BVPs, where
global scaling $h$ 
in the upwind test functions
is replaced by a local scaling that depends on the local element size 
$h_K$. Adaptive mesh refinement is based on error indicators generated by functional type a posteriori error estimates,
which naturally use specific features 
and advantages  of the IgA method.
Since error  majorants of the functional type  are presented by integrals formed
by element-wise contributions, they can  efficiently be used for indication
of the local errors and subsequent mesh refinement.
We consider a fully unstructured space-time adaptive IgA scheme and use localised THB-splines for the
mesh refinement. 
Finally, we illustrated the reliability 
and efficiency of the presented a posterior error estimates in a series  of examples exhibiting 
different features of exact solutions.
Numerical tests performed have demonstrated high efficiency of the approach.
Moreover, we also made a comparative study of the computational
expenses for assembling the systems,
finding an approximate solution, and 
computing a guaranteed and sufficiently accurate error bounds. In the majority of examples, 
error estimation required much lesser time than the reconstruction of the approximate IgA solution. 
The last but not least item to be mentioned is that the numerical examples 
have confirmed high efficiency of the locally stabilized space-time THB-spline-based methods
used in combination with suitable error indicators and mesh adaptive procedures.
Of course, beside THB-spline, other local 
spline refinement techniques such as mention in the introduction can also be utilized in this adaptive framework.
Adaptive methods should be connected with multigrid or 
multilevel solvers or preconditioners for the algebraic systems that we have to solve since the adaptive procedure 
naturally provides a space-time hierarchy of meshes. 
Preceding experiments on massively parallel computers 
presented in \cite{LMRLangerMooreNeumueller2016a}
show that even algebraic multigrid preconditioners
in connection with GMRES accelerations result 
in very efficient solvers for very huge systems 
with billions of space-time unknowns arising from (3+1)d examples.
It is clear that the approach presented can be extended to 
a wider class of parabolic problems and eddy current problems
in electromagnetics.

\end{document}